\documentclass[12pt,english]{amsart}

\usepackage{amsfonts,babel,amssymb,amsmath,amsthm}
\usepackage{graphicx}
\usepackage{float}
\usepackage{framed}
\usepackage{mathabx}
\usepackage{tikz}

\newtheorem{proposition}{Proposition}[section]

\newtheorem{lemma}[proposition]{Lemma}
\newtheorem{theorem}[proposition]{Theorem}

\numberwithin{equation}{section}

\newcommand{\bs}[1]{\boldsymbol{#1}}

\newcommand{\R}{{\mathbb R}}
\newcommand{\Rsym}{{\mathbb R^{3\times 3}_{\mathrm{sym}}}}

\newcommand{\mb}{\mathbf}
\newcommand{\mbb}{\mathbb}
\newcommand{\mr}{\mathrm}
\newcommand{\mc}{\mathcal}
\newcommand{\pp}{\partial}

\renewcommand{\div}{\mathrm{div}\hspace{1pt}}

\newcommand{\dualpr}[1]{\langle #1\rangle}

\title[HDG projection for elasticity and its applications]{New analytical tools for HDG in elasticity, with applications to elastodynamics}

\author{Shukai Du}
\address{University of Delaware}
\email{shukaidu@udel.edu}

\author{Francisco-Javier Sayas}
\address{University of Delaware}
\email{fjsayas@udel.edu}
\thanks{This work was partially supported by the NSF grant DMS-1818867.}

\subjclass[2010]{65N30, 65M60, 74B05}
\date{\today}
\keywords{discontinuous Galerkin method, hybridization, optimal convergence, elastic waves}

\begin{document}

\begin{abstract}
We present some new analytical tools for the error analysis of hybridizable discontinuous Galerkin (HDG) method for linear elasticity. These tools allow us to analyze more variants of HDG method using the projection-based approach, which renders the error analysis simple and concise.  The key result is a tailored projection for the Lehrenfeld-Sch\"{o}berl type HDG (HDG+ for simplicity) methods. 
By using the projection we recover the error estimates of HDG+ for steady-state and time-harmonic elasticity in a simpler analysis. We also present a semi-discrete (in space) HDG+ method for transient elastic waves and prove it is uniformly-in-time optimal convergent by using the projection-based error analysis. Numerical experiments supporting our analysis are presented at the end.
\end{abstract}

\maketitle

\section{Introduction}
The paper is devoted to present some new techniques for the a priori error analysis of a new class of hybridizable discontinuous Galerkin methods. The methods in this class use a special type of stabilization function that was first introduced by Lehrenfeld and Sch\"{o}berl in \cite{LeSc:2010}. We will call them HDG+ methods for simplicity. Instead of attempting to reach for maximal generality, we will focus on linear elasticity on tetrahedral meshes.

We begin by reviewing some existing works. The first HDG method for linear elasticity was proposed in \cite{SoCoSt:2009}. 
The method enforces the symmetry of the stress strongly and uses order $k$ polynomial spaces for all variables. 
It was then proved in \cite{FuCoSt:2015} that the method is optimal for displacement but only suboptimal for stress (order of $k+\frac{1}{2}$); they also showed the order is sharp on triangular meshes in the numerical experiments. 
To recover optimal convergence (based on which the superconvergence by post-processing is possible), there are mainly three approaches.
The first approach relaxes the strong symmetry condition to weak symmetry \cite{CoSh:2013}.
In general, mixed finite element methods based on weak symmetric stress formulations are relatively easier to implement but use more degrees of freedom and therefore can be more costly to compute. 
The second and the third approaches are all based on strong symmetric stress formulations, where the conservation of angular momentum is automatically preserved. The second approach \cite{CoFu:2018} uses $M$-decomposition \cite{CoFuSa:2017} to enrich the approximation space for stress by adding some basis functions. The approach recovers optimal convergence and also provides an associated tailor projection as an useful tool for error analysis. However, the added basis can be rational functions instead of polynomials and therefore can lead to some difficulties in implementation. 

The focus of this paper is the third approach, which is to use HDG+ method for linear elasticity. The method was originally proposed in \cite{LeSc:2010} for diffusion problems, then applied to steady-state linear elasticity in \cite{QiShSh:2018}. It uses only polynomial basis functions, achieves one order higher convergence rate for the displacement without post-processing, and its computational complexity is the same to the standard HDG method (order $k$ for both stress and displacement) for their global systems.
However, the existing error analysis of HDG+ methods are all based on using orthogonal projections \cite{HuPrSa:2017,Oi:2015,QiShSh:2018,QiSh:2016,QiSh2:2016}, which make the analysis slightly more complicated (it requires a bootstrapping argument to prove convergence of all variables, as opposed to consecutive energy and duality proofs), and detached from the existing projection-based error analysis of HDG methods \cite{ChCo:2012,CoFuHuJiSaMaSa:2018,CoGoSa:2010,CoMu:2015,CoQiSh2012,CoQu:2014,CoSh:2013,GrMo:2011}, where specifically constructed projections are used to make the analysis simple and concise. This motivates us to find a new kind of projection for HDG+ for elasticity. The goal of the projection is twofold: first of all, it takes care of all the off-diagonal terms in the matrix form of the equations (except  for a $2\times 2$ block which is considered as a single diagonal term), and allows us to do a simple energy estimate for some of the variables; second, it facilitates a duality argument where the adjoint equation is fed with the missing error terms to estimate, by using the adjoint projection (consisting of a simple change of sign in the stabilization parameter $\boldsymbol\tau$). This has been done in \cite{DuSa:2019} for diffusion problems. A novelty of this paper is the fact that we complement the projection with an error term that does not affect the error bounds or the simplicity of their proofs. We have attached this error term to the projection to make the arguments simpler. 

In summary, we have devised a projection for the HDG+ methods for linear elasticity. The projection enables us to: (1) recycle existing projection-based error analysis techniques for the analysis of HDG+ methods; (2) make the error analyis simple and concise;  (3) build connections between $M$-decompositions \cite{CoFu:2018, CoFuSa:2017} and HDG+ methods. To be more specific,
we present a semi-discrete HDG+ method for transient elastic waves that has a uniform-in-time optimal convergence. We show that the proof for optimal convergence can be easily obtained by using the new  projection and some existing techniques in traditional HDG methods for evolutionary equations \cite{CoFuHuJiSaMaSa:2018}.
Moreover, we recover the error estimates for steady-state elasticity \cite{QiShSh:2018} and frequency domain elastodynamics \cite{HuPrSa:2017} by using the projection-based analysis, and we show that the analysis can be simplified in both cases.
Since the construction of the HDG+ projection involves first constructing a projection associated to an $M$-decomposition, it also shed some light upon the connections between these two kinds of methods. 

To provide a more intuitive view, we put the main procedures of constructing the projection in the flow chart Figure \ref{fig:flow_chart}.
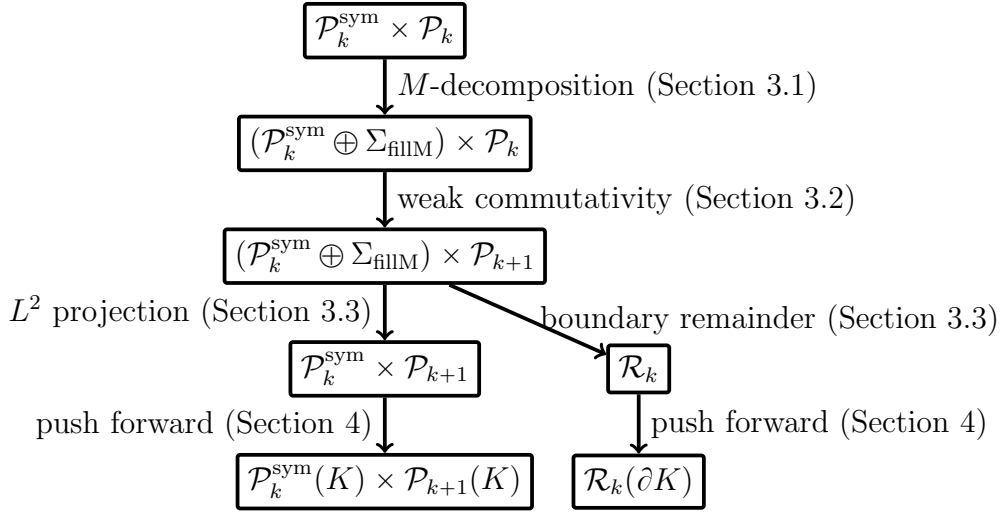
\begin{figure}[ht]
\centering
\begin{tikzpicture}[scale=1.5]
\node[draw,rounded corners=1pt,line width=.5mm]	(init) at (0,3) {$\mc P_k^\mathrm{sym}\times \mc P_k$};
\node[draw,rounded corners=1pt,line width=.5mm]	(Mdecomp) at (0,2) {$(\mc P_k^\mathrm{sym}\oplus\Sigma_\mathrm{fillM})\times \mc P_k$};
\node[draw,rounded corners=1pt,line width=.5mm]	(weakcom) at (0,1) {$(\mc P_k^\mathrm{sym}\oplus\Sigma_\mathrm{fillM})\times \mc P_{k+1}$};
\node[draw,rounded corners=1pt,line width=.5mm]	(HDGplus) at (0,0) {$\mc P_k^\mathrm{sym}\times \mc P_{k+1}$};
\node[draw,rounded corners=1pt,line width=.5mm]	(bdrem) at (2.25,0) {$\mc R_k$};
\node[draw,rounded corners=1pt,line width=.5mm]	(HDGplus_p) at (0,-1) {$\mc P_k^\mathrm{sym}(K)\times \mc P_{k+1}(K)$};
\node[draw,rounded corners=1pt,line width=.5mm]	(bdrem_p) at (2.25,-1) {$\mc R_k(\pp K)$};

\draw [->,line width=.5mm,black] (init) to node [right] {$M$-decomposition (Section \ref{sec:proj_ref}.1)} (Mdecomp);
\draw [->,line width=.5mm,black] (Mdecomp) to node [right] {weak commutativity (Section \ref{sec:proj_ref}.2)} (weakcom);
\draw [->,line width=.5mm,black] (weakcom) to node [left] {$L^2$ projection (Section \ref{sec:proj_ref}.3)} (HDGplus);
\draw [->,line width=.5mm,black] (weakcom) to node [right] {boundary remainder (Section \ref{sec:proj_ref}.3)} (bdrem);
\draw [->,line width=.5mm,black] (HDGplus) to node [left] {push forward (Section \ref{sec:proj_phy})} (HDGplus_p);
\draw [->,line width=.5mm,black] (bdrem) to node [right] {push forward (Section \ref{sec:proj_phy})} (bdrem_p);
\end{tikzpicture}
\caption{Main procedures to contruct the projection. The projection is first constructed on the reference element then pushed forward to the physical element $K$.}
\label{fig:flow_chart}
\end{figure}
The associated boundary remainder term related to the projection behaves like interpolation error and it depends only on the local projection. 
As we will see later in the applications, the boundary remainder together with the stabilization parameter play a key role in the optimal convergence of the HDG+ methods, allowing us more flexibility, since now we only need to find a projection such that its associated boundary remainder is small enough to guarantee optimal convergence, instead of enforcing it to vanish, which is the case of the standard projection. 

The rest of the paper is organized as follows. In Section \ref{sec:proj}, we present the main theorem about the projection and its properties. In Section \ref{sec:proj_ref}, we show the procedures for constructing the projection on the reference element. In Section \ref{sec:proj_phy}, we develop a systematic approach of changes of variables to obtain the projection on the physical element.
In Section \ref{sec:steady_state} and \ref{sec:freq}, we recover the error estimates in steady-state elasticity \cite{QiShSh:2018} and elasto-dynamics \cite{HuPrSa:2017} using projection-based analysis. In Section \ref{sec:trans_elas}, we present a semi-discrete HDG method for transient elastic waves and prove it is optimally convergent, uniformly in the time variable. Finally, we give some numerical experiments to support our analysis.

\section{The projection}\label{sec:proj}

Since the main tool and one of the principal novelties of this article is the new HDG projection for elasticity, we first introduce its main properties in Theorem \ref{th:PROJ}. The reader just interested in the applications can skip the sections devoted to its construction and analysis (Section \ref{sec:proj_ref} and \ref{sec:proj_phy}) and jump directly to how it is used (Section \ref{sec:steady_state}, \ref{sec:freq} and \ref{sec:trans_elas}). To speed up the introduction of the projection we give a quick notation list to be used throughout the article.

\begin{itemize}
\item For a domain $\mathcal O\subset \R^3$, the respective inner products of $L^2(\mathcal O;\R^3)$ and $L^2(\mathcal O;\Rsym)$ will be denoted
\[
(\bs u,\bs v)_{\mathcal O}:=\int_{\mathcal O} \bs u\cdot\bs v,
\qquad
(\bs\sigma,\bs\rho)_{\mathcal O}:=\int_{\mathcal O}
\bs\sigma :\bs\rho,
\]
where in the latter the colon denotes the Frobenius product of matrices and $\Rsym$ is the space of symmetric $3\times 3$ matrices. The norm of both spaces will be denoted $\|\cdot\|_{\mathcal O}$.
\item For a Lipschitz domain $\mathcal O$, the inner product in $L^2(\partial \mathcal O;\R^3)$ will be denoted
\[
\langle \bs\mu,\bs\eta\rangle_{\partial\mathcal O}:=\int_{\partial\mathcal O}
\bs\mu\cdot\bs\eta
\]
and the associated norm will be denoted $\|\cdot\|_{\partial\mathcal O}$.
\item The Sobolev seminorm in $H^m(\mathcal O;\R^3)$ and $H^m(\mathcal O;\Rsym)$ will be denoted $|\cdot|_{m,\mathcal O}$.
\item The symmetric gradient operator (linearized strain) is given by
\[
\bs\varepsilon(\bs u):=\tfrac12 (\mathrm D\bs u+(\mathrm D\bs u)^\top),
\]
and the divergence operator $\mathrm{div}$ will be applied to symmetric-matrix-valued functions by acting on their rows, outputting a column vector-valued function.
\item When $\mathcal O$ is a Lipschitz domain, we will consider the space
\[
H(\mathcal O,\mathrm{div};\Rsym):=
	\{ \bs\sigma\in L^2(\mathcal O;\Rsym)\,:\,
	\mathrm{div}\,\bs\sigma\in L^2(\mathcal O;\R^3)\},
\]
and the normal traction operator
$
\gamma_n:H(\mathcal O,\mathrm{div};\Rsym)\to H^{-1/2}(\partial\mathcal O;\R^3),
$ 
defined by Betti's formula
\[
\langle \gamma_n \bs\sigma,\gamma\bs v\rangle_{\partial\mathcal O}
:=(\bs\sigma,\bs\varepsilon(\bs v))_{\mathcal O}
+(\mathrm{div}\,\bs\sigma,\bs v)_{\mathcal O}
\qquad \forall \bs v\in H^1(\mathcal O;\R^3).
\]
Here $\gamma$ is the trace operator, $H^{-1/2}(\partial\mathcal O;\R^3)$ is the dual space of $H^{1/2}(\partial\mathcal O;\R^3)$, and the angled bracket denotes their duality product that extends the $L^2(\partial\mathcal O;\R^3)$ inner product. 
\end{itemize}
For discretization we will consider a sequence of tetrahedral meshes $\mathcal T_h$ and the following notation: 
\begin{itemize}
\item $K$ is a tetrahedron, of diameter $h_K$ and inradius at most $c\,h_K$ for a fixed shape-regularity constant $c>0$.
\item $\mathcal F(K)$ is the set of faces of $K$.
\item $\mathcal P_k(K;X)$ is the space of $X$-valued polynomial functions of degree up to $k\ge 0$, where $X\in \{\R^3,\Rsym\}$.
\item $\mathcal R_k(\partial K;X):=\prod_{F\in \mathcal F(K)} \mathcal P_k(F;X)$ is the space of piecewise polynomial functions on the boundary of $K$, with $X$ as above.
\item $\mr P_k:L^2(K;X)\rightarrow \mc P_k(K;X)$ is the orthogonal projection onto the image space, with $X$ as above.
\item $\mathrm P_M:L^2(\partial K;\R^3)\to \mathcal R_k(\partial K;\R^3)$ is the orthogonal projection onto the image space. This operator will often be applied on the trace of a function in $H^1(K;\R^3)$ on $\partial K$.
\item $\bs\tau\in \mathcal R_0(\partial K;\Rsym)$ is a so-called stabilization function satisfying
\begin{equation}\label{eq:tauhm1}
C_1 h_K^{-1} \|\bs\mu\|_{\partial K}^2
\le \langle\bs\tau\bs\mu,\bs\mu\rangle_{\partial K}
\le C_2 h_K^{-1} \|\bs\mu\|_{\partial K}^2
\qquad\forall \bs\mu\in L^2(\partial K;\R^3),
\end{equation}
for fixed positive constants $C_1$ and $C_2$, independent of $h$ (i.e., of the particular mesh).
\item The wiggled inequality sign $a\lesssim b$ hides a constant $a\le C\, b$ that is independent of $h$, while $a \approx b$ means $a\lesssim b\lesssim a$.   
\end{itemize}

\begin{theorem}\label{th:PROJ}
For $k\ge 1$, there exists
a family of projections and associated boundary remainder operators
\begin{alignat*}{6}
\Pi:H^1(K;\Rsym)\times H^1(K;\R^3)
& \to \mathcal P_k(K;\Rsym) \times \mathcal P_{k+1}(K;\R^3),\\
\mathrm R: H^1(K;\Rsym)\times H^1(K;\R^3) & \to \mathcal R_k(\partial K;\R^3),
\end{alignat*}
depending on $\bs\tau\in \mathcal R_0(\partial K;\Rsym)$,  where $\bs\tau$ is a constant positive definite matrix on each face of $K$,  and if $(\bs\sigma_K,\bs u_K)=\Pi(\bs\sigma,\bs u;\bs\tau)$ and $\bs\delta_K=\mathrm R(\bs\sigma,\bs u;\bs\tau)$, the following conditions hold:
\begin{subequations}\label{eq:tauProp}
\begin{alignat}{6}\label{eq:tauProp_a}
 (\bs u_K-\bs u,\bs v)_K &=0 &\quad &
	\forall \bs v\in \mathcal P_{k-1}(K;\R^3),\\
\label{eq:tauProp_b}
 -(\mathrm{div}\,(\bs\sigma_K-\bs\sigma),\bs w)_K
+\langle \bs\tau \mathrm P_M(\bs u_K-\bs u),\bs w\rangle_{\partial K} &=
\langle \bs\delta_K,\bs w\rangle_{\partial K}
	&& \forall \bs w\in \mathcal P_{k+1}(K;\R^3),\\
\label{eq:tauProp_c}
-\langle (\bs\sigma_K-\bs\sigma)\bs n
	-\bs\tau(\bs u_K-\bs u),\bs\mu\rangle_{\partial K} &=
	\langle \bs\delta_K,\bs\mu\rangle_{\partial K}
	&& \forall \bs\mu\in \mathcal R_k(\partial K;\R^3).
\end{alignat}
\end{subequations}
Moreover, if $\bs\tau$ satisfies \eqref{eq:tauhm1} and $(\bs\sigma,\bs u)\in H^m(K;\Rsym)\times H^{m+1}(K;\R^3)$ with $1\le m\le k+1$, then we have the estimates:
\begin{equation}
\label{eq:1.3a}
\| \bs\sigma_K-\bs\sigma\|_K+ h_K^{-1}\|\bs u_K-\bs u\|_K
+h_K^{1/2}\|\bs\delta_K\|_{\partial K}
	\le C h_K^m (|\bs\sigma|_{m,K}+|\bs u|_{m+1,K}).
\end{equation} 
The constant $C$ depends only on the polynomial degree $k$, the constants $C_1$ and $C_2$ in \eqref{eq:tauhm1} and the shape-regularity constant $c$.
Finally, the `adjoint' projection can be defined as 
\begin{equation}\label{eq:1.4}
\Pi(\bs\sigma,\bs u;-\bs\tau):=(\bs\sigma_K,\bs u_K),
\qquad\mbox{where}\qquad
(\bs\sigma_K,-\bs u_K):=\Pi(\bs\sigma,-\bs u;\bs\tau).
\end{equation}
This projection satisfies the properties \eqref{eq:tauProp} with the same $\bs\delta_K$, if we substitute $\bs\tau$ by $-\bs\tau$.
\end{theorem}

Note that conditions \eqref{eq:tauProp} are not enough to define the projection $\Pi$ but are exactly the ones that will be needed for the applications. 
As a final note, notice that combining \eqref{eq:tauProp_b} and \eqref{eq:tauProp_c}, we have
\begin{align*}
(\bs\sigma_K-\bs\sigma,\bs\varepsilon(\bs w))_K=0\qquad
\forall \bs w\in\mc P_k(K;\mbb R^3).
\end{align*}
Since $k\ge1$ and $\mc P_0(K;\mbb R_\mr{sym}^{3\times3})\subset \bs\varepsilon(\mc P_1(K;\mbb R^3))$, we have
\begin{equation}\label{eq:rmproj}
(\bs\sigma_K-\bs\sigma,\bs\theta)_K=0\qquad
\forall \bs\theta\in\mc P_0(K;\mbb R_\mr{sym}^{3\times3}).
\end{equation}

\section{The projection in the reference element}
\label{sec:proj_ref}

\subsection{Preparatory work on the reference element}\label{sec:Prep}

In this section we will work on the
reference tetrahedron $\widehat K:=\{ (x_1,x_2,x_3)\in \R^3\,:\, x_1,x_2,x_3,1-x_1-x_2-x_3>0\}$. 
The trace for vector-valued functions on the reference element will be denoted $\widehat\gamma$, and the normal traction operator on the reference element will be denoted $\widehat\gamma_n$. We will use shortened notation for the following spaces:
\begin{subequations}\label{eq:2.1}
\begin{alignat}{6}
& \widehat{\bs V}:=\mathcal P_{k}(\widehat K;\R^3), 
&\qquad &
\widehat{\bs V}_-:=\mathcal P_{k-1}(\widehat K;\R^3),
\qquad
\widehat{\bs V}_+:=\mathcal P_{k+1}(\widehat K;\R^3),\\
& \widehat\Sigma :=\mathcal P_k(\widehat K;\Rsym),
&\qquad &
\widehat\Sigma_-:=\{ \bs\sigma\in\widehat\Sigma\,:\,
\mathrm{div}\,\bs\sigma=0, \quad \widehat\gamma_n\bs\sigma=0\}
\oplus \bs\varepsilon(\widehat{\bs V}).
\end{alignat}
\end{subequations}
A new space $\widehat\Sigma_+$ will be defined once we have introduced some tools for it. Note that these constructions can be done directly on any tetrahedron $K$ \cite{CoFu:2018}.

The first of these constructions is a lifting of the traction operator. It will act on the space
\[
L^2_{\mathcal M}(\partial\widehat K;\R^3)
:=\{\bs\mu\in L^2(\partial\widehat K;\R^3)\,:\, 
\langle\bs\mu,\widehat\gamma\bs m\rangle_{\partial\widehat K}=0
\quad\forall\bs m\in\mathcal M\},
\]
where $\mathcal M$ is the six-dimensional space of infinitesimal rigid motions
\begin{alignat*}{6}
\mathcal M:=& \{\bs m(\mathbf x):=\mathbf b+A\mathbf x\,:\,
A^\top =-A,\quad \mathbf b\in \R^3\}\\
=&\{\bs m\in H^1(\widehat K;\R^3)\,:\, \bs\varepsilon(\bs m)=0\}.
\end{alignat*}
Note that the trace space of $\mathcal M$ is also six-dimensional, i.e., if $\bs m\in \mathcal M$ vanishes on the boundary of $\widehat K$, then $\bs m=\bs 0$. We thus define the operator $\widehat\gamma_n^+:L^2_{\mathcal M}(\partial\widehat K;\R^3)\to H(\widehat K,\mathrm{div};\Rsym)$ by 
\begin{subequations}\label{eq:gmp}
\begin{alignat}{6}
\widehat\gamma_n^+\bs\mu:=\bs\varepsilon(\bs u),
\quad\mbox{where}\quad
& \bs u\in H^1(\widehat K;\R^3)/\mathcal M,\\
& (\bs\varepsilon(\bs u),\bs\varepsilon(\bs v))_{\widehat K}
=\langle\bs\mu,\widehat\gamma\bs v\rangle_{\partial\widehat K}
\quad\forall \bs v\in H^1(\widehat K;\R^3),
\end{alignat}
\end{subequations}
or equivalently
\begin{subequations}\label{eq:gmp2}
\begin{alignat}{6}
\widehat\gamma_n^+\bs\mu:=\bs\sigma,
\quad\mbox{where}\quad
&\bs\sigma\in \bs\varepsilon(H^1(\widehat K;\R^3)),\\
& (\bs\sigma,\bs\varepsilon(\bs v))_{\widehat K}
=\langle\bs\mu,\widehat\gamma\bs v\rangle_{\partial\widehat K}
\quad\forall \bs v\in H^1(\widehat K;\R^3).
\end{alignat}
\end{subequations}
The definition \eqref{eq:gmp} is correct since it involves the solution of a coercive variational problem on a quotient space, due to Korn's Second Inequality. From \eqref{eq:gmp2} it is clear that
\begin{equation}\label{eq:6}
\mathrm{div}\,\widehat\gamma_n^+\bs\mu=0,
\qquad
\widehat\gamma_n\widehat\gamma_n^+\bs\mu=\bs\mu
\qquad \forall \bs\mu\in L^2_{\mathcal M}(\partial K;\R^3). 
\end{equation}
We next consider the spaces
\begin{alignat*}{6}
&\widehat{\bs M}:=\mathcal R_k(\partial\widehat K;\R^3),
\qquad
\widehat\Sigma_S := \{\bs\sigma\in\widehat\Sigma\,:\, \mathrm{div}\,\bs\sigma=0\},\\
& \widehat\Theta:=\{\bs\mu\in \widehat{\bs M}\,:\,
\langle\widehat\gamma_n\bs\sigma
+\widehat\gamma\bs m,\bs\mu\rangle_{\partial\widehat K}=0
\quad\forall (\bs\sigma,\bs m)\in \widehat\Sigma_S\times \mathcal M \},
\end{alignat*}
and
\[
\widehat\Sigma_{\mathrm{fill}}:=\widehat\gamma_n^+\widehat\Theta.
\]

\begin{theorem}\label{th:2.1}
The following properties hold:
\begin{itemize}
\item[\rm (a)] $\mathrm{div}\,\bs\sigma=0$ for all $\bs\sigma\in \widehat\Sigma_{\mathrm{fill}}$,
\item[\rm (b)] $\widehat\gamma_n^+$ is an isomorphism between $\widehat\Theta$ and $\widehat\Sigma_{\mathrm{fill}}$, and its inverse is $\widehat\gamma_n$,
\item[\rm (c)] $\widehat\Sigma_{\mathrm{fill}}\cap \widehat\Sigma=\{0\}$,
\item[\rm (d)] $\widehat{\bs M}=\widehat\Theta\oplus\widehat\gamma_n\widehat\Sigma_S\oplus \gamma\mathcal M$ with orthogonal sum.
\end{itemize}
\end{theorem}

\begin{proof}
Properties (a) and (b) are easy consequences of \eqref{eq:6}. By (a), $\widehat\Sigma_{\mathrm{fill}}\cap \widehat\Sigma=\widehat\Sigma_{\mathrm{fill}}\cap \widehat\Sigma_S$ and, therefore, if $\bs\sigma\in \widehat\Sigma_{\mathrm{fill}}\cap \widehat\Sigma_S$, then $\widehat\gamma_n\bs\sigma\in \widehat\Theta\cap \widehat\gamma_n\widehat\Sigma_S=\{0\}$ (the latter two sets are orthogonal to each other by definition of $\widehat\Theta$), but then (b) proves that $\bs\sigma=0$, which proves (c). Finally, if $\bs\sigma\in \widehat\Sigma_S$ and $\bs m\in \mathcal M$, then
\[
\langle\widehat\gamma_n\bs\sigma,\gamma\bs m\rangle_{\partial\widehat K}
=(\mathrm{div}\,\bs\sigma,\bs m)_{\widehat K}
+(\bs\sigma,\bs\varepsilon(\bs m))_{\widehat K}=0,
\]
which shows that the sum $\widehat\gamma_n\widehat\Sigma_S+\widehat\gamma\mathcal M$ is orthogonal. Since $\widehat\Theta$ is the orthogonal complement of the latter set, (d) is proved.
\end{proof}

The Cockburn-Fu discrete pair for elasticity 	\cite{CoFu:2018} is given by the spaces
\[
\widehat\Sigma_+:=\widehat\Sigma\oplus\widehat\Sigma_{\mathrm{fill}}
\]
(the sum is direct because of Theorem \ref{th:2.1}) and $\widehat{\bs V}$. In their context of $M$-decompositions for elasticity, the following result is a key one, that we will need to work with our extended pair $\widehat\Sigma_+\times \widehat{\bs V}_+$. 

\begin{theorem}\label{th:2.2}
The following properties hold:
\begin{itemize}
\item[\rm (a)] $\bs\varepsilon(\widehat{\bs V})\subset \widehat\Sigma_-$,
\item[\rm (b)] $\mathrm{div}\,\widehat\Sigma_+\subset \widehat{\bs V}_-$,
\item[\rm (c)] $\widehat\gamma_n\widehat\Sigma_+ 
+ \widehat\gamma\widehat{\bs V}\subset \widehat{\bs M}$,
\item[\rm (d)] $\mathrm{dim}\,\widehat{\bs M}
=\mathrm{dim}\,\widehat\Sigma_-^\perp+\mathrm{dim}\,\widehat{\bs V}_-^\perp,$ where the orthogonal complement $\widehat\Sigma_-^\perp$ is taken in $\widehat\Sigma_+$, while $\widehat{\bs V}_-^\perp$ is taken in $\widehat{\bs V}$,
\item[\rm (e)] $\widehat{\bs M}=\widehat\gamma_n\widehat\Sigma_-^\perp
\oplus \gamma \widehat{\bs V}_-^\perp$, with orthogonal sum,
\item[\rm (f)] The map
$
\widehat\Sigma_-^\perp\times\widehat{\bs V}_-^\perp\ni (\bs\sigma,\bs u)
\longmapsto 
\widehat\gamma_n\bs\sigma+\widehat\gamma\bs u\in \widehat{\bs M}
$
is an isomorphism.
\end{itemize}
\end{theorem}

\begin{proof}
Property (a) follows by definition and (b) is a simple consequence of Theorem \ref{th:2.1}(a). To show (c), note simply that $\widehat\gamma_n\widehat\Sigma_{\mathrm{fill}}=\widehat\Theta\subset \widehat{\bs M}$, while all other elements are polynomials of degree less than or equal to $k$. 

To prove (d) to (f), it will be convenient to identify the set
\begin{equation}\label{eq:2.5}
\widehat\Sigma_S^\circ:=
\{ \bs\sigma\in \widehat\Sigma\,:\,
\mathrm{div}\,\bs\sigma=0,\quad\widehat\gamma_n\bs\sigma=0\}
=\ker \widehat\gamma_n|_{\widehat\Sigma_S},
\end{equation}
which appeared in the definition of $\widehat\Sigma_-$ \eqref{eq:2.1}.

By Theorem \ref{th:2.1}(d) and \eqref{eq:2.5}, we have
\begin{alignat}{6}
\mathrm{dim}\,\widehat{\bs M}
	&=\mathrm{dim}\,\widehat\Theta
		+\mathrm{dim}\,\widehat\gamma_n\widehat\Sigma_S
		+\mathrm{dim}\,\gamma\mathcal M
	=\mathrm{dim}\,\widehat\Sigma_{\mathrm{fill}}
		+\mathrm{dim}\,\widehat\Sigma_S-\mathrm{dim}\,\widehat\Sigma_S^\circ
		+\mathrm{dim}\,\mathcal M.
\label{eq:2.6}
\end{alignat}
Using the definitions of $\widehat\Sigma_\pm$ we have
\begin{alignat}{6}
\nonumber
\mathrm{dim}\,\widehat\Sigma_-^\perp
	&=\mathrm{dim}\,\widehat\Sigma_+-\mathrm{dim}\,\widehat\Sigma_-
	 =\mathrm{dim}\,\widehat\Sigma+\mathrm{dim}\,\widehat\Sigma_{\mathrm{fill}}
		-(\mathrm{dim}\,\widehat\Sigma_S^\circ
			+\mathrm{dim}\,\bs\varepsilon(\widehat{\bs V}))\\
	&=\mathrm{dim}\,\widehat{\bs V}_-+\mathrm{dim}\,\widehat\Sigma_S
		+\mathrm{dim}\,\widehat\Sigma_{\mathrm{fill}}
		-(\mathrm{dim}\,\widehat\Sigma_S^\circ+\mathrm{dim}\,\widehat{\bs V}
			-\mathrm{dim}\,\mathcal M),
\label{eq:2.7}
\end{alignat}
where in the last equality we have applied that $\mathrm{div}:\widehat\Sigma\to\widehat{\bs V}_-$ is onto (this is easy to prove), its kernel is $\widehat\Sigma_S$, and the kernel of $\bs\varepsilon:\widehat{\bs V}\to \bs\varepsilon(\widehat{\bs V})$ is $\mathcal M$. The equalities \eqref{eq:2.6} and \eqref{eq:2.7} prove (d).

Due to parts (a) and (b) of this theorem, we have
\[
\langle\widehat\gamma_n\bs\sigma,\widehat\gamma\bs u\rangle_{\partial\widehat K}
=0 \qquad 
\forall \bs\sigma\in \widehat\Sigma_-^\perp, 
\quad\bs u \in \widehat{\bs V}_-^\perp,
\]
which proves that the sum of $\widehat\gamma_n \widehat\Sigma_-^\perp$ and $\widehat\gamma \widehat{\bs V}_-^\perp$ is orthogonal. Since $\widehat\gamma:\widehat{\bs V}_-^\perp\to \widehat{\bs M}$ is injective (this result is known; see, for instance, \cite{sayas2013raviart} and \cite{DuSa:2019}), the property (e) will follow from having proved that $\widehat\gamma_n:\widehat\Sigma_-^\perp\to \widehat{\bs M}$ is injective. 

We first prove the following technical result: if $\bs\sigma\in \widehat\Sigma_+$ satisfies $\mathrm{div}\,\bs\sigma=0$ and $\widehat\gamma_n\bs\sigma=0$, then $\bs\sigma\in \widehat\Sigma$, i.e., the component in $\widehat\Sigma_{\mathrm{fill}}$ vanishes. To do that, take $\bs\sigma=\bs\sigma_1+\bs\sigma_2\in \widehat\Sigma+\widehat\Sigma_{\mathrm{fill}}$ and note that Theorem \ref{th:2.1}(a) shows that if $\mathrm{div}\,\bs\sigma=0$, then $\mathrm{div}\,\bs\sigma_1=0$. Since
\[
0=\widehat\gamma_n\bs\sigma=
\widehat\gamma_n\bs\sigma_1
+\widehat\gamma_n\bs\sigma_2\in \widehat\gamma_n\widehat\Sigma_S
\oplus\widehat\Theta,
\]
by Theorem \ref{th:2.1}(d), it follows that $\widehat\gamma_n\bs\sigma_1=
\widehat\gamma_n\bs\sigma_2=0$, which proves that $\bs\sigma_2=0$ by Theorem \ref{th:2.1}(b). 

The proof of injectivity of $\widehat\gamma_n|_{\widehat\Sigma_-^\perp}$ is then simple. If $\bs\sigma\in \widehat\Sigma_-^\perp$ satisfies $\widehat\gamma_n\bs\sigma=0$, then by part (a) 
\[
(\mathrm{div}\,\bs\sigma,\bs v)_{\widehat K}
=-(\bs\sigma,\bs\varepsilon(\bs v))_{\widehat K}=0
\qquad\forall \bs v\in \widehat{\bs V}
\]
and, taking $\bs v=\mathrm{div}\,\bs\sigma$ (by part (b)), we prove that $\mathrm{div}\,\bs\sigma=0$. Therefore, $\bs\sigma\in \widehat\Sigma_S^\circ\subset\widehat\Sigma_-$ and hence $\bs\sigma=0$. This completes the proof of (e) and (f).
\end{proof}

For the rest of this section we fix the stabilization parameter $\bs\tau\in \mathcal R_0(\partial\widehat K;\Rsym)$ such that
\begin{equation}\label{eq:3.0}
\bs\tau|_F \mbox{ is positive definite } \quad\forall F\in \mathcal F(\widehat K).
\end{equation}

\subsection{A projection on an extended space}
  
Given $(\bs\sigma,\bs u)\in H^1(\widehat K;\Rsym)\times H^1(\widehat K,\R^3)$, we look for
\begin{subequations}\label{eq:3.1}
\begin{equation}
\widehat\Pi_0(\bs\sigma,\bs u;\bs\tau):=
(\bs\sigma_{\widehat K},\bs u_{\widehat K})
\in \widehat\Sigma_+\times \widehat{\bs V}_+
\end{equation}
satisfying
\begin{alignat}{6}
\label{eq:3.1b}
(\bs u_{\widehat K}-\bs u,\bs v)_{\widehat K} = 0 
& \qquad && \forall \bs v\in \widehat{\bs V}_-,\\
\label{eq:3.1c}
(\bs\sigma_{\widehat K}-\bs\sigma,\bs\theta)_{\widehat K}=0 
&&&\forall \bs\theta\in \widehat\Sigma_-,\\
\label{eq:3.1d}
\langle \widehat\gamma_n(\bs\sigma_{\widehat K}-\bs\sigma)
-\bs\tau \mathrm P_{\widehat M}\widehat\gamma (\bs u_{\widehat K}-\bs u),
\bs\mu\rangle_{\partial\widehat K} =0 
&&&\forall \bs\mu\in \widehat{\bs M},\\
\label{eq:3.1e}
 -(\mathrm{div}\,(\bs\sigma_{\widehat K}-\bs\sigma),\bs w)_{\widehat K}
+\langle\bs\tau \mathrm P_{\widehat M}\widehat\gamma (\bs u_{\widehat K}-\bs u),
\widehat\gamma \bs w\rangle_{\partial\widehat K} =0
&&&\forall \bs w\in \widehat{\bs V}^\perp,
\end{alignat}
\end{subequations}
where $\widehat{\bs V}^\perp$ is the orthogonal complement of $\widehat{\bs V}$ in $\widehat{\bs V}_+$. Note that \eqref{eq:3.1e} is equivalent to
\[
(\bs\sigma_{\widehat K}-\bs\sigma,\bs\varepsilon(\bs w))_{\widehat K}
-\langle \widehat\gamma_n(\bs\sigma_{\widehat K}-\bs\sigma)
-\bs\tau \mathrm P_{\widehat M}\widehat\gamma (\bs u_{\widehat K}-\bs u),
\widehat\gamma\bs w\rangle_{\partial\widehat K} =0 
\qquad \forall \bs w\in \widehat{\bs V}^\perp,
\]
and, since $\bs\varepsilon(\widehat{\bs V})\subset \widehat\Sigma_-$ and $\widehat\gamma\widehat{\bs V}\subset\widehat{\bs M}$, equations \eqref{eq:3.1c} and \eqref{eq:3.1d} imply that we can substitute \eqref{eq:3.1e} by the condition
\begin{equation}\label{eq:3.2}
 -(\mathrm{div}\,(\bs\sigma_{\widehat K}-\bs\sigma),\bs w)_{\widehat K}
+\langle\bs\tau \mathrm P_{\widehat M}\widehat\gamma (\bs u_{\widehat K}-\bs u),
\widehat\gamma \bs w\rangle_{\partial\widehat K} =0\qquad 
\forall \bs w\in \widehat{\bs V}_+,
\end{equation}
which contains redundant restrictions already imposed in the other equations. Note also that the projection $\mathrm P_{\widehat M}$ can be eliminated in \eqref{eq:3.1d} but not in \eqref{eq:3.1e} or in \eqref{eq:3.2}, and that the bilinear form
\[
\langle\bs\tau\mathrm P_{\widehat M}\bs\mu,\bs\eta\rangle_{\partial\widehat K}
\]
is symmetric, bounded, and positive semi-definite in $L^2(\partial\widehat K;\R^3)$. We next prove that $\widehat\Pi_0$ is actually a well-defined projection onto $\widehat\Sigma_+\times\widehat{\bs V}_+$. 

\begin{proposition}\label{prop:3.1}
The process of defining the projection $\widehat\Pi_0$ in \eqref{eq:3.1} is equivalent to a square invertible linear system.
\end{proposition}

\begin{proof}
Note first that by Theorem \ref{th:2.2}(d)
\begin{alignat*}{6}
\mathrm{dim}\,\widehat{\bs V}_-+\mathrm{dim}\,\widehat\Sigma_-
+\mathrm{dim}\,\widehat{\bs M}+\mathrm{dim}\,\widehat{\bs V}^\perp
=&\mathrm{dim}\,\widehat{\bs V}_-+\mathrm{dim}\,\widehat\Sigma_-
+\mathrm{dim}\,\widehat{\bs V}_-^\perp+\mathrm{dim}\,\widehat\Sigma_-^\perp \\
& +\mathrm{dim}\,\widehat{\bs V}_+-\mathrm{dim}\,\widehat{\bs V}\\
=&\mathrm{dim}\,\widehat\Sigma_++\mathrm{dim}\,\widehat{\bs V}_+,
\end{alignat*}
which proves that \eqref{eq:3.1} is equivalent to a linear system with as many equations as unknowns. We thus only need to prove uniqueness of solution.

A homogeneous solution of \eqref{eq:3.1} is a pair $(\bs\sigma_{\widehat K},\bs u_{\widehat K})\in \widehat\Sigma_+\times \widehat{\bs V}_+$ satisfying (recall how \eqref{eq:3.2} is a consequence of \eqref{eq:3.1})
\begin{subequations}
\begin{alignat}{6}
\label{eq:3.3a}
(\bs u_{\widehat K},\bs v)_{\widehat K} = 0 &
\qquad && \forall \bs v\in \widehat{\bs V}_-,\\
\label{eq:3.3b}
(\bs\sigma_{\widehat K}\,,\bs\theta)_{\widehat K}=0 
&&&\forall \bs\theta\in \widehat\Sigma_-,\\
\label{eq:3.3c}
\langle \widehat\gamma_n\bs\sigma_{\widehat K}
-\bs\tau \mathrm P_{\widehat M}\widehat\gamma \bs u_{\widehat K},
\bs\mu\rangle_{\partial\widehat K} =0 
&&&\forall \bs\mu\in \widehat{\bs M},\\
\label{eq:3.3d}
 -(\mathrm{div}\,\bs\sigma_{\widehat K},\bs w)_{\widehat K}
+\langle\bs\tau \mathrm P_{\widehat M}\widehat\gamma \bs u_{\widehat K},
\widehat\gamma \bs w\rangle_{\partial\widehat K} =0&
&&\forall \bs w\in \widehat{\bs V}_+.
\end{alignat}
\end{subequations}
We now take $\bs w=\bs u_{\widehat K}\in \widehat{\bs V}_-^\perp$ (with the complement in $\widehat{\bs V}_+$; also see \eqref{eq:3.3a}) in \eqref{eq:3.3d}, recalling that $\mathrm{div}\,\bs\sigma_{\widehat K}\in \widehat{\bs V}_-$ (cf. Theorem \ref{th:2.2}(b)) and obtain
\[
\langle\bs\tau \mathrm P_{\widehat M}\widehat\gamma\bs u_{\widehat K},
	\widehat\gamma\bs u_{\widehat K}\rangle_{\partial\widehat K}
	=\langle\bs\tau \mathrm P_{\widehat M}\widehat\gamma\bs u_{\widehat K},
	\mathrm P_{\widehat M}\widehat\gamma\bs u_{\widehat K}\rangle_{\partial\widehat K}
	=0.
\]
This argument uses that $\bs\tau$ is piecewise constant, so that multiplication by $\bs\tau$ is an endomorphism in $\widehat{\bs M}$. Since $\bs\tau$ is positive definite on each face, this proves that $\mathrm P_{\widehat M}\widehat\gamma\bs u_{\widehat K}=\bs 0$. 

Using the above conclusion and taking $\bs\mu=\widehat\gamma_n \bs\sigma_{\widehat K}\in \widehat{\bs M}$ in \eqref{eq:3.3c} (cf. Theorem \ref{th:2.2}(c)), it follows that $\widehat\gamma_n\bs\sigma_{\widehat K}=\bs 0$. Given the fact that $\bs\sigma_{\widehat K}\in \widehat\Sigma_-^\perp$ (by \eqref{eq:3.3b}) and Theorem \ref{th:2.2}(f), this proves that $\bs\sigma_{\widehat K}=\bs 0$. 

Note finally that $\mathrm{div}\, \bs\varepsilon(\bs u_{\widehat K})\in \widehat{\bs V}_-$ (Theorem \ref{th:2.2}(b)) and $\widehat\gamma_n\bs\varepsilon(\bs u_{\widehat K})\in 
\widehat{\bs M}$ (Theorem \ref{th:2.2}(c)), so that
\begin{alignat*}{6}
(\bs\varepsilon(\bs u_{\widehat K}),\bs\varepsilon(\bs u_{\widehat K}))_{\widehat K}
=& -(\mathrm{div}\,\bs\varepsilon(\bs u_{\widehat K}),\bs u_{\widehat K})_{\widehat K}
+\langle \widehat\gamma_n \bs\varepsilon(\bs u_{\widehat K}),
\widehat\gamma\bs u_{\widehat K}\rangle_{\partial\widehat K}\\
=&\langle \widehat\gamma_n \bs\varepsilon(\bs u_{\widehat K}),
\mathrm P_{\widehat M}\widehat\gamma\bs u_{\widehat K}\rangle_{\partial\widehat K}=0,
\end{alignat*}
and therefore $\bs u_{\widehat K}\in \mathcal M$. This implies that $\widehat\gamma\bs u_{\widehat K}=\mathrm P_{\widehat M}\widehat\gamma\bs u_{\widehat K}=\bs 0$ and, therefore, $\bs u_{\widehat K}=\bs 0,$ which completes the proof.
\end{proof}

Looking at the proof of Proposition \ref{prop:3.1}, it is clear that we could have also defined the projection for any $\bs\tau\in \mathcal R_0(\partial\widehat K;\Rsym)$ such that $\bs\tau|_F$ is negative definite on each face.

\begin{proposition}[Stability]\label{prop:3.2}
For any $\tau_{\max}\ge \tau_{\min}>0$ and $k\ge 1$, there exists $C=C(\tau_{\max},\tau_{\min},k)$ such that if $(\bs\sigma_{\widehat K},\bs u_{\widehat K})=\widehat\Pi_0(\bs\sigma,\bs u;\bs\tau)$, then
\begin{equation}\label{eq:3.55}
\|\bs\sigma_{\widehat K}\|_{\widehat K}
+\|\bs u_{\widehat K}\|_{\widehat K}
\le C (\|\bs\sigma\|_{1,\widehat K}+\|\bs u\|_{1,\widehat K})
\quad \forall (\bs\sigma,\bs u)\in H^1(\widehat K;\Rsym\times \R^3),
\end{equation}
whenever $\bs\tau\in \mathcal R_0(\partial\widehat K;\Rsym)$ satisfies
\begin{equation}\label{eq:3.5}
\tau_{\min}\le \bs\mu\cdot (\bs\tau|_F\bs\mu) \le \tau_{\max}
 \qquad  
\forall\bs\mu\in \R^3, |\bs\mu|=1, \quad \forall F\in \mathcal F(\widehat K).
\end{equation}
\end{proposition}

\begin{proof}
Numbering the faces of $\widehat K$ and the entries of a symmetric matrix with indices from one to six, we can identify $\mathcal R_0(\partial\widehat K;\Rsym)\equiv (\Rsym)^4\equiv \R^{24}$. We can thus make the identification
\[
\{\bs\tau\in \mathcal R_0(\partial\widehat K;\Rsym)\,:\,
\bs\tau\mbox{ satisfies  \eqref{eq:3.0}}\}  \equiv \mathbb O \subset \R^{24},
\]
where $\mathbb O$ is an open set. We can also identify
\[
\{\bs\tau\in \mathcal R_0(\partial\widehat K;\Rsym)\,:\,
\bs\tau\mbox{ satisfies  \eqref{eq:3.5}}\}  \equiv \mathbb K 
 \subset \R^{24},
\]
where $\mathbb K\subset\mathbb O$ is compact.  To be more specific, we can write the identification $\bs\tau\equiv \mathbf x=(x_1,\ldots,x_{24})$ as follows:
\begin{align*}
\bs\tau=\sum_{i=1}^{24}x_i\bs\chi_i,
\end{align*}
where $\bs\chi_i\in\mc R_0(\pp\widehat{K};\mbb R_\mr{sym}^{3\times3})$ satisfies that
$\bs\chi_{6(p-1)+q}$ is supported on the $p$-th face ($p=1\rightarrow 4$) and its $q$-th entry ($q=1\rightarrow 6$) in $\mbb R_\mr{sym}^{3\times3}$ is equal to $1$ and the rest of the entries are all equal to $0$.

Consider now the continuous trial and test spaces
\begin{alignat*}{6}
\mathcal U:= & H^1(\widehat K;\Rsym)\times H^1(\widehat K;\R^3),\\
\mathcal V:= & 
	L^2(\widehat K;\Rsym) \times L^2(\widehat K;\R^3)
	\times  L^2(\partial\widehat K;\R^3)\times H^1(\widehat K;\R^3), 
\end{alignat*}
and their discrete counterparts
\[
\mathcal U_k := \widehat\Sigma_+\times \widehat{\bs V}_+,
\qquad
\mathcal V_k :=  
	\widehat{\bs V}_- \times \widehat\Sigma_-
	\times \widehat{\bs M} \times \widehat{\bs V}^\perp.
\]
We can define bilinear forms
\[
a_j:\mathcal U\times \mathcal V\to \mathbb R \qquad j=0,\ldots,24
\]
as follows:
\begin{align*}
a_0((\bs\sigma,\bs u),(\bs v,\bs\theta,\bs\mu,\bs w))&:=
(\bs\sigma,\bs\theta)_{\widehat{K}}+(\bs u,\bs v)_{\widehat{K}}
+\dualpr{\widehat{\gamma}_n\bs\sigma,\bs\mu}_{\pp\widehat{K}}
-(\div\bs\sigma,\bs w)_{\widehat{K}},\\
a_j((\bs\sigma,\bs u),(\bs v,\bs\theta,\bs\mu,\bs w))&:=
-\dualpr{\bs\chi_j\mr P_{\widehat{M}}\widehat{\gamma}\bs u,\bs\mu}_{\pp\widehat{K}}
+\dualpr{\bs\chi_j\mr P_{\widehat{M}}\widehat{\gamma}\bs u,\widehat{\gamma}\bs w}_{\pp\widehat{K}}.
\end{align*}
We can see that $a_j$ are all bounded $\bs\tau$-independent (but dependent on $k$ through the operator $\mathrm P_{\widehat M}$) bilinear forms.
Then equations \eqref{eq:3.1} can be rephrased in the following form: given $U\in \mathcal U$, find $\widehat\Pi_0 U\in \mathcal U_k$ such that
\begin{equation}\label{eq:3.7}
a_0(\widehat\Pi_0 U-U,V)+\sum_{j=1}^{24} x_j a_j(\widehat\Pi_0 U-U,V)=0
\qquad\forall V\in \mathcal V_k,
\end{equation}
Equations \eqref{eq:3.7} are uniquely solvable for every $\mathbf x\in \mathbb O$ (this is a restatement of Proposition \ref{prop:3.1})  and define a bounded linear operator
$T_k(\mathbf x): \mathcal U \to \mathcal U_k.$ The function $T_k$, from $\mathbb O$ to the space of bounded linear operators from $\mathcal U$ to $\mathcal U_k$, is rational and therefore bounded on the compact set $\mathbb K$. We can thus bound
\begin{equation}\label{eq:3.8}
\| T_k(\mathbf x) U\|_\star \le C(k,\mathbb K,\star) \| U\|_{\mathcal U}
\qquad\forall U\in \mathcal U,
\end{equation}
where $\|\cdot\|_\star$ is any norm we choose in $\mathcal U_k$. If we select the norm
\[
\|(\bs\sigma,\bs u)\|_\star:=
	\|\bs\sigma\|_{\widehat K}
	+\|\bs u\|_{\widehat K},
\]
in $\mathcal U_k$, then \eqref{eq:3.8} becomes \eqref{eq:3.55}.
\end{proof}

A simple argument shows that algebraic condition \eqref{eq:3.5} is equivalent to asking that the spectrum of $\bs\tau|_F$ is contained in $[\tau_{\min},\tau_{\max}]$ for all $F\in \mathcal F(\widehat K)$ and also to the inequality
\begin{equation}\label{eq:3.55a}
\tau_{\min}\|\bs\mu\|_{\partial\widehat K}^2
	\le \langle\bs\tau\,\bs\mu,\bs\mu\rangle_{\partial\widehat K}
	\le \tau_{\max}\|\bs\mu\|_{\partial\widehat K}^2
\qquad
\forall\bs\mu\in L^2(\partial\widehat K;\R^3).
\end{equation}

\subsection{The projection and the remainder}

Let now $\mathrm P_{\widehat\Sigma}:L^2(\widehat K;\Rsym)\to \widehat\Sigma$ be the orthogonal projection onto $\widehat\Sigma$. For $(\bs\sigma,\bs u)\in H^1(\widehat K;\Rsym)\times H^1(\widehat K;\R^3)$ we define
\begin{subequations}\label{eq:3.10}
\begin{equation}
\label{eq:3.10a}
\widehat\Pi(\bs\sigma,\bs u;\bs\tau):=(\bs\sigma_{\widehat K}^c,\bs u_{\widehat K})
	\in \widehat\Sigma \times \widehat{\bs V}_+,
\qquad
\widehat{\mathrm R}(\bs\sigma,\bs u;\bs\tau):=
	\widehat\gamma_n(\bs\sigma_{\widehat K}-\bs\sigma_{\widehat K}^c)
	\in \widehat{\bs M},
\end{equation}
where
\begin{equation}
\label{eq:3.10b}
\bs\sigma_{\widehat K}^c:=\mathrm P_{\widehat\Sigma}\bs\sigma_{\widehat K},
	\qquad
(\bs\sigma_{\widehat K},\bs u_{\widehat K})=\widehat\Pi_0(\bs\sigma,\bs u;\bs\tau).
\end{equation}
\end{subequations}
Since $\widehat\Sigma\times \widehat{\bs V}_+\subset \widehat\Sigma_+\times \widehat{\bs V}_+$, it follows that $\widehat\Pi_0(\bs\theta,\bs v;\bs\tau)=(\bs\theta,\bs v)$ for all $(\bs\theta,\bs v)\in \widehat\Sigma\times \widehat{\bs V}_+$ and therefore
\begin{equation}\label{eq:3.111}
\widehat\Pi(\bs\theta,\bs v;\bs\tau)=(\bs\theta,\bs v),
	\qquad
\widehat{\mathrm R}(\bs\theta,\bs v;\bs\tau)=0
	\qquad
\forall (\bs\theta,\bs v)\in \widehat\Sigma\times \widehat{\bs V}_+.
\end{equation}
In particular, $\widehat\Pi$ is a projection onto $\widehat\Sigma\times \widehat{\bs V}_+$. 
Note that we do not give a set of equations to define $\widehat\Pi$, which is given as the application of $\widehat\Pi_0$ followed by an orthogonal projection applied to the first component of the output. However, the following equations relate the projection $\widehat\Pi$ and the associated remainder $\widehat{\mathrm R}$.

\begin{proposition}\label{prop:3.3}
Let $\bs\tau\in \mathcal R_0(\partial\widehat K;\Rsym)$ satisfy \eqref{eq:3.0}, then $(\bs\sigma_{\widehat K}^c,\bs u_{\widehat K}):=\widehat\Pi(\bs\sigma,\bs u;\bs\tau)$ and $\bs\delta_{\widehat K}:=\widehat{\mathrm R}(\bs\sigma,\bs u;\bs\tau)$ satisfy
\begin{subequations}
\begin{alignat}{6}
\label{eq:3.11a}
(\bs u_{\widehat K}-\bs u,\mathrm{div}\,\bs\theta)_{\widehat K} &= 0 
 \qquad && \forall \bs\theta \in \widehat\Sigma,\\
 \label{eq:3.11c}
 -(\mathrm{div}\,(\bs\sigma_{\widehat K}^c-\bs\sigma),\bs w)_{\widehat K}
	+\langle\bs\tau \mathrm P_{\widehat M}\widehat\gamma (\bs u_{\widehat K}-\bs u),
	\widehat\gamma \bs w\rangle_{\partial\widehat K} &=
	\langle\bs\delta_{\widehat K},\widehat\gamma \bs w\rangle_{\partial\widehat K}
&\quad&\forall \bs w\in \widehat{\bs V}_+,\\
\label{eq:3.11b}
-\langle \widehat\gamma_n(\bs\sigma_{\widehat K}^c-\bs\sigma)
	-\bs\tau \mathrm P_{\widehat M}\widehat\gamma (\bs u_{\widehat K}-\bs u),
	\bs\mu\rangle_{\partial\widehat K} &=
	\langle\bs\delta_{\widehat K}, \bs\mu \rangle_{\partial\widehat K} 
&&\forall \bs\mu\in \widehat{\bs M}.
\end{alignat}
\end{subequations}
\end{proposition}

\begin{proof}
Following the definition \eqref{eq:3.10}, we introduce the intermediate projection $(\bs\sigma_{\widehat K},\bs u_{\widehat K})=\widehat\Pi_0(\bs\sigma,\bs u;\bs\tau)$ so that $\bs\sigma_{\widehat K}^c=\mathrm P_{\widehat\Sigma}\bs\sigma_{\widehat K}$ and $\bs\delta_{\widehat K}=\widehat\gamma_n(\bs\sigma_{\widehat K}-\bs\sigma_{\widehat K}^c)$. 

Since $\mathrm{div}\,\widehat\Sigma\subset \widehat{\bs V}_-$ (this is a direct consequence of the definitions), \eqref{eq:3.11a} is a consequence of \eqref{eq:3.1b}. Note also that $\bs\varepsilon(\widehat{\bs V}_+)\subset \widehat\Sigma$ (again by definition) and therefore 
\begin{equation}\label{eq:3.12}
(\bs\sigma_{\widehat K}^c-\bs\sigma_{\widehat K},
	\bs\varepsilon(\bs w))_{\widehat K}
	=0 \quad \forall \bs w\in \widehat{\bs V}_+.
\end{equation}
The identity \eqref{eq:3.11b} is a direct consequence of \eqref{eq:3.1d}. Finally, by \eqref{eq:3.2} and \eqref{eq:3.12}
\begin{alignat*}{6}
 -(\mathrm{div}\,(\bs\sigma_{\widehat K}^c-\bs\sigma),\bs w)_{\widehat K}
	+\langle\bs\tau \mathrm P_{\widehat M}\widehat\gamma (\bs u_{\widehat K}-\bs u),
	\widehat\gamma \bs w\rangle_{\partial\widehat K}
	=& -(\mathrm{div}\,(\bs\sigma_{\widehat K}^c-\bs\sigma_{\widehat K}),
		\bs w)_{\widehat K}\\
	=&-\langle\widehat\gamma_n (\bs\sigma_{\widehat K}^c-\bs\sigma_{\widehat K}),
		\widehat\gamma \bs w\rangle_{\partial\widehat K},
\end{alignat*}
which proves \eqref{eq:3.11c}.
\end{proof}

\begin{proposition}[Estimate in the reference element]\label{prop:3.4}
For any $\tau_{\max}\ge \tau_{\min}>0$, and integers $k\ge 1$, $1\le m\le k+1$, there exists $C=C(\tau_{\max},\tau_{\min},k,m)$ such that if
\begin{alignat*}{6}
 (\bs\sigma_{\widehat K}^c,\bs u_{\widehat K})&=
\widehat\Pi(\bs\sigma,\bs u;\eta\bs\tau), 
\qquad (\bs\sigma,\bs u)\in H^m(\widehat K;\Rsym)
\times H^{m+1}(\widehat K;\R^3)
\qquad 0\neq \eta\in \R,\\
 \bs\delta_{\widehat K}&=\widehat{\mathrm R}(\bs\sigma,\bs u;\eta\bs\tau),
\end{alignat*}
with $\bs\tau\in \mathcal R_0(\partial\widehat K;\Rsym)$ satisfying \eqref{eq:3.5}, then 
\[
\|\bs\sigma-\bs\sigma_{\widehat K}^c\|_{\widehat K}
+|\eta|\, \|\bs u-\bs u_{\widehat K}\|_{\widehat K}
+\|\bs\delta_{\widehat K}\|_{\partial\widehat K}
\le C \left( |\bs\sigma|_{m,\widehat K}+|\eta|\, |\bs u|_{m+1,\widehat K}\right).
\]
\end{proposition}

\begin{proof}
Let $(\bs\sigma_{\widehat K},\bs u_{\widehat K})=\widehat\Pi_0(\bs\sigma,\bs u;\eta\,\bs\tau)$. Note that $(\bs\sigma_{\widehat K},\eta\,\bs u_{\widehat K})=\widehat\Pi_0(\bs\sigma,\eta\,\bs u;\bs\tau)$ (this is obvious from the equations that define $\widehat\Pi_0$, namely \eqref{eq:3.1}) and therefore 
\[
(\bs\sigma_{\widehat K}^c,\eta\,\bs u_{\widehat K})
	=\widehat\Pi(\bs\sigma,\eta\,\bs u;\bs\tau),
		\qquad
\bs\delta_{\widehat K}=\widehat{\mathrm R}(\bs\sigma,\eta\,\bs u;\bs\tau).
\]
By Proposition \ref{prop:3.2}, we have  
\[
\|\bs\sigma_{\widehat K}\|_{\widehat K}
+|\eta|\, \|\bs u_{\widehat K}\|_{\widehat K}
\le C(\|\bs\sigma\|_{1,\widehat K}+|\eta|\,\|\bs u\|_{1,\widehat K}),
\]
for some constant $C$ depending only on $\tau_{\min}$, $\tau_{\max}$ and $k$. Since $\widehat\gamma_n:\widehat\Sigma_+\to \widehat{\bs M}$ is bounded, there exists a constant $D=D(k)$ such that
\[
\|\bs\delta_{\widehat K}\|_{\partial\widehat K}
\le D\,\|\bs\sigma_{\widehat K}-\bs\sigma_{\widehat K}^c\|_{\widehat K}
\le 2 D\, \|\bs\sigma-\bs\sigma_{\widehat K}\|_{\widehat K}
+D\|\bs\sigma-\mr P_{\widehat\Sigma}\bs\sigma\|_{\widehat K}. 
\]
Taking now $\bs\theta\in \widehat\Sigma$ and $\bs v\in \widehat{\bs V}_+$, and applying \eqref{eq:3.111} to the pair $(\bs\theta,\eta\,\bs v)$, we have
\[
\|\bs\sigma-\bs\sigma_{\widehat K}\|_{\widehat K}
+|\eta|\, \|\bs u-\bs u_{\widehat K}\|_{\widehat K}
\le (1+C) (\|\bs\sigma-\bs\theta\|_{1,\widehat K}
		+|\eta|\, \|\bs u-\bs v\|_{1,\widehat K}),
\]
and
\[
\|\bs\delta_{\widehat K}\|_{\partial\widehat K}
\le 2 D (1+C) (\|\bs\sigma-\bs\theta\|_{1,\widehat K}+|\eta|\, \|\bs u-\bs v\|_{1,\widehat K})+D\|\bs\sigma-\mr P_{\widehat{\Sigma}}\bs\sigma\|_{\widehat K}.
\]
Note that $\mr P_{\widehat{\Sigma}}$ is the $L_2$ projection onto $\widehat\Sigma$.
Therefore, there exists a constant $C'=C'(\tau_{\min},\tau_{\max},k)$ such that
\[
\|\bs\sigma-\bs\sigma_{\widehat K}\|_{\widehat K}
+|\eta|\, \|\bs u-\bs u_{\widehat K}\|_{\widehat K}
+\|\bs\delta_{\widehat K}\|_{\partial\widehat K}
\le 
	C' \left(\inf_{\bs\theta\in \widehat\Sigma}\|\bs\sigma-\bs\theta\|_{1,\widehat K}
		+|\eta|\, \inf_{\bs v\in \widehat{\bs V}_+}\|\bs u-\bs v\|_{1,\widehat K}\right).
\]
Finally, notice that $\bs\sigma-\bs\sigma_{\widehat{K}}^c=\bs\sigma-\mr P_{\widehat{\Sigma}}\bs\sigma+\mr P_{\widehat{\Sigma}}(\bs\sigma-\bs\sigma_{\widehat{K}})$. The result follows now by a compactness argument (Bramble-Hilbert lemma).
\end{proof}

\section{The projection in the physical elements}\label{sec:proj_phy}
\subsection{Pull-backs and push-forwards}

In this section we derive a systematic approach to changes of variables for vector- and matrix-valued functions from a general shape-regular tetrahedron to the reference element. The language mimics that of \cite{sayas2013raviart} (or \cite{DuSa:2019}). Let $K$ be a tetrahedron 
and $F:\widehat K\to K$ be an invertible affine map from the reference element to $K$. We will denote $\mathrm B:=\mathrm DF$ and $J:=\det\mathrm B$. We also consider the piecewise constant function $a:\partial\widehat K\to (0,\infty)$ such that for all integrable $\phi$,
\[
\int_{\partial K} \phi=\int_{\partial\widehat K} |a| \, \phi\circ F.
\]
The trace and normal trace operators on $K$ will be denoted $\gamma$ and $\gamma_n$. 
Given
\[
\bs u,\bs u^*:K\to \R^3,
	\quad
\bs\sigma,\bs\sigma^*:K\to \Rsym,
	\quad
\bs\mu,\bs\mu^*:\partial K \to \R^3,
\]
we define
\begin{alignat*}{6}
\widehat{\bs u} & := \mathrm B^\top \bs u\circ F & & :\widehat K\to \R^3, 
	& \qquad &
\widecheck{\bs u}^* &:= |J|\mathrm B^{-1} \bs u^*\circ F &  :\widehat K\to \R^3,\\
\widehat{\bs\sigma} &:= |J|\mathrm B^{-1}(\bs\sigma\circ F)\mathrm B^{-\top}
	 &  & : \widehat K \to \Rsym,
	 & &
\widecheck{\bs\sigma}^* &:= \mathrm B^\top (\bs\sigma^*\circ F)\mathrm B
	 &  : \widehat K \to \Rsym,\\
\widehat{\bs\mu} &:= \mathrm B^\top \bs\mu\circ F &&:\partial\widehat K \to \R^3,
	& &
\widecheck{\bs\mu}^* &:=|a| \mathrm B^{-1}\bs\mu^*\circ F
	&:\partial\widehat K \to \R^3,
\end{alignat*}
where $\mathrm B^{-\top}=(\mathrm B^\top)^{-1}$. The following properties are easy to prove.

\begin{proposition}[Change of variables in integrals]\label{prop:4.1}
We have the following identities
\begin{subequations}
\begin{alignat}{6}
\label{eq:4.1a}
(\bs u,\bs u^*)_K 
	&=(\widehat{\bs u},\widecheck{\bs u}^*)_{\widehat K}
		& \qquad & \forall \bs u,\bs u^*\in L^2(K;\R^3),\\
\label{eq:4.1b}
(\bs\sigma,\bs\sigma^*)_K
	&=(\widehat{\bs\sigma},\widecheck{\bs\sigma}^*)_{\widehat K}
		& & \forall \bs\sigma,\bs\sigma^*\in L^2(K;\Rsym),\\
\label{eq:4.1c}
\langle\bs\mu,\bs\mu^*\rangle_{\partial K}
	&=\langle\widehat{\bs\mu},\widecheck{\bs\mu}^*\rangle_{\partial\widehat K}		
		& & \forall \bs\mu,\bs\mu^*\in L^2(\partial K;\R^3).
\end{alignat}
\end{subequations}
\end{proposition}

The next group of results about changes of variables involve the interaction of integrals with differential operators or trace operators.

\begin{proposition}[Change of variables in bilinear forms]\label{prop:4.2}
We have the following indentities:
\begin{subequations}
\begin{alignat}{6}
\label{eq:4.2a}
(\bs\sigma,\bs\varepsilon(\bs u))_K
	&=(\widehat{\bs\sigma},\widehat{\bs\varepsilon}(\widehat{\bs u}))_{\widehat K}
	& \qquad & \forall \bs\sigma\in L^2(K;\Rsym), 
				&&\forall\bs u\in H^1(K;\R^3),\\
\label{eq:4.2b}
(\bs u,\mathrm{div}\,\bs\sigma)_K
	&=(\widehat{\bs u},\widehat{\mathrm{div}}\,\widehat{\bs\sigma})_{\widehat K}
	& & \forall \bs \sigma\in H^1(K;\Rsym), 
				&&\forall\bs u\in L^2(K;\R^3),\\
\label{eq:4.2co}
\langle \gamma_n\bs\sigma,\gamma\bs u\rangle_{\partial K}
	& =\langle\widehat\gamma_n \widehat{\bs\sigma},
		\widehat\gamma\widehat{\bs u}\rangle_{\partial\widehat K}
	& & \forall \bs\sigma\in H^1(K;\Rsym), 
				&&\forall\bs u\in H^1(K;\R^3),\\
\label{eq:4.2c}
\langle \gamma_n\bs\sigma,\bs\mu\rangle_{\partial K}
	& =\langle\widehat\gamma_n \widehat{\bs\sigma},
		\widehat{\bs\mu}\rangle_{\partial\widehat K}
	& & \forall \bs\sigma\in H^1(K;\Rsym), 
				&&\forall\bs\mu\in L^2(\partial K;\R^3),\\
\label{eq:4.2d}
\langle \bs\tau\,\gamma\bs u,\bs\mu\rangle_{\partial K}
	& =\langle\widecheck{\bs\tau}\,\widehat\gamma \widehat{\bs u},
		\widehat{\bs\mu}\rangle_{\partial\widehat K}
	&& \forall \bs u\in H^1(K;\R^3), 
				&&\forall\bs\mu\in L^2(\partial K;\R^3),\\
\label{eq:4.2e}
\langle \bs\tau\,\gamma\bs u,\gamma \bs v\rangle_{\partial K}
	& =\langle\widecheck{\bs\tau}\,\widehat\gamma \widehat{\bs u},
		\widehat\gamma \widehat{\bs v}\rangle_{\partial\widehat K}
	&& \forall \bs u, \bs v\in H^1(K;\R^3).
\end{alignat}
\end{subequations}
Here $\bs\tau\in \mathcal R_0(\partial K;\Rsym)$ and
\begin{equation}\label{eq:taucheck}
\widecheck{\bs\tau}:=|a| \mathrm B^{-1}(\bs\tau\circ F)\mathrm B^{-\top}
	\in \mathcal R_0(\partial\widehat K;\Rsym).
\end{equation}
\end{proposition}

\begin{proof}
Using the definitions, it is easy to prove that
\begin{equation}\label{eq:4.44}
\widecheck{\bs\varepsilon(\bs u)}=\widehat{\bs\varepsilon}(\widehat{\bs u}),
	\qquad
\widehat{\gamma\bs u}=\widehat\gamma\widehat{\bs u},
	\qquad 
\widecheck{\bs\tau\bs\mu}=\widecheck{\bs\tau}\,\widehat{\bs\mu}.
\end{equation}
Then \eqref{eq:4.2a}, \eqref{eq:4.2d}, and \eqref{eq:4.2e} are easy consequences of Proposition \ref{prop:4.1}. Let now $\mathrm B=(b_{ij})_{i,j=1}^3$ and $\mathrm A=(a_{ij})_{i,j=1}^3=\mathrm B^{-1}$. Using implicit summation for repeated indices, we have for each $i$
\begin{alignat*}{6}
(\widehat{\mathrm{div}}\,\widehat{\bs\sigma})_i
	&=\partial_{\widehat x_k}\widehat{\bs\sigma}_{ki}
	=|J| a_{kl}a_{im}\partial_{\widehat x_k}(\sigma_{lm}\circ F)\\
	&=|J| a_{kl} a_{im} b_{jk} (\partial_{x_j}\sigma_{lm}) \circ F
	=|J| \delta_{jl} a_{im} (\partial_{x_j}\sigma_{lm}) \circ F\\
	&=|J| a_{im} (\partial_{x_j}\sigma_{jm})\circ F 
	=|J| a_{im} (\mathrm{div}\,\bs\sigma)_m\circ F\\
	&=|J| (\mathrm B^{-1}\,\mathrm{div}\,\bs\sigma)_i\circ F
	=(\widecheck{\mathrm{div}\,\bs\sigma})_i
\end{alignat*}
and therefore $\widehat{\mathrm{div}}\,\widehat{\bs\sigma}=\widecheck{\mathrm{div}\,\bs\sigma}$. This and \eqref{eq:4.1a} prove \eqref{eq:4.2b}. Note next that, using identities we have already proved and Proposition \ref{prop:4.1}, we have
\begin{subequations}\label{eq:5.2}
\begin{alignat}{6}
\langle\widecheck{\gamma_n\bs\sigma},
	\widehat\gamma\widehat{\bs u}\rangle_{\partial\widehat K}
		=&\langle\widecheck{\gamma_n\bs\sigma},
	\widehat{\gamma \bs u}\rangle_{\partial\widehat K}
		= \langle\gamma_n\bs\sigma,\gamma\bs u\rangle_{\partial K}
		=(\mathrm{div}\,\bs\sigma,\bs u)_K
		+(\bs\sigma,\bs\varepsilon(\bs u))_K\\ 
		=&(\widehat{\mathrm{div}}\,\widehat{\bs\sigma},\widehat{\bs u})_{\widehat K}
		+(\widehat{\bs\sigma},\widehat{\bs\varepsilon}(\widehat{\bs u}))_{\widehat K}
		=\langle\widehat\gamma_n\widehat{\bs\sigma},
			\widehat\gamma\widehat{\bs u}\rangle_{\partial \widehat K},
\end{alignat}
\end{subequations}
and \eqref{eq:4.2co} is thus proved.
Since $H^{1/2}(\partial K;\R^3)$ is dense in $L^2(\partial K;\R^3)$,  \eqref{eq:4.2c} follows from \eqref{eq:4.2co}.  
\end{proof}

The final result of this section contains all scaling inequalities. 
Shape-regularity can be rephrased as the asymptotic equivalences (recall that we are in three dimensions)
\begin{equation}\label{eq:4.3}
\|\mathrm B\|\approx h_K, 
	\qquad
\|\mathrm B^{-1}\|\approx h_K^{-1},
	\qquad
|J| \approx h_K^3,
	\qquad
|a| \approx h_K^2.
\end{equation}
Therefore
\begin{equation}\label{eq:4.4}
\|\phi\|_K \approx h_K^{3/2} \|\phi\circ F\|_{\widehat K},
	\qquad
\|\phi\|_{\partial K} \approx h_K \|\phi\circ F\|_{\partial\widehat K}.
\end{equation}

For the hat transformations, we have the following scaling rules, which can be easily proved using \eqref{eq:4.3}, \eqref{eq:4.4}, and the chain rule. 

\begin{proposition}[Scaling equivalences]\label{prop:4.3}
With hidden constants depending on shape-regularity and on $m\ge 0$, we have
\begin{subequations}
\begin{alignat}{6}
|\widehat{\bs u}|_{m,\widehat K} 
	& \approx h_K^{m-1/2}|\bs u|_{m,K}
	& \qquad & \forall \bs u\in H^m(K;\R^3),\\
|\widehat{\bs\sigma}|_{m,\widehat K} 
	& \approx h_K^{m-1/2} |\bs\sigma|_{m,K}
	& & \forall \bs\sigma\in H^m(K;\Rsym),\\
\|\widecheck{\bs\mu}\|_{\partial\widehat K}\approx
\|\widehat{\bs\mu}\|_{\partial\widehat K}
	& \approx \|\bs\mu\|_{\partial K} 
	& & \forall \bs\mu\in L^2(\partial K;\R^3). 
\end{alignat}
\end{subequations}
\end{proposition}

\subsection{The projection and the remainder on $K$}
Consider the spaces
\begin{alignat*}{6}
\Sigma(K):=&
	\{ \bs\sigma\in L^2(K;\Rsym)\,:\,\widehat{\bs\sigma}\in \widehat\Sigma\}
	=\mathcal P_k(K;\Rsym),\\
\bs V_+(K) :=& \{\bs u\in L^2(K;\R^3)\,:\,\widehat{\bs u}\in \widehat{\bs V}_+\}
	=\mathcal P_{k+1}(K;\R^3),\\
\bs M(\partial K):=&
	\{\bs\mu\in L^2(\partial K;\R^3)\,:\,\widehat{\bs\mu}\in \widehat{\bs M}\}\\
	=&\{\bs\mu\in L^2(\partial K;\R^3)\,:\,\widecheck{\bs\mu}\in \widehat{\bs M}\}
	=\mathcal R_k(\partial K;\R^3).
\end{alignat*}
The projection and the remainder are defined by a pull-back process: given $(\bs\sigma,\bs u)\in H^1(K;\Rsym)\times H^1(K;\R^3)$, we define
\begin{subequations}\label{eq:5.3}
\begin{equation}
\Pi(\bs\sigma,\bs u;\bs\tau):=(\bs\sigma_K,\bs u_K)
	\in \Sigma(K)\times \bs V_+(K), 
\qquad
\mathrm R(\bs\sigma,\bs u;\bs\tau):=\bs\delta_K\in \bs M(\partial K),
\end{equation}
by the relations
\begin{alignat}{6}
\label{eq:4.9bbb}
(\widehat{\bs\sigma_K},\widehat{\bs u_K})
	&=\widehat\Pi(\widehat{\bs\sigma},\widehat{\bs u};\widecheck{\bs\tau}),
	\qquad \widecheck{\bs\tau}:=|a| \mathrm B^{-1}(\bs\tau\circ F)\mathrm B^{-\top},\\
\widecheck{\bs\delta_K} 
	&=\widehat{\mathrm R}(\widehat{\bs\sigma},\widehat{\bs u};\widecheck{\bs\tau}).
\end{alignat}
\end{subequations}
We now prove Theorem \ref{th:PROJ}. We start by proving a technical lemma, continue showing that equations \eqref{eq:tauProp} hold (we present this as Proposition \ref{prop:5.2}), and finish by proving the estimates \eqref{eq:1.3a}. 

\begin{lemma}\label{lemma:5.1}
If $\mathrm P_M$ is the $L^2(\partial K;\R^3)$ orthogonal projector onto $\bs M(\partial K)$, then
\begin{equation}\label{eq:5.1}
\widehat{\mathrm P_M\bs\mu}=\mathrm P_{\widehat M}\widehat{\bs\mu}
	\qquad
	\forall\bs\mu \in L^2(\partial K;\R^3).
\end{equation}
Therefore, if $\bs\tau\in \mathcal R_0(\partial K;\Rsym)$, we have 
\begin{alignat*}{6}
\langle \bs\tau\,\mathrm P_M\gamma\bs u,\bs\mu\rangle_{\partial K}
	& =\langle\widecheck{\bs\tau}\,
	\mathrm P_{\widehat M}\widehat\gamma \widehat{\bs u},
		\widehat{\bs\mu}\rangle_{\partial\widehat K}
	&\qquad & \forall \bs u\in H^1(K;\R^3), 
				&&\forall\bs\mu\in L^2(\partial K;\R^3),\\
\langle \bs\tau\,\mathrm P_M\gamma\bs u,\gamma \bs v\rangle_{\partial K}
	& =\langle\widecheck{\bs\tau}\,
	\mathrm P_{\widehat M}\widehat\gamma \widehat{\bs u},
		\widehat\gamma \widehat{\bs v}\rangle_{\partial\widehat K}
	&& \forall \bs u, \bs v\in H^1(K;\R^3),
\end{alignat*}
with $\widecheck{\bs\tau}$ defined in \eqref{eq:4.9bbb}.
\end{lemma}

\begin{proof}
It follows from the definitions that $\widehat{\mathrm P_M\bs\mu}\in \widehat{\bs M}$ and that $\bs\mu^*\in \bs M(K)$ if and only if $\widecheck{\bs\mu}^*\in \widehat{\bs M}$. Therefore,
\[
\langle\widehat{\mathrm P_M\bs\mu},\widecheck{\bs\mu}^*\rangle_{\partial\widehat K}
	=\langle\mathrm P_M\bs\mu,\bs\mu^*\rangle_{\partial K}
	=\langle\bs\mu,\bs\mu^*\rangle_{\partial K}
	=\langle \widehat{\bs\mu},\widecheck{\bs\mu}^*\rangle_{\partial\widehat K}
	=\langle \mathrm P_{\widehat M}\widehat{\bs\mu},
		\widecheck{\bs\mu}^* \rangle_{\partial\widehat K}
		\qquad\forall\bs\mu^*\in \widehat{\bs M},
\]
and \eqref{eq:5.1} is proved. The next two identities in the statement follow from \eqref{eq:5.1} and Proposition \ref{prop:4.1} and  \eqref{eq:4.44}. 
\end{proof}

\begin{proposition}\label{prop:5.2}
Assume that $\bs\tau\in \mathcal R_0(\partial K;\Rsym)$ satisfies
\[
\langle\bs\tau\bs\mu,\bs\mu\rangle_{\partial K}>0 \qquad
\forall \bs\mu \in L^2(\partial K;\R^3), \quad\bs\mu\neq\bs 0.
\] 
 If $(\bs\sigma_K,\bs u_K)=\Pi(\bs\sigma,\bs u;\bs\tau)$ and $\bs\delta_K=\mathrm R(\bs\sigma,\bs u;\bs\tau)$, then the following equations hold:
\begin{subequations}
\begin{alignat}{6}
\label{eq:5.33a}
(\bs u_K-\bs u,\mathrm{div}\,\bs\theta)_{K} &= 0 
 \qquad && \forall \bs\theta \in \Sigma(K),\\  
\label{eq:5.33b}
  -(\mathrm{div}\,(\bs\sigma_K-\bs\sigma),\bs w)_K
	+\langle\bs\tau \mathrm P_M\gamma (\bs u_K-\bs u),
	\gamma \bs w\rangle_{\partial K} &=
	\langle\bs\delta_K,\gamma \bs w\rangle_{\partial K}
&\,\,&\forall \bs w\in \bs V_+(K),\\
\label{eq:5.33c}
-\langle \gamma_n(\bs\sigma_K-\bs\sigma)
	-\bs\tau \mathrm P_M \gamma (\bs u_K-\bs u),
	\bs\mu\rangle_{\partial  K} &=
	\langle\bs\delta_K, \bs\mu \rangle_{\partial K} 
&&\forall \bs\mu\in \bs M(\partial K).
\end{alignat}
\end{subequations}
\end{proposition}

\begin{proof}
The equality \eqref{eq:5.33a} follows from changing variables to the reference element, applying \eqref{eq:3.11a} and \eqref{eq:4.2b}. Similarly, \eqref{eq:5.33b} follows from \eqref{eq:3.11c} using \eqref{eq:4.2b}, Lemma \ref{lemma:5.1}, \eqref{eq:4.1c} and \eqref{eq:4.44}. Finally, \eqref{eq:5.33c} follows from \eqref{eq:3.11b} applying \eqref{eq:4.2c}, Lemma \ref{lemma:5.1}, and \eqref{eq:4.1c}. 
\end{proof}

Assume now that $\bs\tau$ is of order $h_K^{-1}$, as expressed in \eqref{eq:tauhm1}. By \eqref{eq:4.3} and Proposition \ref{prop:4.3} we can write (take $\bs\eta=\widecheck{\bs\mu}$ for $\bs\mu\in L^2(\partial K;\R^3)$)
\begin{equation}\label{eq:5.5}
\alpha_1 C_1 h_K^{-1}\|\bs\eta\|_{\partial\widehat K}^2
	\le \langle \widecheck{\bs\tau}\, \bs\eta,
		\bs\eta\rangle_{\partial\widehat K}
	\le \alpha_2 C_2 h_K^{-1}\|\bs\eta\|_{\partial\widehat K}^2
\qquad
\forall \bs\eta\in L^2(\partial\widehat K;\R^3),
\end{equation}
where $\alpha_1$ and $\alpha_2$ are constants related to shape-regularity of $K$. We are now ready to apply Proposition \ref{prop:3.4} with $\tau_{\min}=\alpha_1 C_1$, $\tau_{\max}=\alpha_2C_2$ and $\eta=h_K^{-1}$ (compare \eqref{eq:5.5} with \eqref{eq:3.55a}). Using the definition of the projection \eqref{eq:5.3}, Proposition \ref{prop:4.3} for the scaling properties, and Proposition \ref{prop:3.4} for the estimates, we have
\begin{alignat*}{6}
\|\bs\sigma-\bs\sigma_K\|_K+h_K^{-1}\|\bs u-\bs u_K\|_K
	+ h_K^{1/2}\|\bs\delta_K\|_K
	\approx & h_K^{1/2}
		(\|\widehat{\bs\sigma}-\widehat{\bs\sigma_K}\|_{\widehat K}
		+h_K^{-1} \|\widehat{\bs u}-\widehat{\bs u_K}\|_{\widehat K}
		+\|\widecheck{\bs\delta_K}\|_{\partial\widehat K})\\
	\lesssim & h_K^{1/2} (|\widehat{\bs\sigma}|_{m,\widehat K}
		+h_K^{-1} |\widehat{\bs u}|_{m+1,\widehat K})\\
	\approx & h_K^m (|\bs\sigma|_{m,K}+ |\bs u|_{m+1,K}),
\end{alignat*}
and \eqref{eq:1.3a} is thus proved. 

\section{Steady-state elasticity}\label{sec:steady_state}

\subsection{Method and convergence estimates}\label{sec:st_main}

From now on, we shift our attention from the construction of the projection to its applications.
We begin by introducing more notation for the rest of the paper. Let $\Omega$ be a Lipschitz polyhedral domain in $\mbb R^3$. We denote the compliance tensor by $\mc A\in L^\infty(\Omega;\mc B(\mbb R_\mathrm{sym}^{3\times3}))$, where $\mc B(\Rsym)$ is the space of linear maps from $\Rsym$ to itself.  We assume $\rho\ge\rho_0$ for some positive constant $\rho_0$, and $\mc A$ is uniformly symmetric and positive almost everywhere on $\Omega$, i.e., there exists $C_0>0$ such that for almost all $\mathbf x\in \Omega$,
\begin{align*}
\left.\begin{array}{l}
(\mc A(\mathbf x)\bs\xi):\bs\chi = \bs\xi:(\mc A(\mathbf x)\bs\chi)\\[5pt]
(\mc A(\mathbf x)\bs\xi):\bs\xi \ge C_0 \bs\xi:\bs\xi
\end{array}
\right\}
\quad\forall\bs\xi,\bs\chi\in\mbb R_\mr{sym}^{3\times3}.
\end{align*} 

Let $\mc T_h$ be a family of conforming tetrahedral partitions of $\Omega$, which
we assume to be shape-regular.
Namely, there exists a fixed constant $c_0>0$, such that
$\frac{h_K}{\rho_K}\le c_0$ for all $K\in\mc T_h$,
where $\rho_K$ denotes the inradius of $K$. We denote $h:=\max_{K\in\mc T_h}h_K$ as the mesh size and set the following discrete spaces:
\begin{align*}
\bs V_h:=\prod_{K\in\mc T_h}\mathcal P_{k}(K;\mathbb R_\mr{sym}^{3\times3}),\quad
\bs W_h:=\prod_{K\in\mc T_h}\mathcal P_{k+1}(K;\mathbb R^3),\quad
\bs M_h:=\prod_{K\in\mc T_h}\mathcal R_k(\pp K;\mathbb R^3).
\end{align*}
The related discrete inner products are denoted as
\begin{align*}
(\bs\sigma,\bs\tau)_{\mc T_h}:=\sum_{K\in\mc T_h}(\bs\sigma,\bs\tau)_K,\quad
(\bs u,\bs v)_{\mc T_h}:=\sum_{K\in\mc T_h}(\bs u,\bs v)_K,\quad
\dualpr{\bs\mu,\bs\lambda}_{\pp\mc T_h}
:=\sum_{K\in\mc T_h}\dualpr{\bs\mu,\bs\lambda}_{\pp K}.
\end{align*}
Finally, we use the following notation for the discrete and the weighted discrete norms:
\begin{alignat*}{5}
&\|\cdot\|_{\mc T_h}:=(\cdot,\cdot)_{\mc T_h}^{1/2},
&\qquad &\|\cdot\|_{\pp\mc T_h}:=\dualpr{\cdot,\cdot}_{\pp\mc T_h}^{1/2},\\
&\|\cdot\|_{*_1}:=(*_1\,\cdot,\cdot)_{\mc T_h}^{1/2},
&&\|\cdot\|_{*_2}:=\dualpr{*_2\,\cdot,\cdot}_{\pp\mc T_h}^{1/2},
\end{alignat*}
where $*_1=\mc A$ or $\rho$, and $*_2=\bs\tau$ or $\bs\tau^{-1}$. Here $\bs\tau\in \prod_{K\in\mc T_h}\mc R_0(\pp K;\Rsym)$ is the stabilization function which we assume to satisfy \eqref{eq:tauhm1} on every element $K$.

In this section, we give a projection-based error analysis to the HDG+ method introduced in \cite{QiShSh:2018}. We will show that the analysis can be simplified by using Theorem \ref{th:PROJ}. To begin with, we review the steady-state linear elasticity equations:
\begin{subequations}\label{eq:st_pde}
\begin{alignat}{5}
\mc A\bs\sigma - \bs\varepsilon(\bs u) &= \bs 0 &\qquad&\mr{in}\ \Omega,\\
-\div\bs\sigma &= \bs f && \mr{in}\ \Omega,\\
\gamma\bs u&= \bs g &&\mr{on}\ \Gamma:=\pp\Omega,
\end{alignat}
\end{subequations}
where $\bs f\in L^2(\Omega;\mbb R^3)$ and $\bs g\in H^{1/2}(\Gamma;\mbb R^3)$. The HDG+ method for \eqref{eq:st_pde} is: find $(\bs\sigma_h,\bs u_h,\widehat{\bs u}_h)\in \bs V_h\times\bs W_h\times\bs M_h$ such that
\begin{subequations}\label{eq:st_HDG}
\begin{alignat}{5}
(\mc A\bs\sigma_h,\bs\theta)_{\mc T_h} 
+ (\bs u_h,\div\bs\theta)_{\mc T_h}
-\dualpr{\widehat{\bs u}_h,\bs\theta\bs n}_{\pp\mc T_h}&
=0&\quad& \forall \bs\theta\in\bs V_h,\\
-(\div\bs\sigma_h,\bs w)_{\mc T_h} 
+\dualpr{\bs\tau\mr P_M(\bs u_h-\widehat{\bs u}_h),\bs w}_{\pp\mc T_h}
&=(\bs f,\bs w)_{\mc T_h}&\quad & \forall\bs w\in\bs W_h,\\
\dualpr{
\bs\sigma_h\bs n 
-\bs\tau\mr P_M(\bs u_h-\widehat{\bs u}_h),\bs \mu
}_{\partial\mc T_h\backslash\Gamma}
&=0&\quad & \forall\bs\mu\in\bs M_h,\\
\dualpr{\widehat{\bs u}_h,\bs\mu}_\Gamma &= \dualpr{\bs g,\bs\mu}_{\Gamma}
&\quad & \forall\bs\mu\in\bs M_h.
\end{alignat}
\end{subequations}
Since we will use a duality argument to estimate the convergence of ${\bs u}_h$, here we write down the adjoint equations for \eqref{eq:st_pde}:
\begin{subequations}\label{eq:st_adpde}
\begin{alignat}{5}
\mc A\bs\Psi + \bs\varepsilon(\bs\Phi) &= 0 &\qquad& \mr{in}\ \Omega,\\
\div\bs\Psi &= \bs\Theta && \mr{in}\ \Omega,\\
\gamma\bs\Phi&=0 && \mr{on}\ \Gamma,
\end{alignat}
\end{subequations}
with input data $\bs\Theta\in L^2(\Omega;\mbb R^3)$, and we assume the additional elliptic regularity estimate
\begin{align}\label{eq:ell_reg}
\|\bs\Psi\|_{1,\Omega}+\|\bs\Phi\|_{2,\Omega}\le C_\mr{reg}\|\bs\Theta\|_\Omega,
\end{align}
where $C_\mr{reg}$ is a constant depending only on $\mc A$ and $\Omega$. The following convergence theorem will be proved in Subsection \ref{sec:st_prf}.

\begin{theorem}\label{th:st}
For $k\ge1$ and $1\le m\le k+1$, we have
\begin{align*}
\|\bs\sigma-\bs\sigma_h\|_{\mc A}\le C_1h^m(\|\bs\sigma\|_{m,\Omega}+\|\bs u\|_{m+1,\Omega}).
\end{align*} 
If \eqref{eq:ell_reg} holds, then we also have
\begin{align*}
\|\bs u-\bs u_h\|_\Omega\le C_2 h^{m+1}(\|\bs\sigma\|_{m,\Omega}+\|\mb u\|_{m+1,\Omega}).
\end{align*}
Here, $C_1$ depends only on the polynomial degree $k$ and the shape-regularity of $\mc T_h$, while $C_2$ depends also on $\mc A$ and  $C_\mr{reg}$.
\end{theorem}

Note that in \cite[Theorem 2.1]{QiShSh:2018} the meshes $\mc T_h$ are assumed to be quasi-uniform, whereas  here we only require $\mc T_h$ to be shape-regular. 

\subsection{Proof of Theorem \ref{th:st}}\label{sec:st_prf}

To begin with, we use the element-by-element projection defined in \eqref{eq:5.3} (we will only use Theorem \ref{th:PROJ}, not the definition) on the solution $(\bs\sigma,\bs u)$:
\begin{align*}
(\Pi\bs\sigma,\Pi\bs u)
:=\prod_{K\in\mc T_h}\Pi(\bs\sigma\big|_K,\bs u\big|_K;\bs\tau\big|_{\pp K}),\quad
\bs\delta 
:= \prod_{K\in\mc T_h}\mathrm R(\bs\sigma\big|_K,\bs u\big|_K;\bs\tau\big|_{\pp K}),
\end{align*}
and define the error terms and the approximation terms:
\begin{alignat*}{5}
\bs\varepsilon_h^u&:=\Pi\bs u-\bs u_h,&\quad
\bs\varepsilon_h^\sigma&:=\Pi\bs\sigma-\bs\sigma_h,\quad
\bs\varepsilon_h^{\widehat{u}}:=\mr P_M\bs u-\widehat{\bs u}_h,\\
\bs e_\sigma&:=\Pi\bs\sigma-\bs\sigma,&
\bs e_u&:=\Pi\bs u-\bs u.
\end{alignat*}
We aim to control the terms $\bs\varepsilon_h^*$ ($*=\bs u,\bs\sigma,\widehat{\bs u}$) by the terms $\bs e_*$ ($*=\bs\sigma,\bs u$) and $\bs\delta$.

\begin{proposition}[Energy estimate]
\label{prop:energy_id}
The following energy identity holds
\begin{align}\label{eq:st_id}
(\mc A\bs\varepsilon_h^\sigma,\bs\varepsilon_h^\sigma)_{\mc T_h} 
+\dualpr{\bs\tau \mr P_M(\bs\varepsilon_h^u-\bs\varepsilon_h^{\widehat{u}})
,(\bs\varepsilon_h^u-\bs\varepsilon_h^{\widehat{u}})}_{\pp\mc T_h}
=(\mc A\bs e_\sigma,\bs\varepsilon_h^\sigma)_{\mc T_h}+\dualpr{\bs\delta,\bs\varepsilon_h^u-\bs\varepsilon_h^{\widehat{u}}}_{\pp\mc T_h}.
\end{align}
Consequently,
\begin{align}\label{eq:st_est}
\|\bs\varepsilon_h^\sigma\|_{\mc A}^2+\|\mathrm P_M(\bs\varepsilon_h^u-\bs\varepsilon_h^{\widehat u})\|_{\bs\tau}^2
\le 
\|\bs e_\sigma\|_{\mc A}^2+\|\bs\delta\|_{\bs\tau^{-1}}^2.
\end{align}
\end{proposition}
\begin{proof}
The proof here will be similar to the proof of \cite[Lemma 3.1 \& 3.2]{CoGoSa:2010}. By Theorem \ref{th:PROJ}, we first obtain a set of projection equations satisfied by $\Pi\bs\sigma$, $\Pi\bs u$ and $\bs\delta$. We then subtract \eqref{eq:st_HDG} from the projection equations and obtain the following error equations:
\begin{subequations}
\label{eq:st_ereq}
\begin{align}
\label{eq:st_ereq_1}
(\mc A\bs\varepsilon_h^\sigma,\bs\theta)_{\mc T_h}
+ (\bs\varepsilon_h^u,\div\bs\theta)_{\mc T_h}
-\dualpr{\bs\varepsilon_h^{\widehat{u}},\bs\theta\bs n}_{\pp\mc T_h}
&=(\mc A\bs e_\sigma,\bs\theta)_{\mc T_h},\\
\label{eq:st_ereq_2}
-(\div\bs\varepsilon_h^\sigma,\bs w) _{\mc T_h}
+ \dualpr{\bs\tau \mr P_M(\bs\varepsilon_h^u-\bs\varepsilon_h^{\widehat{u}}),\bs w}_{\pp\mc T_h}
&=\dualpr{\bs\delta,\bs w}_{\pp\mc T_h},\\
\label{eq:st_ereq_3}
\dualpr{\bs\varepsilon_h^\sigma\bs n 
-\bs\tau \mr P_M(\bs\varepsilon_h^u-\bs\varepsilon_h^{\widehat{u}}),\bs\mu}
_{\partial\mc T_h\backslash\Gamma}
&=-\dualpr{\bs\delta,\bs\mu}_{\partial\mc T_h\backslash\Gamma},\\
\label{eq:st_ereq_4}
\dualpr{\bs\varepsilon_h^{\widehat{u}},\bs\mu}_\Gamma &= 0.
\end{align}
\end{subequations}
By \eqref{eq:st_ereq_3} and \eqref{eq:st_ereq_4} we have
\begin{align}\label{eq:st_ereq_3a}
\dualpr{\bs\varepsilon_h^\sigma\bs n 
-\bs\tau \mr P_M(\bs\varepsilon_h^u-\bs\varepsilon_h^{\widehat{u}}),
\bs\varepsilon_h^{\widehat{u}}}
_{\partial\mc T_h}
&=-\dualpr{\bs\delta,
\bs\varepsilon_h^{\widehat{u}}}_{\partial\mc T_h}.
\end{align}
Now taking $\bs\theta = \bs\varepsilon_h^\sigma$ in \eqref{eq:st_ereq_1}, $\bs w=\bs\varepsilon_h^u$ in \eqref{eq:st_ereq_2}, and then adding \eqref{eq:st_ereq_1}, \eqref{eq:st_ereq_2} and \eqref{eq:st_ereq_3a},
we obtain the energy identity \eqref{eq:st_id}, from which the estimate \eqref{eq:st_est} follows easily.
\end{proof}

\begin{proposition}[Estimate by duality]\label{prop:st_dual}
If $k\ge1$ and \eqref{eq:ell_reg} holds, then
\begin{align*}
\|\bs\varepsilon_h^u\|_{\mc T_h}\le hC
\left(
\|\bs e_\sigma\|_{\mc T_h}+\|\bs\delta\|_{\bs\tau^{-1}}
\right).
\end{align*}
Here, $C$ is independent of $h$, but depends on $\mc A$ and $C_\mr{reg}$.
\end{proposition}
\begin{proof}
The proof here will be similar to the proof of \cite[Lemma 4.1 \& Theorem 4.1]{CoGoSa:2010}.
Consider the adjoint equations \eqref{eq:st_adpde}. 
We take $\bs\Theta=\bs\varepsilon_h^u$ as the input data and apply the projection on $(\bs\Psi,\bs\Phi)$ with $-\bs\tau$ as the stabilization function. Namely,
\begin{align*}
(\Pi\bs\Psi,\Pi\bs\Phi):=
\prod_{K\in\mc T_h}\Pi(\bs\Psi\big|_K,\bs\Phi\big|_K;-\bs\tau\big|_{\pp K}),\quad
\bs\Delta 
:= \prod_{K\in\mc T_h}\mathrm R(\bs\Psi\big|_K,\bs\Phi\big|_K;-\bs\tau\big|_{\pp K}).
\end{align*}
By Theorem \ref{th:PROJ}, $\Pi\bs\Psi$, $\Pi\bs\Phi$ and $\bs\Delta$ satisfy
\begin{align*}
(\mc A\Pi\bs\Psi,\bs\theta)_{\mc T_h} - (\Pi\bs\Phi,\div\bs\theta)_{\mc T_h}
+\dualpr{\mr P_M\bs\Phi,\bs\theta\bs n}_{\pp\mc T_h}
&=(\mc A(\Pi\bs\Psi-\bs\Psi),\bs\theta)_{\mc T_h},\\
(\div\Pi\bs\Psi,\bs w)_{\mc T_h} + \dualpr{\bs\tau \mr P_M(\Pi\bs\Phi-\bs\Phi),\bs w}_{\pp\mc T_h}
&=(\bs\varepsilon_h^u,\bs w)_{\mc T_h}-\dualpr{\bs\Delta,\bs w}_{\pp\mc T_h},\\
-\dualpr{\Pi\bs\Psi\bs n +\bs\tau \mr P_M(\Pi\bs\Phi-\bs\Phi),\bs\mu}
_{\partial\mc T_h\backslash\Gamma}
&=\dualpr{\bs\Delta,\bs\mu}_{\partial\mc T_h\backslash\Gamma},\\
\dualpr{\mr P_M \bs\Phi,\bs\mu}_\Gamma &= 0.
\end{align*}
We now test the above equations with 
$\bs\theta=\bs\varepsilon_h^\sigma$, $\bs w=\bs\varepsilon_h^u$ and $\bs\mu=\bs\varepsilon_h^{\widehat{u}}$, then test the error equations \eqref{eq:st_ereq} with $\bs\theta=\Pi\bs\Psi$, $\bs w=\Pi\bs\Phi$ and $\bs\mu=\mr P_M\bs\Phi$. By comparing the two sets of equations, we obtain
\begin{align*}
(\mc A(\Pi\bs\Psi-\bs\Psi),\bs\varepsilon_h^\sigma)_{\mc T_h}
+\|\bs\varepsilon_h^u\|_{\mc T_h}^2
-\dualpr{\bs\Delta,\mr P_M\bs\varepsilon_h^u
	-\bs\varepsilon_h^{\widehat{u}}}_{\pp\mc T_h}
=(\mc A\bs e_\sigma,\Pi\bs\Psi)_{\mc T_h}
+\dualpr{\bs\delta,\mr P_M(\Pi\bs\Phi-\bs\Phi)}_{\pp\mc T_h}.
\end{align*}
After rearranging terms, we have
\begin{align*}
\|\bs\varepsilon_h^u\|_{\mc T_h}^2
&=-(\mc A(\Pi\bs\Psi-\bs\Psi),\bs\sigma-\bs\sigma_h)_{\mc T_h}-(\Pi\bs\sigma-\bs\sigma,
	\bs\varepsilon(\bs\Phi))_{\mc T_h}\\
&\quad\, +\dualpr{\bs\Delta,\mr P_M\bs\varepsilon_h^u-\bs\varepsilon_h^{\widehat{u}}}_{\pp\mc T_h}
+\dualpr{\bs\delta,\mr P_M(\Pi\bs\Phi-\bs\Phi)}_{\pp\mc T_h}.
\end{align*}
By Theorem \ref{th:PROJ} (with $m=1$) we have
\begin{align*}
\|\Pi\bs\Psi-\bs\Psi\|_{\mc T_h}+\|\bs\Delta\|_{\bs\tau^{-1}}
\lesssim h(\|\bs\Psi\|_{1,\Omega} + \|\bs\Phi\|_{2,\Omega}).
\end{align*}
By \eqref{eq:rmproj} we have
\begin{align*}
(\Pi\bs\sigma-\bs\sigma,
	\bs\varepsilon(\bs\Phi))_{\mc T_h}
	=(\Pi\bs\sigma-\bs\sigma,
	\bs\varepsilon(\bs\Phi)-\mr P_0\bs\varepsilon(\bs\Phi))_{\mc T_h}
\lesssim h\|\bs e_\sigma\|_{\mc T_h}\|\bs\Phi\|_{2,\Omega}.
\end{align*}
Finally, \eqref{eq:tauhm1} implies
\begin{align*}
\|\mr P_M(\Pi\bs\Phi-\bs\Phi)\|_{\bs\tau}\approx \left(\sum_{K\in\mc T_h}\|h_K^{-1/2}\mr P_M(\Pi\bs\Phi-\bs\Phi)\|_{\partial K}^2\right)^{1/2}
\lesssim h(\|\bs\Psi\|_{1,\Omega} + \|\bs\Phi\|_{2,\Omega}).
\end{align*}
We next use  \eqref{eq:ell_reg} to control the term $(\|\bs\Psi\|_{1,\Omega} + \|\bs\Phi\|_{2,\Omega})$ and use Proposition \ref{prop:energy_id} to control the terms $\|\mc A(\bs\sigma-\bs\sigma_h)\|_{\mc T_h}$ and $\|\mr P_M\bs\varepsilon_h^u-\bs\varepsilon_h^{\widehat{u}}\|_{\bs\tau}$. The proof is thus completed.
\end{proof}

Combing Proposition \ref{prop:energy_id}, Proposition \ref{prop:st_dual} and Theorem \ref{th:PROJ}, Theorem \ref{th:st} follows readily.

\section{Frequency-domain elastodynamics}\label{sec:freq}
\subsection{Method and convergence estimates}
In this section we give new proofs of error estimates to the second HDG+ method studied in \cite{HuPrSa:2017} by using projection based analysis
for the following time-harmonic linear elasticity problem 
\begin{subequations}\label{eq:fq_pde}
\begin{alignat}{5}
\mc A\bs\sigma - \bs\varepsilon(\bs u) &=\bs 0 &\qquad& \mr{in}\ \Omega,\\
-\div\bs\sigma - \kappa^2\rho\bs u&= \bs f && \mr{in}\ \Omega,\\
\gamma\bs u&= \bs g && \mr{on}\ \Gamma:=\pp\Omega,
\end{alignat}
\end{subequations}
where $\bs f\in L^2(\Omega;\mbb C^3)$, $\bs g\in H^{1/2}(\Gamma;\mbb C^3)$, and $\rho\in W^{1,\infty}(\Omega;\mbb R)$ denotes the density function. We assume that the wave number $\kappa>0$ and $\kappa^2$ is not a Dirichlet eigenvalue so that \eqref{eq:fq_pde} is well-posed. Note that because of the different proof techniques, the dependence on the wave numbers of the error estimates will be different and not easy to compare. The emphasis here will be the simplified error analysis thanks to the introduction of the projection.	
 
The HDG+ method for \eqref{eq:fq_pde} is: find $(\bs\sigma_h,\bs u_h,\widehat{\bs u}_h)\in \bs V_h\times\bs W_h\times\bs M_h$ such that
\begin{subequations}\label{eq:fq_hdg}
\begin{align}
(\mc A\bs\sigma_h,\bs\theta)_{\mc T_h} 
+ (\bs u_h,\div\bs\theta)_{\mc T_h}
-\dualpr{\widehat{\bs u}_h,\bs\theta\bs n}_{\pp\mc T_h}
&=0,\\
-(\div\bs\sigma_h,\bs w)_{\mc T_h} 
+ \dualpr{\bs\tau \mr P_M(\bs u_h-\widehat{\bs u}_h),\bs w}_{\pp\mc T_h}
-(\kappa^2\rho \bs u_h,\bs w)_{\mc T_h}
&=(\bs f,\bs w)_{\mc T_h},\\
\dualpr{\bs\sigma_h\bs n -\bs\tau \mr P_M(\bs u_h-\widehat{\bs u}_h),\bs\mu}
_{\partial\mc T_h\backslash\Gamma}
&=0,\\
\dualpr{\widehat{\bs u}_h,\bs\mu}_\Gamma &= \dualpr{\bs g,\bs\mu}_{\Gamma},
\end{align}
\end{subequations}
for all $\bs\theta\in\bs V_h$, $\bs w\in\bs W_h$ and $\bs\mu\in\bs M_h$. Here, the spaces $\bs V_h$, $\bs W_h$ and $\bs M_h$ are the same as those defined in Section \ref{sec:st_main} except now the functions contained in these spaces take complex values.

For $\bs\Theta\in L^2(\Omega;\mbb R^3)$,
we assume that the solution to the adjoint equations for \eqref{eq:fq_pde} 
\begin{subequations}\label{eq:fq_adpde}
\begin{alignat}{5}
\mc A\bs\Psi+ \bs\varepsilon(\bs\Phi) &=\bs 0 &\qquad & \mr{in}\ \Omega,\\
\div\bs\Psi - \kappa^2\rho\bs\Phi&= \bs\Theta && \mr{in}\ \Omega,\\
\gamma\bs\Phi&= \bs 0 &&\mr{on}\ \Gamma,
\end{alignat}
\end{subequations}
satisfy
\begin{align}\label{eq:fq_reg}
\|\bs\Phi\|_{2,\Omega} + \|\bs\Psi\|_{1,\Omega}
\le C_\kappa\|\bs\Theta\|_\Omega,
\end{align}
where we make the dependence on $\kappa$ explicit for the constant $C_\kappa$. For the rest of this section, the wiggles sign $\lesssim$ will hide constants independent of $h$ and $\kappa$.
We aim to prove the following theorem.

\begin{theorem}\label{th:fq}
Suppose $k\ge1$ and \eqref{eq:fq_reg} holds. If $(h^2\kappa^2+h\kappa)C_\kappa$ is small enough, then
\begin{align*}
\|\bs u-\bs u_h\|_{\Omega}&\lesssim
\left(
	(1+C_\kappa)h^{m+1} + C_\kappa(\kappa^2+\kappa) h^{m+2}
\right)
\left(
	\|\bs\sigma\|_{m,\Omega}+\|\bs u\|_{m+1,\Omega}
\right),\\
\|\bs\sigma-\bs\sigma_h\|_{\Omega}&\lesssim
\left(
	h^m+(1+C_\kappa)\kappa h^{m+1} + C_\kappa(\kappa^3+\kappa^2)h^{m+2}
\right)
\left(
	\|\bs\sigma\|_{m,\Omega}+\|\bs u\|_{m+1,\Omega}
\right),
\end{align*}
with $1\le m\le k+1$. 
\end{theorem}

\subsection{Proof of Theorem \ref{th:fq}}
Since the solution $(\bs\sigma,\bs u)$ in \eqref{eq:fq_pde} can take complex values, applying directly Theorem \ref{th:PROJ} is not feasible. However, this can be easily fixed by defining a new complex projection based on the original one: for $(\bs\sigma,\bs u)\in H^1(K;\mbb C_\mr{sym}^{3\times3})\times H^1(K;\mbb C^3)$, we define
\begin{align}\label{def:com_proj}
\Pi(\bs\sigma,\bs u):=\Pi(\mr{Re}\,\bs\sigma,\mr{Re}\,\bs u)+i\Pi(\mr{Im}\,\bs\sigma,\mr{Im}\,\bs u).
\end{align}
It is easy to show that this complex projection also satisfies \eqref{eq:tauProp} and \eqref{eq:1.3a} (the only difference is that now the test functions $\bs\theta,\bs w,\bs\mu$ in \eqref{eq:tauProp} can take complex values). 

Similar to the beginning of Section \ref{sec:st_prf}, we first define the projections $(\Pi\bs\sigma,\Pi\bs u)$, the remainder term $\bs\delta$, the error $\bs\varepsilon_h^u$, $\bs\varepsilon_h^\sigma$, $\bs\varepsilon^{\widehat{u}}$, and approximation terms $\bs e_\sigma$, $\bs e_u$.

\begin{proposition}[G\aa rding-type identity]\label{prop:fq_id}
The following energy identity holds
\begin{align*}
\|\bs\varepsilon_h^\sigma\|_{\mc A}^2
+\|\mr P_M(\bs\varepsilon_h^u-\bs\varepsilon_h^{\widehat{u}})\|_{\bs\tau}^2
-\kappa^2\|\bs\varepsilon_h^u\|_\rho^2
&=(\mc A\overline{\bs e_\sigma},\bs\varepsilon_h^\sigma)_{\mc T_h}
+\dualpr{\overline{\bs\delta},\bs\varepsilon_h^u-\bs\varepsilon_h^{\widehat{u}}}_{\pp\mc T_h}
-\kappa^2(\rho\overline{\bs e_u},\bs\varepsilon_h^u)_{\mc T_h}.
\end{align*}
Consequently,
\begin{align}\label{eq:fq_ene}
\|\bs\varepsilon_h^\sigma\|_{\mc T_h} + \|\mr P_M(\bs\varepsilon_h^u-\bs\varepsilon_h^{\widehat{u}})\|_{\bs\tau}\lesssim \kappa\|\bs\varepsilon_h^u\|_{\mc T_h}+\kappa\|\bs e_u\|_{\mc T_h}
+\|\bs e_\sigma\|_{\mc T_h}+\|\bs\delta\|_{\bs\tau^{-1}}.
\end{align}
\end{proposition}
\begin{proof}
By similar ideas used in the proof of Proposition \ref{prop:energy_id}, we obtain the following error equations:
\begin{subequations}\label{eq:fq_erreq}
\begin{alignat}{5}
(\mc A\bs\varepsilon_h^\sigma,\bs\theta)_{\mc T_h} + (\bs\varepsilon_h^u,\div\bs\theta)_{\mc T_h}-\dualpr{\bs\varepsilon_h^{\widehat{u}},\bs\theta\bs n}_{\pp\mc T_h}&=(\mc A\bs e_\sigma,\bs\theta)_{\mc T_h},\\
-(\div\bs\varepsilon_h^\sigma,\bs w)_{\mc T_h} + \dualpr{\bs\tau \mr P_M(\bs\varepsilon_h^u-\bs\varepsilon_h^{\widehat{u}}),\bs w}_{\pp\mc T_h}
-(\kappa^2\rho \bs\varepsilon_h^u,\bs w)_{\mc T_h}
&=\dualpr{\bs \delta,\bs w}_{\pp\mc T_h}-(\kappa^2\rho \bs e_u,\bs w)_{\mc T_h},\\
\dualpr{\bs\varepsilon_h^\sigma\bs n -\bs\tau \mr P_M(\bs\varepsilon_h^u-\bs\varepsilon_h^{\widehat{u}}),\bs\mu}
_{\partial\mc T_h\backslash\Gamma}
&=-\dualpr{\bs\delta,\bs\mu}_{\partial\mc T_h\backslash\Gamma},\\
\dualpr{\bs\varepsilon_h^{\widehat{u}},\bs\mu}_\Gamma &= 0,
\end{alignat}
\end{subequations}
for all $\bs\theta\in\bs V_h$, $\bs w\in\bs W_h$ and $\bs\mu\in\bs M_h$. 
Taking $\bs\theta=\overline{\bs\varepsilon_h^\sigma}$, $\bs w=\overline{\bs\varepsilon_h^u}$ and $\bs\mu=\overline{\bs\varepsilon_h^{\widehat{u}}}$, then adding the equations, we have the energy identity.

Denoting $A^2:=\|\bs\varepsilon_h^\sigma\|_{\mc T_h}^2 + \|\mr P_M(\bs\varepsilon_h^u-\bs\varepsilon_h^{\widehat{u}})\|_{\bs\tau}^2$, and
applying the Cauchy-Schwarz inequality on the identity, we have
\[
A^2\lesssim \kappa^2\|\bs\varepsilon_h^u\|_{\mc T_h}^2 + \kappa^2\|\bs e_u\|_{\mc T_h}^2
+A(\|\bs e_\sigma\|_{\mc T_h}+\|\bs\delta\|_{\bs\tau^{-1}}).
\]
The estimate \eqref{eq:fq_ene} now follows by using Young's inequality.
\end{proof}

\begin{proposition}[Estimate by bootstrapping] \label{prop:fq_est}
Suppose $k\ge1$ and \eqref{eq:fq_reg} holds. If $(h^2\kappa^2+h\kappa)C_\kappa$ is small enough, then
\begin{align*}
\|\bs\varepsilon_h^u\|_{\mc T_h}&\lesssim h(\kappa^2+\kappa)C_\kappa\|\bs e_u\|_{\mc T_h}
+h C_\kappa(\|\bs\delta\|_{\bs\tau^{-1}}+\|\bs e_\sigma\|_{\mc T_h}),\\
\|\bs\varepsilon_h^\sigma\|_{\mc T_h} + \|\mr P_M(\bs\varepsilon_h^u-\bs\varepsilon_h^{\widehat{u}})\|_{\bs\tau}&\lesssim
(1+h\kappa C_\kappa)(\|\bs e_\sigma\|_{\mc T_h}+\|\bs\delta\|_{\bs\tau^{-1}}) + (\kappa+h C_\kappa(\kappa^3+\kappa^2))\|\bs e_u\|_{\mc T_h}.
\end{align*}
\end{proposition}
\begin{proof}
Consider the adjoint equations \eqref{eq:fq_adpde} and take $\bs\Theta = \bs\varepsilon_h^u$.
Following similar ideas in the proof of Proposition \ref{prop:st_dual}, we first apply the projection on $(\bs\Psi,\bs\Phi)$ with stabilization function $-\bs\tau$ to obtain the projection equations (about $(\Pi\bs\Psi,\Pi\bs\Phi,\bs\Delta)$). We next test the projection equations with $\bs\varepsilon_h^\sigma$, $\bs\varepsilon_h^u$ and $\bs\varepsilon_h^{\widehat{u}}$, test the error equations \eqref{eq:fq_erreq} with $\overline{\Pi\bs\Psi}$, $\overline{\Pi\bs\Psi}$ and $\overline{\mr P_M\bs\Phi}$, and compare the two sets of equations. Then we obtain (define $\bs e_\Psi:=\Pi\bs\Psi-\bs\Psi$ and $\bs e_\Phi:=\Pi\bs\Phi-\bs\Phi$ for convenience)
\begin{align*}
&(\mc A\bs e_\sigma,\overline{\Pi\bs\Psi})_{\mc T_h} + \dualpr{\bs\delta,\overline{\Pi\bs\Phi}-\overline{\mr P_M\bs\Phi}}_{\pp\mc T_h}+\kappa^2(\rho(\bs u-\bs u_h),\overline{\Pi\bs\Phi})_{\mc T_h}\\
&\qquad\qquad=(\mc A\overline{\bs e_\Psi}, \bs\varepsilon_h^\sigma)_{\mc T_h}+\kappa^2(\rho\overline{\bs\Phi},\bs\varepsilon_h^u)_{\mc T_h}+\|\bs\varepsilon_h^u\|_{\mc T_h}^2-\dualpr{\overline{\bs\Delta},\bs\varepsilon_h^u-\bs\varepsilon_h^{\widehat{u}}}_{\pp\mc T_h}.
\end{align*}
After rearranging terms, we have
\begin{alignat}{6}\nonumber
\|\bs\varepsilon_h^u\|_{\mc T_h}^2
&=\kappa^2(\overline{\bs e_\Phi},\rho(\bs u-\bs u_h))_{\mc T_h} + \kappa^2(\rho\overline{\bs\Phi},\bs u-\Pi\bs u)_{\mc T_h}
+(\mc A(\bs\sigma_h-\bs\sigma),\overline{\bs e_\Psi})_{\mc T_h} - (\bs e_\sigma,\overline{\bs\varepsilon(\bs\Phi)})_{\mc T_h}\\
\label{eq:6.99}	
&\quad +\dualpr{\bs\delta,\overline{\Pi\bs\Phi}-\overline{\mathrm{P}_M\bs\Phi}}_{\pp\mc T_h}
+\dualpr{\overline{\bs\Delta},\bs\varepsilon_h^u-\bs\varepsilon_h^{\widehat{u}}}_{\pp\mc T_h}.
\end{alignat}
By Theorem \ref{th:PROJ} (with $m=1$), we have
\[
h^{-1}\|\bs e_\Phi\|_{\mc T_h}+\|\bs e_\Psi\|_{\mc T_h}+\|\Pi\bs\Phi-\mr P_M\bs\Phi\|_{\bs\tau}
+\|\bs\Delta\|_{\bs\tau^{-1}}
\lesssim h(\|\bs\Psi\|_{1,\Omega}+\|\bs\Phi\|_{2,\Omega}),
\]
and also
\begin{align*}
\kappa^2(\rho\overline{\bs\Phi},\bs u-\Pi\bs u)_{\mc T_h}
=\kappa^2(\rho\overline{\bs\Phi}-\mr P_0(\rho\overline{\bs\Phi}),\bs u-\Pi\bs u)_{\mc T_h}\lesssim h\kappa^2\|\rho\|_{W^{1,\infty}(\Omega)}\|\bs\Phi\|_{1,\Omega}\|\bs e_u\|_{\mc T_h}.
\end{align*}
By \eqref{eq:rmproj} we have $(\bs e_\sigma,\bs\varepsilon(\overline{\bs\Phi}))_{\mc T_h}
=(\bs e_\sigma,\bs\varepsilon(\overline{\bs\Phi})-\mr P_0\bs\varepsilon(\overline{\bs\Phi}))_{\mc T_h}
\lesssim h\|\bs e_\sigma\|_{\mc T_h}\|\bs\Phi\|_{2,\Omega}$. Now we use \eqref{eq:fq_reg} in \eqref{eq:6.99} to obtain
\begin{align*}
\|\bs\varepsilon_h^u\|_{\mc T_h}
&\lesssim h^2\kappa^2C_\kappa\|\bs u-\bs u_h\|_{\mc T_h}+ h\kappa^2C_\kappa\|\Pi\bs u-\bs u\|_{\mc T_h}
+hC_\kappa\|\bs\sigma-\bs\sigma_h\|_{\mc T_h}+hC_\kappa\|\bs e_\sigma\|_{\mc T_h}\\
&\quad +hC_\kappa\|\bs\delta\|_{\bs\tau^{-1}}+hC_\kappa\|\mr P_M\bs\varepsilon_h^u-\bs\varepsilon_h^{\widehat{u}}\|_{\bs\tau},\\
&\lesssim h^2\kappa^2C_\kappa\|\bs\varepsilon_h^u\|_{\mc T_h}+ h\kappa^2C_\kappa\|\bs e_u\|_{\mc T_h}
+hC_\kappa\|\bs e_\sigma\|_{\mc T_h}\\
&\quad +hC_\kappa\|\bs\delta\|_{\bs\tau^{-1}}
+hC_\kappa(\|\bs\varepsilon_h^\sigma\|_{\mc T_h}+\|\mr P_M\bs\varepsilon_h^u-\bs\varepsilon_h^{\widehat{u}}\|_{\bs\tau}).
\end{align*}
Next we use \eqref{eq:fq_ene} and bound
\begin{align*}
\|\bs\varepsilon_h^u\|_{\mc T_h}\lesssim
(h^2\kappa^2+h\kappa)C_\kappa\|\bs\varepsilon_h^u\|_{\mc T_h}+(h\kappa+h\kappa^2)C_\kappa\|\bs e_u\|_{\mc T_h} + hC_\kappa(\|\bs e_\sigma\|_{\mc T_h}+\|\bs\delta\|_{\bs\tau^{-1}}).
\end{align*}
Therefore, when $(h^2\kappa^2+h\kappa)C_\kappa$ is small enough, we have
\begin{align*}
\|\bs\varepsilon_h^u\|_{\mc T_h}\lesssim
(h\kappa+h\kappa^2)C_\kappa\|\bs e_u\|_{\mc T_h} + hC_\kappa(\|\bs e_\sigma\|_{\mc T_h}+\|\bs\delta\|_{\bs\tau^{-1}}).
\end{align*}
The first inequality is thus proved.
The second inequality can be proved by combining \eqref{eq:fq_ene} and the above inequality. 
\end{proof}

Theorem \ref{th:fq} now follows easily from Proposition \ref{prop:fq_est} and Theorem \ref{th:PROJ}.

\section{Transient elastodynamics}\label{sec:trans_elas}

\subsection{The semi-discrete HDG+ method}
In this section, we present a semi-discrete (in space) HDG+ method for transient elastic waves and prove it is uniformly-in-time optimal convergent. The equations we consider are
\begin{subequations}
\label{eq:elwv_pde}
\begin{alignat}{5}
\mc A\bs\sigma(t) 
- \bs\varepsilon(\bs u(t)) &=0
&\qquad&\text{in}\ \Omega\times[0,T],\\
\rho\ddot{\bs u}(t)
- \div\bs\sigma(t) &= \bs f(t)
&&\text{in}\ \Omega\times[0,T],\\
\gamma\bs u(t)&=\bs g(t)
&&\text{on}\ \Gamma\times[0,T],\\
\bs u(0)&= \bs u_0
&&\text{on}\ \Omega,\\
\dot{\bs u}(0)&= \bs v_0
&&\text{on}\ \Omega,
\end{alignat}
\end{subequations}
where $\bs f\in C([0,\infty);L^2(\Omega;\mbb R^3))$ and $\bs g\in C([0,\infty);H^{1/2}(\Gamma;\mbb R^3))$. For the initial conditions, we assume $\bs u_0,\bs v_0\in L^2(\Omega;\mbb R^3)$.

The HDG+ method for \eqref{eq:elwv_pde} looks for
\begin{align*}
\bs\sigma_h,\bs u_h,\widehat{\bs u}_h:
[0,\infty)\rightarrow \bs V_h\times\bs W_h\times\bs M_h,
\end{align*}
such that for all $t\ge0$
\begin{subequations}
\label{eq:elwv_HDG}
\begin{align}
\label{eq:elwv_HDGa}
(\mc A\bs\sigma_h(t),\bs\theta)_{\mc T_h} + (\bs u_h(t),\div\bs\theta)_{\mc T_h}
-\dualpr{\widehat{\bs u}_h(t),\bs\theta\bs n}_{\pp\mc T_h}
&=0,\\
\label{eq:elwv_HDGb}
(\rho\ddot{\bs u}_h(t),\bs w)_{\mc T_h} - (\div\bs\sigma_h(t),\bs w)_{\mc T_h}
+\dualpr{\bs\tau\mathrm P_M(\bs u_h(t)-\widehat{\bs u}_h(t)),\bs w}_{\pp\mc T_h}
&=(\bs f(t),\bs w)_{\mc T_h},\\
\label{eq:elwv_HDGc}
\dualpr{\bs\sigma_h(t)\bs n-\bs\tau\mr P_M(\bs u_h(t)-\widehat{\bs u}_h(t)),\bs\mu}_{\pp\mc T_h\backslash\Gamma}
&= 0,\\
\label{eq:elwv_HDGd}
\dualpr{\widehat{\bs u}_h(t),\bs\mu}_\Gamma 
&= \dualpr{\bs g(t),\bs\mu}_{\Gamma},
\end{align}
\end{subequations}
for all $\bs\theta\in\bs V_h$, $\bs w\in\bs W_h$ and $\bs\mu\in\bs M_h$. 
For the initial conditions of the method, we use ideas from \cite{CoFuHuJiSaMaSa:2018}. 
The initial velocity $\dot{\bs u}_h(0)$ is defined by using the projection in Theorem \ref{th:PROJ}
\begin{align}\label{eq:elwv_iniv}
(\times,\dot{\bs u}_h(0))&= \Pi(\mc A^{-1}\bs\varepsilon(\bs v_0),\bs v_0;\bs\tau),
\end{align}
and the initial displacement $\bs u_h(0)$ is defined to be the solution of the HDG+ discretization of the steady-state system
\begin{subequations}\label{eq:elwv_pdet0}
\begin{alignat}{5}
\mc A\bs\sigma(0) - \bs\varepsilon(\bs u(0))&=0
&\qquad &\text{in}\ \Omega,\\
-\div\bs\sigma(0) &= 
-\div(\mc A^{-1}\bs\varepsilon(\bs u_0))
&&\text{in}\ \Omega,\\
\bs u(0) &= g(0)
&&\text{on}\ \Gamma.
\end{alignat}
\end{subequations}
Namely, we find $(\bs\sigma_h(0),\bs u_h(0), \widehat{\bs u}_h(0))\in \bs V_h\times\bs W_h\times\bs M_h$ such that
\begin{subequations}
\label{eq:elwv_HDGt0}
\begin{align}
\label{eq:elwv_HDGt0a}
(\mc A\bs\sigma_h(0),\bs\theta)_{\mc T_h} + (\bs u_h(0),\div\bs\theta)_{\mc T_h}
-\dualpr{\widehat{\bs u}_h(0),\bs\theta\bs n}_{\pp\mc T_h}
&=0,\\
\label{eq:elwv_HDGt0b}
- (\div\bs\sigma_h(0),\bs w)_{\mc T_h}
+\dualpr{\bs\tau\mathrm P_M(\bs u_h(0)-\widehat{\bs u}_h(0)),\bs w}_{\pp\mc T_h}
&=(-\div(\mc A^{-1}\bs\varepsilon(\bs u_0)),\bs w)_{\mc T_h},\\
\label{eq:elwv_HDGt0c}
\dualpr{\bs\sigma_h(0)\bs n-\bs\tau\mr P_M(\bs u_h(0)-\widehat{\bs u}_h(0)),\bs\mu}_{\pp\mc T_h\backslash\Gamma}
&= 0,\\
\label{eq:elwv_HDGt0d}
\dualpr{\widehat{\bs u}_h(0),\bs\mu}_\Gamma 
&= \dualpr{\bs g(0),\bs\mu}_{\Gamma},
\end{align}
\end{subequations}
for all $(\bs\theta,\bs w,\bs\mu)\in \bs V_h\times \bs W_h \times\bs M_h$.
For notational convenience, we define the following space-time norms
\begin{align*}
\vvvert\cdot\vvvert_{*,p}^{[0,T]}:=
\left(\int_0^T \|\cdot(t)\|_*^p\,\mathrm dt\right)^{1/p},\quad
\end{align*}
where $*$ can be replaced by $H^m(\Omega)$, $\Omega$, $\mc T_h$, $\mc A$, $\rho$, $\pp\mc T_h$, $\bs\tau$ or $\bs\tau^{-1}$. The parameter $p$ takes values in $\{1,2,\infty\}$, and when $p=\infty$, we consider the supreme norm in time instead of $L^p$ integration. 
Now we define the projections and the remainder terms for all $t\ge0$:
\begin{align*}
(\Pi\bs\sigma(t),\Pi\bs u(t)) 
:= \prod_{K\in\mc T_h}\Pi(\bs\sigma(t)\big|_K,\bs u(t)\big|_K;\bs\tau\big|_{\pp K}),
\quad
\bs\delta(t):=\prod_{K\in\mc T_h}\mathrm R(\bs\sigma(t)\big|_K,\bs u(t)\big|_K;\bs\tau\big|_{\pp K}).
\end{align*}
The related error and approximation terms (for all $t\ge0$) are denoted
\begin{alignat*}{4}
\bs\varepsilon_h^\sigma(t) &:= \Pi\bs\sigma(t)-\bs\sigma_h(t),&\quad
\bs\varepsilon_h^u(t)&:=\Pi\bs u(t)-\bs u_h(t),\quad
\bs\varepsilon_h^{\widehat{u}}(t)
:=\mathrm P_M\bs u(t)-\widehat{\bs u}_h(t),\\
\bs e_\sigma(t)&:=\Pi\bs\sigma(t)-\bs\sigma(t),&
\bs e_u(t)&:=\Pi\bs u(t)-\bs u(t).
\end{alignat*}
For the rest of this section, the wiggles sign $\lesssim$ will hide constants independent of $h$ and $T$.
The main results in this section are Theorems \ref{th:elwv_th1} and \ref{th:elwv_spconv}.

\begin{theorem}\label{th:elwv_th1}
For $k\ge1$, the following estimates hold
\begin{align*}
&\|\bs\varepsilon_h^\sigma(T)\|_{\mc A}
+\|\mathrm P_M(\bs\varepsilon_h^u(T)
	-\bs\varepsilon_h^{\widehat{u}}(T))\|_{\bs\tau}
+\|\dot{\bs\varepsilon}_h^u(T)\|_\rho\\
&\qquad\qquad\lesssim 
\|\bs e_\sigma(0)\|_{\mc A}
+\vvvert\dot{\bs e}_\sigma\vvvert_{\mc A,1}^{[0,T]}
+
\vvvert\ddot{\bs e}_u\vvvert_{\rho,1}^{[0,T]}
+
\vvvert\bs\delta\vvvert_{\bs\tau^{-1},\infty}^{[0,T]}
+\vvvert\dot{\bs\delta}\vvvert_{\bs\tau^{-1},1}^{[0,T]}
,\\
&\|\dot{\bs\varepsilon}_h^\sigma(T)\|_{\mc A}
+\|\mathrm P_M(\dot{\bs\varepsilon}_h^u(T)
	-\dot{\bs\varepsilon}_h^{\widehat{u}}(T))\|_{\bs\tau}
+\|\ddot{\bs\varepsilon}_h^u(T)\|_\rho\\
&\qquad\qquad\lesssim 
\|\dot{\bs e}_\sigma(0)\|_{\mc A}
+\|\ddot{\bs e}_u(0)\|_\rho
+\vvvert\ddot{\bs e}_\sigma\vvvert_{\mc A,1}^{[0,T]}
+\vvvert\dddot{\bs e}_u\vvvert_{\rho,1}^{[0,T]}
+\vvvert\dot{\bs\delta}\vvvert_{\bs\tau^{-1},\infty}^{[0,T]}
+\vvvert\ddot{\bs\delta}\vvvert_{\bs\tau^{-1},1}^{[0,T]}
.
\end{align*}
\end{theorem}

For the next theorem, we need to assume that the elliptic regularity condition
\begin{align}\label{eq:elwv_ellreg}
\|\bs\Phi\|_{2,\Omega}+\|\mc A^{-1}\bs\varepsilon(\bs\Phi)\|_{1,\Omega}\le C_\mr{reg}\|\div(\mc A^{-1}\bs\varepsilon(\bs\Phi))\|,
\end{align}
holds for any $\bs\Phi\in H^1(\Omega;\mbb R^3)$ such that the right hand side of the above inequality is finite. Note that \eqref{eq:elwv_ellreg} is the same as \eqref{eq:ell_reg}. We rephrase it here to have it in the form we will use it.

\begin{theorem}\label{th:elwv_spconv}
If $k\ge1$ and \eqref{eq:elwv_ellreg} holds, then
\begin{align*}
\|\bs\varepsilon_h^u(T)\|_{\Omega}&\lesssim
h (1+T)^2
\bigg(
\|\bs e_\sigma(0)\|_{\Omega}
+\vvvert\dot{\bs e}_\sigma\vvvert_{\Omega,\infty}^{[0,T]}
+\vvvert\ddot{\bs e}_\sigma\vvvert_{\Omega,\infty}^{[0,T]}\\
&\phantom{\lesssim h(1+T)^2\bigg(}+\|\bs\delta(0)\|_{\bs\tau^{-1}}
+\vvvert\dot{\bs\delta}\vvvert_{\bs\tau^{-1},\infty}^{[0,T]}
+\vvvert\ddot{\bs\delta}\vvvert_{\bs\tau^{-1},\infty}^{[0,T]}\bigg)\\
&\quad +(1+T)^2\bigg(\|\bs e_u(0)\|_{\Omega}
+\vvvert\ddot{\bs e}_u\vvvert_{\Omega,\infty}^{[0,T]}
+\vvvert\dddot{\bs e}_u\vvvert_{\Omega,\infty}^{[0,T]}\bigg).
\end{align*}
Therefore, for $1\le m\le k+1$,
\[
\|\bs\varepsilon_h^u(T)\|_{\Omega}\lesssim
h^{m+1} (1+T)^2\sum_{i=0}^3 \bigg(\vvvert\bs\sigma^{(i)}\vvvert_{H^m(\Omega),\infty}^{[0,T]}
+\vvvert\bs u^{(i)}\vvvert_{H^{m+1}(\Omega),\infty}^{[0,T]}\bigg),
\]
\end{theorem}

We next give the proofs for the above two theorems in the following two subsections respectively.

\subsection{Energy estimates}
In this subsection, we give a proof to Theorem \ref{th:elwv_th1}.
We first present two lemmas, which give the error equations when $t\ge0$ and the error equations when $t=0$, respectively.
\begin{lemma}
For all $t\ge0$, the following error equations
\begin{subequations}
\label{eq:elwv_ereq}
\begin{align}
\label{eq:elwv_ereqa}
(\mc A\bs\varepsilon_h^\sigma(t),\bs\theta)_{\mc T_h} 
+ ({\bs\varepsilon_h^u}(t),\div\bs\theta)_{\mc T_h}
-\dualpr{\bs\varepsilon_h^{\widehat u}(t),\bs\theta\bs n}_{\pp\mc T_h}
&=(\mc A\bs e_\sigma(t),\bs\theta)_{\mc T_h},\\
(\rho\ddot{\bs\varepsilon}_h^u(t),\bs w)_{\mc T_h} 
- (\div\bs\varepsilon_h^\sigma(t),\bs w)_{\mc T_h}
+\dualpr{\bs\tau\mathrm P_M({\bs\varepsilon_h^u}(t)
-\bs\varepsilon_h^{\widehat u}(t)),\bs w}_{\pp\mc T_h}
&=(\rho\ddot{\bs e}_u(t),\bs w)_{\mc T_h}\nonumber \\
\label{eq:elwv_ereqb}
&\quad +\dualpr{\bs\delta(t),\bs w}_{\pp\mc T_h},\\
\label{eq:elwv_ereqc}
\dualpr{\bs\varepsilon_h^\sigma(t)\bs n
	-\bs\tau({\bs\varepsilon_h^u}(t)
	-\bs\varepsilon_h^{\widehat u}(t)),\bs\mu}
	_{\pp\mc T_h\backslash\Gamma}
&= -\dualpr{\bs\delta(t),\bs\mu}_{\pp\mc T_h\backslash\Gamma},\\
\label{eq:elwv_ereqd}
\dualpr{\bs\varepsilon_h^{\widehat u}(t),\bs\mu}_\Gamma 
&= 0,\\
\label{eq:elwv_ereqe}
\bs\varepsilon_h^u(0) &= \Pi\bs u_0-{\bs u}_h(0),\\
\label{eq:elwv_ereqf}
\dot{\bs\varepsilon}_h^u(0) &= \bs 0,
\end{align}
\end{subequations}
hold for all $(\bs\theta,\bs w,\bs\mu)\in \bs V_h\times \bs W_h \times\bs M_h$.
\end{lemma}
\begin{proof}
Use \eqref{eq:elwv_pde}, \eqref{eq:elwv_HDG}, \eqref{eq:elwv_iniv}, \eqref{eq:elwv_HDGt0}, and Theorem \ref{th:PROJ}.
\end{proof}

\begin{lemma}
The following error equations
\begin{subequations}
\label{eq:elwv_ereqt0}
\begin{align}
\label{eq:elwv_ereqt0a}
(\mc A\bs\varepsilon_h^\sigma(0),\bs\theta)_{\mc T_h} 
+ ({\bs\varepsilon_h^u}(0),\div\bs\theta)_{\mc T_h}
-\dualpr{\bs\varepsilon_h^{\widehat u}(0),\bs\theta\bs n}_{\pp\mc T_h}
&=(\mc A\bs e_\sigma(0),\bs\theta)_{\mc T_h},\\
\label{eq:elwv_ereqt0b}
- (\div\bs\varepsilon_h^\sigma(0),\bs w)_{\mc T_h}
+\dualpr{\bs\tau\mathrm P_M({\bs\varepsilon_h^u}(0)
	-\bs\varepsilon_h^{\widehat u}(0)),\bs w}_{\pp\mc T_h}
&=\dualpr{\bs\delta(0),\bs w}_{\pp\mc T_h},\\
\label{eq:elwv_ereqt0c}
\dualpr{\bs\varepsilon_h^\sigma(0)\bs n
-\bs\tau({\bs\varepsilon_h^u}(0)
-\bs\varepsilon_h^{\widehat u}(0)),\bs\mu}_{\pp\mc T_h\backslash\Gamma}
&= -\dualpr{\bs\delta(0),\bs\mu}_{\pp\mc T_h\backslash\Gamma},\\
\label{eq:elwv_ereqt0d}
\dualpr{\bs\varepsilon_h^{\widehat u}(0),\bs\mu}_\Gamma 
&= 0,
\end{align}
\end{subequations}
hold for all $(\bs\theta,\bs w,\bs\mu)\in \bs V_h\times \bs W_h \times\bs M_h$.
\end{lemma}
\begin{proof}
Use \eqref{eq:elwv_pdet0}, \eqref{eq:elwv_HDGt0}, and Theorem \ref{th:PROJ}.
\end{proof}

The next proposition gives estimates to the error terms when $t=0$.

\begin{proposition}\label{prop:ene_id_t0}
The following estimates hold:
\begin{align}
\label{eq:prp_enid_t0_a}
\|\bs\varepsilon_h^\sigma(0)\|_{\mc A}^2
+\|\mathrm P_M(\bs\varepsilon_h^u(0)-\bs\varepsilon_h^{\widehat{u}}(0))\|_{\bs\tau}^2
&\le 
\|\bs e_\sigma(0)\|_{\mc A}^2
+\|\bs\delta(0)\|_{\bs\tau^{-1}}^2,\\
\label{eq:prp_enid_t0_b}
\|\dot{\bs\varepsilon}_h^\sigma(0)\|_{\mc A}^2
+\|\mathrm P_M(\dot{\bs\varepsilon}_h^u(0)-\dot{\bs\varepsilon}_h^{\widehat{u}}(0))\|_{\bs\tau}^2
&\le 
\|\dot{\bs e}_\sigma(0)\|_{\mc A}^2
+\|\dot{\bs\delta}(0)\|_{\bs\tau^{-1}}^2,\\
\label{eq:prp_enid_t0_c}
\|\ddot{\bs\varepsilon}_h^u(0)\|_\rho
&\le \|\ddot{\bs e}_u(0)\|_\rho.
\end{align}
\end{proposition}
\begin{proof}
Taking $\bs\theta = \bs\varepsilon_h^\sigma(0)$, $\bs w=\bs\varepsilon_h^u(0)$ and $\bs\mu = \bs\varepsilon_h^{\widehat u}(0)$
in the error equations \eqref{eq:elwv_ereqt0} and adding the equations, we have
\begin{align*}
\|\bs\varepsilon_h^\sigma(0)\|_{\mc A}^2
+\|\mathrm P_M(\bs\varepsilon_h^u(0)-\bs\varepsilon_h^{\widehat u}(0))\|_{\bs\tau}^2
=(\mc A\bs e_\sigma(0),\bs\varepsilon_h^\sigma(0))_{\mc T_h}
+\dualpr{\bs\delta(0),\bs\varepsilon_h^u(0) 
	- \bs\varepsilon_h^{\widehat u}(0)}_{\pp\mc T_h}.
\end{align*}
The first estimate \eqref{eq:prp_enid_t0_a} then follows from the latter identity.

Consider the error equations \eqref{eq:elwv_ereq}. We 
take the first order derivative of \eqref{eq:elwv_ereqa} - \eqref{eq:elwv_ereqc} and test the equations with
${\bs\theta} = \dot{\bs\varepsilon}_h^\sigma$, 
$\bs w=\dot{\bs\varepsilon}_h^u$,  
$\bs\mu = \dot{\bs\varepsilon}_h^{\widehat u}$. We next add the equations, evaluate them at $t=0$, and use the fact that $\dot{\bs\varepsilon}_h^u(0)=0$. Then
\begin{align*}
\|\dot{\bs\varepsilon}_h^\sigma(0)\|_{\mc A}^2
+\|\mathrm P_M(\dot{\bs\varepsilon}_h^u(0)
	-\dot{\bs\varepsilon}_h^{\widehat u}(0))\|_{\bs\tau}^2
=(\mc A\dot{\bs e}_\sigma(0),\dot{\bs\varepsilon}_h^\sigma(0))_{\mc T_h}
+\dualpr{\dot{\bs\delta}(0),\dot{\bs\varepsilon}_h^u (0)
	- \dot{\bs\varepsilon}_h^{\widehat u}(0)}_{\pp\mc T_h},
\end{align*}
from which the second estimate \eqref{eq:prp_enid_t0_b} follows.

Finally, taking $t=0$ in \eqref{eq:elwv_ereqb} and subtracting \eqref{eq:elwv_ereqt0b}, then taking $\bs w=\ddot{\bs\varepsilon}_h^u(0)$, we have
\begin{align*}
(\ddot{\bs\varepsilon}_h^u(0),\ddot{\bs\varepsilon}_h^u(0))_\rho
=(\ddot{\bs e}_u(0),\ddot{\bs\varepsilon}_h^u(0))_\rho,
\end{align*}
which implies \eqref{eq:prp_enid_t0_c}.
\end{proof}

Note that the boundary remainder operator 
$\mathrm R$ defined in Theorem \ref{th:PROJ} is linear, thus commuting with the time derivative
\begin{align*}
\dot{\bs\delta}(t) = 
\mathrm R(\dot{\bs\sigma}(t),\dot{\bs u}(t);\bs\tau).
\end{align*}
This commutativity holds for the projection $\Pi$ as well for similar reasons.

\begin{proposition}\label{prop:ene_id_tge0}
For $t\ge0$, we have
\begin{align}
&\frac{1}{2}\frac{\mathrm d}{\mathrm dt}
\left(
\|\bs\varepsilon_h^\sigma(t)\|_{\mc A}^2
+\|\mathrm P_M(\bs\varepsilon_h^u(t)
	-\bs\varepsilon_h^{\widehat{u}}(t))\|_{\bs\tau}^2
+\|\dot{\bs\varepsilon}_h^u(t)\|_\rho^2
\right)\nonumber\\
\label{eq:prp_id_tge0_a}
&\qquad\qquad =(\rho\ddot{\bs e}_u(t),\dot{\bs\varepsilon}_h^u(t))
+(\mc A\dot{\bs e}_\sigma(t),\bs\varepsilon_h^\sigma(t))
+\dualpr{\bs\delta(t),\dot{\bs\varepsilon}_h^u(t)
	-\dot{\bs\varepsilon}_h^{\widehat u}(t)},\\
&\frac{1}{2}\frac{\mathrm d}{\mathrm dt}
\left(
\|\dot{\bs\varepsilon}_h^\sigma(t)\|_{\mc A}^2
+\|\mathrm P_M(\dot{\bs\varepsilon}_h^u(t)
	-\dot{\bs\varepsilon}_h^{\widehat{u}}(t))\|_{\bs\tau}^2
+\|\ddot{\bs\varepsilon}_h^u(t)\|_\rho^2
\right)\nonumber\\
\label{eq:prp_id_tge0_b}
&\qquad\qquad=(\rho\dddot{\bs e}_u(t),\ddot{\bs\varepsilon}_h^u(t))
+(\mc A\ddot{\bs e}_\sigma(t),\dot{\bs\varepsilon}_h^\sigma(t))
+\dualpr{\dot{\bs\delta}(t),\ddot{\bs\varepsilon}_h^u(t)
	-\ddot{\bs\varepsilon}_h^{\widehat u}(t)}.
\end{align}
\end{proposition}

\begin{proof}
Taking the first order derivative of \eqref{eq:elwv_ereqa} and testing it with $\bs\theta=\bs\varepsilon_h^\sigma(t)$, then choosing $\bs w=\dot{\bs\varepsilon}_h^u(t)$ in \eqref{eq:elwv_ereqb} and $\bs\mu=\dot{\bs\varepsilon}_h^{\widehat{u}}(t)$ in \eqref{eq:elwv_ereqc}, and finally taking the first order derivative of \eqref{eq:elwv_ereqd} then adding the equations, we obtain \eqref{eq:prp_id_tge0_a}.

Taking the second order derivative of \eqref{eq:elwv_ereqa} and testing it with $\bs\theta=
\dot{\bs\varepsilon}_h^\sigma(t)$, then taking the first order derivative of
\eqref{eq:elwv_ereqb} and testing it with $\bs w=\ddot{\bs\varepsilon}_h^u(t)$, taking the first order derivative of \eqref{eq:elwv_ereqc} and testing it with $\bs\mu = \ddot{\bs\varepsilon}_h^{\widehat{u}}(t)$, and finally taking the second order derivative of \eqref{eq:elwv_ereqd} and adding the equations, we obtain \eqref{eq:prp_id_tge0_b}.
\end{proof}

The final ingredient we need is a Gr\"{o}nwall type inequality.

\begin{lemma}\label{lm:int_ineq}
Suppose $\phi,\beta,l$ are continuous and positive functions defined on $[0,\infty)$ and $r$ is a constant. If
\begin{align*}
\phi^2(t)\le r + 2\int_0^t\phi(s)\beta(s)\mathrm d s
+\phi(t)l(t),
\end{align*}
then
\begin{align*}
\phi^2(t)\le 2 r
+\left(
2\int_0^t\beta(s)\mathrm ds
+\sup_{s\in[0,t]}l(s)
\right)^2.
\end{align*}
\end{lemma}
\begin{proof}
Consider the interval $[0,t]$ and let $\phi(t^*)$ maximize $\phi$ in the interval. Then we have
\begin{align*}
\phi^2(t^*)
&\le r+\phi(t^*)\left(2\int_0^t\beta(s)\mathrm ds + \sup_{s\in[0,t]}l(s)\right)\\
&\le r + \frac{1}{2}
\left(
\phi^2(t^*)+\left(2\int_0^t\beta(s)\mathrm ds + \sup_{s\in[0,t]}l(s)\right)^2
\right).
\end{align*}
\end{proof}

With Proposition \ref{prop:ene_id_t0}, Proposition \ref{prop:ene_id_tge0} and Lemma \ref{lm:int_ineq}, we are ready to prove Theorem \ref{th:elwv_th1}.

\begin{proof}[Proof of Theorem \ref{th:elwv_th1}]
Integrating \eqref{eq:prp_id_tge0_a}
from $0$ to $T$, we have
\begin{align*}
\frac{1}{2}
\bigg(&
\|\bs\varepsilon_h^\sigma(t)\|_{\mc A}^2
+\|\mathrm P_M(\bs\varepsilon_h^u(t)
	-\bs\varepsilon_h^{\widehat{u}}(t))\|_{\bs\tau}^2
+\|\dot{\bs\varepsilon}_h^u(t)\|_\rho^2
\bigg)\bigg|_{t=0}^{t=T}\\
=&
\dualpr{\bs\delta(t),\mathrm P_M(\bs\varepsilon_h^u(t)
	-\bs\varepsilon_h^{\widehat{u}}(t))}_{\pp\mc T_h}\,\bigg|_{t=0}^{t=T}\\
&+\int_0^T \left(
\left(\rho\ddot{\bs e}_u(t),\dot{\bs\varepsilon}_h^u(t)\right)_{\mc T_h}
+\left(\mc A\dot{\bs e}_\sigma(t),\bs\varepsilon_h^\sigma(t)\right)_{\mc T_h}
-\dualpr{\dot{\bs\delta}(t),\mathrm P_M(\bs\varepsilon_h^u(t)-\bs\varepsilon_h^{\widehat{u}}(t))}_{\pp\mc T_h}
\right)\mathrm dt\\
	\le & 
\|\bs\delta(T)\|_{\bs\tau^{-1}}
\|\mathrm P_M(\bs\varepsilon_h^u(T)
	-\bs\varepsilon_h^{\widehat{u}}(T))\|_{\bs\tau}
+\|\bs\delta(0)\|_{\bs\tau^{-1}}
\left(\|\bs e_\sigma(0)\|_{\mc A}^2
	+\|\bs\delta(0)\|_{\bs\tau^{-1}}^2\right)^{1/2}\\
& +
\int_0^T
\left(\|\dot{\bs\varepsilon}_h^u(t)\|_\rho^2
+\|\bs\varepsilon_h^\sigma(t)\|_{\mc A}^2
+\|\mathrm P_M(\bs\varepsilon_h^u(t)
	-\bs\varepsilon_h^{\widehat{u}}(t))\|_{\bs\tau}^2
\right)^{1/2}\\
& \qquad \left(
\|\ddot{\bs e}_u(t)\|_\rho^2+\|\dot{\bs e}_\sigma(t)\|_{\mc A}^2
+\|\dot{\bs\delta}(t)\|_{\bs\tau^{-1}}^2 
\right)^{1/2}\mathrm dt,
\end{align*}
where we used \eqref{eq:prp_enid_t0_a} to estimate $\|\mr P_M(\bs\varepsilon_h^u(0)
-\bs\varepsilon_h^{\widehat{u}}(0))\|_{\bs\tau}$ in the last step.

Now we define
\begin{align*}
\phi^2(t)&:=\|\dot{\bs\varepsilon}_h^u(t)\|_\rho^2+
\|\bs\varepsilon_h^\sigma(t)\|_{\mc A}^2
+\|\mathrm P_M(\bs\varepsilon_h^u(t)
	-\bs\varepsilon_h^{\widehat{u}}(t))\|_{\bs\tau}^2,\\
\beta^2(t)&:=
\|\ddot{\bs e}_u(t)\|_\rho^2+\|\dot{\bs e}_\sigma(t)\|_{\mc A}^2
+\|\dot{\bs\delta}(t)\|_{\bs\tau^{-1}}^2,\\
l(t)&:=2\|\bs\delta(t)\|_{\bs\tau^{-1}},\\
r&:=2\|\bs e_\sigma(0)\|_{\mc A}^2+3\|\bs\delta(0)\|_{\bs\tau^{-1}}^2.
\end{align*}
Note that by \eqref{eq:prp_enid_t0_a} and \eqref{eq:elwv_ereqf}, we have
$
\phi^2(0)
\le \|\bs e_\sigma(0)\|_{\mc A}^2
	+\|\bs\delta(0)\|_{\bs\tau^{-1}}^2,
$
and therefore,
\begin{align*}
\phi^2(T)
\le r+ 2 \int_0^T \phi(t)\beta(t)\mr{d}t +\phi(T)l(T).
\end{align*}
By Lemma \ref{lm:int_ineq}, we have
\begin{align*}
\phi^2(T)&\le 2 r
+\left(
2\int_0^T\beta(t)\mathrm dt
+\sup_{t\in[0,T]}l(t)
\right)^2
\lesssim r +\sup_{t\in[0,T]}l^2(t)+ \left(\int_0^T\beta(t)\mr{d}t\right)^2\\
&\lesssim \|\bs\delta(0)\|_{\bs\tau^{-1}}^2 + \|\bs e_\sigma(0)\|_{\mc A}^2
+\left(\vvvert\bs\delta\vvvert_{\bs\tau^{-1},\infty}^{[0,T]}\right)^2\\
&\quad\,+\left(\int_0^T\left(
\|\ddot{\bs e}_u(t)\|_\rho+\|\dot{\bs e}_\sigma(t)\|_{\mc A}
+\|\dot{\bs\delta}(t)\|_{\bs\tau^{-1}}\right)
\mathrm{d}t\right)^2\\
&\lesssim \left(\|\bs e_\sigma(0)\|_{\mc A}
+\vvvert\bs\delta\vvvert_{\bs\tau^{-1},\infty}^{[0,T]}
+\vvvert\ddot{\bs e}_u\vvvert_{\rho,1}^{[0,T]}
+ \vvvert\dot{\bs e}_\sigma\vvvert_{\mc A,1}^{[0,T]}+\vvvert\dot{\bs\delta}\vvvert_{\bs\tau^{-1},1}^{[0,T]}\right)^2.
\end{align*}

Similarly, for the second estimate, we
integrate \eqref{eq:prp_id_tge0_b} from $0$ to $T$, then use \eqref{eq:prp_enid_t0_b} to estimate $\|\mathrm P_M(\dot{\bs\varepsilon}_h^u(0)
	-\dot{\bs\varepsilon}_h^{\widehat{u}}(0))\|_{\bs\tau}$, and obtain
\begin{align*}
&\frac{1}{2}
\left(
\|\dot{\bs\varepsilon}_h^\sigma(t)\|_{\mc A}^2
+\|\mathrm P_M(\dot{\bs\varepsilon}_h^u(t)
	-\dot{\bs\varepsilon}_h^{\widehat{u}}(t))\|_{\bs\tau}^2
+\|\ddot{\bs\varepsilon}_h^u(t)\|_\rho^2
\right)\bigg|_{t=0}^{t=T}\\
&\le 
\|\dot{\bs\delta}(T)\|_{\bs\tau^{-1}}
\|\mathrm P_M(\dot{\bs\varepsilon}_h^u(T)
	-\dot{\bs\varepsilon}_h^{\widehat{u}}(T))\|_{\bs\tau}
+\|\dot{\bs\delta}(0)\|_{\bs\tau^{-1}}
\left(\|\dot{\bs e}_\sigma(0)\|_{\mc A}^2
	+\|\dot{\bs\delta}(0)\|_{\bs\tau^{-1}}^2\right)^{1/2}\\
&\quad\, +
\int_0^T\hspace{-4pt}
\left(\|\dot{\bs\varepsilon}_h^\sigma(t)\|_{\mc A}^2
\!+\!\|\mathrm P_M(\dot{\bs\varepsilon}_h^u(t)
	\!-\!\dot{\bs\varepsilon}_h^{\widehat{u}}(t))\|_{\bs\tau}^2
\!+\!\|\ddot{\bs\varepsilon}_h^u(t)\|_\rho^2\right)^{\frac12}
\left(
\|\dddot{\bs e}_u(t)\|_\rho^2\!+\!\|\ddot{\bs e}_\sigma(t)\|_{\mc A}^2
\!+\!\|\ddot{\bs\delta}(t)\|_{\bs\tau^{-1}}^2
\right)^{\frac12}\hspace{-4pt}\mathrm dt.
\end{align*}

Now we define
\begin{align*}
\phi^2(t)&:=
\|\dot{\bs\varepsilon}_h^\sigma(t)\|_{\mc A}^2
+\|\mathrm P_M(\dot{\bs\varepsilon}_h^u(t)
	-\dot{\bs\varepsilon}_h^{\widehat{u}}(t))\|_{\bs\tau}^2
+\|\ddot{\bs\varepsilon}_h^u(t)\|_\rho^2,\\
\beta^2(t)&:=
\|\dddot{\bs e}_u(t)\|_\rho^2+\|\ddot{\bs e}_\sigma(t)\|_{\mc A}^2
+\|\ddot{\bs\delta}(t)\|_{\bs\tau^{-1}}^2,\\
l(t)&:=2\|\dot{\bs\delta}(t)\|_{\bs\tau^{-1}},\\
r&:=2\|\dot{\bs e}_\sigma(0)\|_{\mc A}^2
+3\|\dot{\bs\delta}(0)\|_{\bs\tau^{-1}}^2
+\|\ddot{\bs e}_u(0)\|_{\rho}^2.
\end{align*}
Since by \eqref{eq:prp_enid_t0_b} and \eqref{eq:prp_enid_t0_c} we have
$
\phi^2(0)
\le 
\|\dot{\bs e}_\sigma(0)\|_{\mc A}^2
+\|\dot{\bs\delta}(0)\|_{\bs\tau^{-1}}^2
+\|\ddot{\bs e}_u(0)\|_\rho^2,
$
and it follows that
\begin{align*}
\phi^2(T)
\le r+ 2 \int_0^T \phi(t)\beta(t)\mr{d}t +\phi(T)l(T).
\end{align*}
Using Lemma \ref{lm:int_ineq} again we obtain the second estimate.
\end{proof}

\subsection{Duality argument}
In this subsection, we give a proof to Theorem \ref{th:elwv_spconv} by using the duality argument. To begin with, we consider the adjoint equations of \eqref{eq:elwv_pde}:
\begin{subequations}\label{eq:elwv_adpde}
\begin{alignat}{5}\label{eq:elwv_adpdea}
\mc A\bs\Psi + \bs\varepsilon(\bs\Phi) &= 0&\qquad& \mr{in}\ \Omega\times[0,T],\\
\label{eq:elwv_adpdeb}
\rho\ddot{\bs\Phi}+\div\bs\Psi &= 0 && \mr{in}\ \Omega\times[0,T],\\
\label{eq:elwv_adpdec}
\gamma\bs\Phi &= 0 && \mr{on}\ \Gamma\times[0,T],\\
\label{eq:elwv_adpded}
\bs\Phi(T) &= 0 && \mr{on}\ \Omega,\\
\label{eq:elwv_adpdee}
\dot{\bs\Phi}(T) &= \bs\varepsilon_h^u(T) &&\mr{on}\ \Omega.
\end{alignat}
\end{subequations}

For a time-dependent function $f:[0,\infty)\rightarrow X$, we write
\begin{align*}
\underline{f}(t):=\int_t^T f(s)\mr{d}s.
\end{align*}
The following proposition allows us to control the solution of \eqref{eq:elwv_adpde} in certain energy norms. Similar results can be found in \cite{CoQu:2014}.
\begin{proposition}\label{prop:elwv_reg}
The following inequality holds
\begin{align}\label{eq:elwv_reg}
\vvvert\bs\Phi\vvvert_{H^1(\Omega),\infty}^{[0,T]}+\vvvert\dot{\bs\Phi}\vvvert_{\Omega,\infty}^{[0,T]}
\le C
\|\bs\varepsilon_h^u(T)\|_\Omega.
\end{align}
If \eqref{eq:elwv_ellreg} holds, then
\begin{align}\label{eq:elwv_mrreg}
\vvvert\underline{\bs\Psi}\vvvert_{H^1(\Omega),\infty}^{[0,T]}
+\vvvert\underline{\bs\Phi}\vvvert_{H^2(\Omega),\infty}^{[0,T]}
\le C\|\bs\varepsilon_h^u(T)\|_\Omega.
\end{align}
\end{proposition}

\begin{proof}
By conservation of energy we have
\begin{align*}
\|\dot{\bs\Phi}(t)\|_\rho^2 + \|\bs\varepsilon(\bs\Phi)(t)\|_{\mc A^{-1}}^2 = \|\bs\varepsilon_h^u(T)\|_\rho^2,
\end{align*}
for all $t\in[0,T]$. Now \eqref{eq:elwv_reg} follows by using Korn's second inequality.

Integrating \eqref{eq:elwv_adpdea} and \eqref{eq:elwv_adpdeb} from $t$ to $T$, we have
\begin{align*}
\mc A\underline{\bs\Psi}(t) + \bs\varepsilon(\underline{\bs\Phi}(t))=0,\qquad
\div\underline{\bs\Psi}(t)=\rho\dot{\bs\Phi}(t)-\rho\dot{\bs\Phi}(T),
\end{align*}
for all $t\in[0,T]$. Combining the latter equations with \eqref{eq:elwv_ellreg} and \eqref{eq:elwv_reg}, we obtain \eqref{eq:elwv_mrreg}.
\end{proof}

Now we define the dual projections and boundary remainder terms for the adjoint problem \eqref{eq:elwv_adpde} (for all $t\in[0,T]$)
\begin{align*}
(\Pi\bs\Psi(t),\Pi\bs\Phi(t)) 
&:= \prod_{K\in\mc T_h}\Pi(\bs\Psi(t)\big|_K,\bs\Phi(t)\big|_K;-\bs\tau\big|_{\pp K}),\\
\bs\Delta(t)&:=\prod_{K\in\mc T_h}\mathrm R(\bs\Psi(t)\big|_K,\bs\Phi(t)\big|_K;-\bs\tau\big|_{\pp K}).
\end{align*}
We also define the corresponding approximation terms
\begin{align*}
\bs e_\Phi(t) := \Pi\bs\Phi(t)-\bs\Phi(t),\qquad
\bs e_\Psi(t) := \Pi\bs\Psi(t) - \bs\Psi(t).
\end{align*}

\begin{proposition}\label{prop:wav_dual_id}
The following identity holds
\begin{align*}
\|\bs\varepsilon_h^u(T)\|_{\rho}^2 = 
\sum_{i=1}^7 T_i,
\end{align*}
where
\begin{alignat*}{5}
T_1&:=(\rho\dot{\bs\Phi}(0),\Pi\bs u_0-\bs u_h(0))_{\mc T_h}, &\qquad\quad
T_2&:=-\int_0^T\dualpr{\bs\delta(t),\bs e_\Phi(t)}_{\pp\mc T_h}\mathrm dt,\\
T_3&:=-\int_0^T\dualpr{\bs\Delta(t),\bs\varepsilon_h^u(t)-\bs\varepsilon_h^{\widehat u}(t)}_{\pp\mc T_h}\mathrm dt, &
T_4&:=\int_0^T(\mc A\bs e_\Psi(t),\bs\sigma(t)-\bs\sigma_h(t))_{\mc T_h}\mathrm dt,\\
T_5&:=-\int_0^T(\mc A\bs e_\sigma(t),\bs\Psi(t))_{\mc T_h}\mathrm dt,&
T_6&:=-\int_0^T(\rho\bs\Phi(t),\ddot{\bs e}_u(t))_{\mc T_h}\mathrm dt,\\
T_7&:=\int_0^T(\rho(\ddot{\bs u}(t)-\ddot{\bs u}_h(t)),\bs e_\Phi(t))_{\mc T_h}\mathrm dt.
\end{alignat*}
\end{proposition}
\begin{proof}
From the adjoint equations \eqref{eq:elwv_adpde} tested with $(\bs\varepsilon_h^\sigma(t),\bs\varepsilon_h^u(t),\bs\varepsilon_h^{\widehat u}(t))$, and from the properties of the adjoint projection (see Theorem \ref{th:PROJ}), we obtain 
\begin{align*}
(\mc A\Pi\bs\Psi(t),\bs\varepsilon_h^\sigma(t))_{\mc T_h} 
- (\Pi\bs\Phi(t),\div\bs\varepsilon_h^\sigma(t))_{\mc T_h} 
+ \dualpr{\mathrm P_M\bs\Phi(t),\bs\varepsilon_h^\sigma(t)\bs n}_{\pp\mc T_h}&
=(\mc A\bs e_\Psi(t),\bs\varepsilon_h^\sigma(t))_{\mc T_h},\\
(\div\Pi\bs\Psi(t),\bs\varepsilon_h^u(t))_{\mc T_h} 
+\dualpr{\bs\tau\mathrm P_M(\Pi\bs\Phi(t)-\bs\Phi(t)),\bs\varepsilon_h^u(t)}_{\pp\mc T_h} &= 
-\dualpr{\bs\Delta(t),\bs\varepsilon_h^u(t)}_{\pp\mc T_h}\\
& \phantom{=}-(\rho\ddot{\bs\Phi}(t),\bs\varepsilon_h^u(t))_{\mc T_h},\\
-\dualpr{\Pi\bs\Psi(t)\bs n+\bs\tau(\Pi\bs\Phi(t)-\bs\Phi(t)),\bs\varepsilon_h^{\widehat u}(t)}
_{\pp\mc T_h\backslash\Gamma}&=
\dualpr{\bs\Delta,\bs\varepsilon_h^{\widehat u}(t)}_{\pp\mc T_h\backslash\Gamma},\\
\dualpr{\mathrm P_M\bs\Phi(t),
	\bs\varepsilon_h^\sigma(t)\bs n
	-\bs\tau\mr P_M(\bs\varepsilon_h^u(t)-\bs\varepsilon_h^{\widehat u}(t))
	}_\Gamma&=0.
\end{align*}
(Note that $\bs\Phi\equiv 0$ on $\Gamma$.) 
Taking now $\bs\theta = \Pi\bs\Psi(t)$, $\bs w=\Pi\bs\Phi(t)$, $\bs\mu=\mathrm P_M\bs\Phi(t)$ and $\bs\mu=\Pi\bs\Psi(t)\bs n+\bs\tau\mr P_M(\Pi\bs\Phi(t)-\bs\Phi(t))$ in the error equations \eqref{eq:elwv_ereqa} to \eqref{eq:elwv_ereqd},  and comparing the two sets of equations, we have
\begin{align*}
&-\dualpr{\bs\Delta(t),\bs\varepsilon_h^u(t)-\bs\varepsilon_h^{\widehat u}(t)}_{\pp\mc T_h}
+(\mc A\bs e_\Psi(t),\bs\varepsilon_h^\sigma(t)) _{\mc T_h}
- (\rho\ddot{\bs\Phi}(t),\bs\varepsilon_h^u(t))_{\mc T_h}\\
&\qquad\qquad=(\mc A\bs e_\sigma(t),\Pi\bs\Psi(t))_{\mc T_h}
+\dualpr{\bs\delta(t),\Pi\bs\Phi(t)-\bs\Phi(t)}_{\pp\mc T_h}
+(\rho(\ddot{\bs u}_h(t)-\ddot{\bs u}(t)),\Pi\bs\Phi(t))_{\mc T_h}.
\end{align*}
After rearranging terms, we have
\begin{align}\label{eq:7.200}
\nonumber
(\rho\ddot{\bs\Phi}(t),\bs\varepsilon_h^u(t))_{\mc T_h}=
&-\dualpr{\bs\Delta(t),\bs\varepsilon_h^u(t)
-\bs\varepsilon_h^{\widehat u}(t)}_{\pp\mc T_h}
-\dualpr{\bs\delta(t),\Pi\bs\Phi(t)-\bs\Phi(t)}_{\pp\mc T_h}\\
&-(\rho(\ddot{\bs u}_h(t)-\ddot{\bs u}(t)),\Pi\bs\Phi)_{\mc T_h}
+(\mc A\bs e_\Psi(t),\bs\sigma(t)-\bs\sigma_h(t))_{\mc T_h}\\
\nonumber
&-(\mc A\bs e_\sigma(t),\bs\Psi(t))_{\mc T_h}.
\end{align}

Defining now
\begin{align*}
\eta(t):=
(\rho\dot{\bs\Phi}(t),\bs\varepsilon_h^u(t))_{\mc T_h}
-(\rho\bs\Phi(t),\dot{\bs\varepsilon}_h^u(t))_{\mc T_h},
\end{align*}
which satisfies
\begin{align*}
\eta(T) = (\rho\bs\varepsilon_h^u(T),\bs\varepsilon_h^u(T))_{\mc T_h},\quad
\eta(0) = (\rho\dot{\bs\Phi}(0),\Pi\bs u_0-\bs u_h(0))_{\mc T_h},
\end{align*}
due to 
$\bs\Phi(T)=0$, $\dot{\bs\Phi}(T)=\bs\varepsilon_h^u(T)$, $\bs\varepsilon_h^u(0)=\Pi\bs u_0-\bs u_h(0)$ and $\dot{\bs\varepsilon}_h^u(0)=0$ (see \eqref{eq:elwv_ereq} and \eqref{eq:elwv_adpde}).
By \eqref{eq:7.200}, we have
\begin{align*}
\dot{\eta}(t) =& (\rho\ddot{\bs\Phi}(t),\bs\varepsilon_h^u(t))_{\mc T_h}
-(\rho\bs\Phi(t),\ddot{\bs\varepsilon}_h^u(t))_{\mc T_h}\\
=&
-\dualpr{\bs\Delta(t),\bs\varepsilon_h^u(t)-\bs\varepsilon_h^{\widehat u}(t)}_{\pp\mc T_h}
-\dualpr{\bs\delta(t),\bs e_\Phi(t)}_{\pp\mc T_h}\\
&+(\mc A\bs e_\Psi(t),\bs\sigma(t)-\bs\sigma_h(t))_{\mc T_h}
-(\mc A\bs e_\sigma(t),\bs\Psi(t))_{\mc T_h}\\
&+(\rho(\ddot{\bs u}(t)-\ddot{\bs u}_h(t)),\bs e_\Phi(t))_{\mc T_h}
-(\rho\bs\Phi(t),\ddot{\bs e}_u(t))_{\mc T_h}.
\end{align*}
Therefore
\begin{align*}
\|\rho^{1/2}\bs\varepsilon_h^u(T)\|_{\mc T_h}^2 = \int_0^T\dot{\eta}(t)\mathrm dt+
(\rho\dot{\bs\Phi}(0),\Pi\bs u_0-\bs u_h(0))_{\mc T_h},
\end{align*}
and the proof is completed.
\end{proof}

%\noindent{\bf A weighted projection.}
%We define
%\begin{align*}
%\mr P_{k+1}^\rho:L^2(\Omega;\mbb R^3)
%&\rightarrow \bs W_h,\\
%\bs u&\mapsto\mr P_{k+1}^\rho\bs u,
%\end{align*}
%by solving the following equation on each element $K$:
%\begin{align*}
%(\mr P_{k+1}^\rho\bs u,\rho\bs v)_K = (\bs u,\rho\bs v)_K\qquad \forall \bs v\in\mc P_{k+1}(K;\mbb R^3).
%\end{align*}
%Since $\rho\in W^{1,\infty}(\Omega;\mbb R^3)$ and bounded below by $\rho_0>0$, it is easy to see that $\mr P_{k+1}^\rho$ is an orthogonal projection in the weighted inner product $(\cdot,\cdot)_{\rho,K}:=(\rho\,\cdot,\cdot)_K$. Hence it
%is well defined, and for all $K\in\mc T_h$ we have
%\begin{itemize}
%\item $\mr P_{k+1}^\rho\bs u =\bs u$ for all $\bs u\in\mc P_{k+1}(K;\mbb R^3)$.
%\item $\|\mr P_{k+1}^\rho\bs u -\bs u\|_K\le Ch_K^2|\bs u|_{2,K}$ for all $\bs u\in H^2(\Omega;\mbb R^3)$. Here $C$ depends only on $k$, the shape-regularity of $K$ and the density function $\rho$.
%\end{itemize}

%In order to see how the time $T$ affects the estimates, for the rest of this subsection, we redefine the meaning of the wiggled sign `$\lesssim$' to hide a constant independent of $h$ and $T$.
The next proposition gives an estimate for the term $T_7$ in Proposition \ref{prop:wav_dual_id}.

\begin{proposition}\label{prop:elwv_tricky}
If \eqref{eq:elwv_ellreg} holds, then
\begin{align*}
\left|T_7\right|
&\lesssim h
\|\bs\varepsilon_h^u(T)\|_{\mc T_h}
\bigg(
\|\bs\varepsilon_h^\sigma(0)\|_{\mc T_h}
+T\vvvert\dot{\bs\varepsilon}_h^{\sigma}\vvvert_{\mc T_h,\infty}^{[0,T]}\\
&\quad\, 
+\|\mr P_M\bs\varepsilon_h^u(0)-\bs\varepsilon_h^{\widehat{u}}(0)\|_{\bs\tau}
+T\vvvert\mr P_M\dot{\bs\varepsilon}_h^u-\dot{\bs\varepsilon}_h^{\widehat{u}}\vvvert_{\bs\tau,\infty}^{[0,T]}\\
&\quad\, 
+h\|\ddot{\bs e}_u(0)\|_{\mc T_h}
+hT\vvvert\dddot{\bs e}_u\vvvert_{\mc T_h,\infty}^{[0,T]}
+\|\bs\delta(0)\|_{\bs\tau^{-1}}+T\vvvert\dot{\bs\delta}\vvvert_{\bs\tau^{-1},\infty}^{[0,T]}
\bigg).
\end{align*}
\end{proposition}

\begin{proof}
In the coming arguments, for the sake of shortening some estimates, we
will prove bounds in terms of the quantity
\begin{align}\label{eq:def_theta}
\bs\Theta(T):=\sup_{t\in[0,T]}|\underline{\bs\Phi}(t)|_{2,\Omega}
+\sup_{t\in[0,T]}|\underline{\bs\Psi}(t)|_{1,\Omega},
\end{align}
which we have shown in Proposition \ref{prop:elwv_reg} that, assuming  \eqref{eq:elwv_ellreg}, we have the estimate
\begin{align}\label{eq:thtabnd}
\bs\Theta(T)\lesssim \|\bs\varepsilon_h^u(T)\|_{\Omega}.
\end{align}
Note that
\begin{align}\label{eq:elwv_trick_decom}
T_7
=\int_0^T(\rho\ddot{\bs\varepsilon}_h^u(t),\bs e_\Phi(t))_{\mc T_h}\mathrm dt - \int_0^T (\rho\ddot{\bs e}_u(t),\bs e_\Phi(t))_{\mc T_h}\mathrm dt.
\end{align}
For the second term of \eqref{eq:elwv_trick_decom}, we have
\begin{align*}
\left|\int_0^T (\rho \ddot{\bs e}_u(t),\bs e_\Phi(t))_{\mc T_h}\mathrm dt\right|
&=\left|(\rho\ddot{\bs e}_u(0),\bs e_{\underline{\Phi}}(0))_{\mc T_h}
+\int_0^T (\rho\dddot{\bs e}_u(t), \bs e_{\underline{\Phi}}(t))_{\mc T_h}\mathrm dt\right|\\
&\lesssim h^2\left(\|\ddot{\bs e}_u(0)\|_{\mc T_h}+T\vvvert\dddot{\bs e}_u\vvvert_{\mc T_h,\infty}^{[0,T]}\right)\bs\Theta(T).
\end{align*}
We next estimate the remaining term in \eqref{eq:elwv_trick_decom}. 
Since $\ddot{\bs\varepsilon}_h^u(t)\big|_K\in\mc P_{k+1}(K;\mbb R^3)$ for all $K$, we have
\begin{align*}
\int_0^T(\rho\ddot{\bs\varepsilon}_h^u(t),\bs e_\Phi(t))_{\mc T_h}\mathrm dt
=\int_0^T(\rho\ddot{\bs\varepsilon}_h^u(t),\mr P_{k+1}^\rho\bs e_\Phi(t))_{\mc T_h}\mathrm dt,
\end{align*}
where, $\mr P_{k+1}^\rho$ is the $\rho$-weighted $L^2$ projection onto $\mc P_{k+1}(K;\mbb R^3)$.
Testing the second error equation \eqref{eq:elwv_ereqb} with $\bs w=\mr P_{k+1}^\rho \bs e_\Phi(t)$ we have
\begin{align*}
(\rho\ddot{\bs\varepsilon}_h^u(t),\mr P_{k+1}^\rho\bs e_\Phi(t))_{\mc T_h} 
&= (\div\bs\varepsilon_h^\sigma(t),\mr P_{k+1}^\rho\bs e_\Phi(t))_{\mc T_h}
-\dualpr{\bs\tau\mathrm P_M({\bs\varepsilon_h^u}(t)
-\bs\varepsilon_h^{\widehat u}(t)),\mr P_{k+1}^\rho\bs e_\Phi(t)}_{\pp\mc T_h}\\
&\quad\,
+(\rho\ddot{\bs e}_u(t),\mr P_{k+1}^\rho\bs e_\Phi(t))_{\mc T_h}\nonumber+\dualpr{\bs\delta(t),\mr P_{k+1}^\rho\bs e_\Phi(t)}_{\pp\mc T_h}\\
&=: Q_1(t)+Q_2(t)+Q_3(t)+Q_4(t).
\end{align*}
Now we use integration by part and obtain
\begin{align*}
\int_0^T Q_1(t)\mathrm dt
&= (\div{\bs\varepsilon}_h^\sigma(0),\mr P_{k+1}^\rho\bs e_{\underline{\Phi}}(0))_{\mc T_h} 
+ \int_0^T (\div\dot{\bs\varepsilon}_h^\sigma(t),\mr P_{k+1}^\rho\bs e_{\underline{\Phi}}(t))_{\mc T_h}\mathrm dt,\\
\int_0^T Q_2(t)\mathrm dt
&=-\dualpr{\bs\tau\mathrm P_M({\bs\varepsilon_h^u}(0)
-\bs\varepsilon_h^{\widehat u}(0)),\mr P_{k+1}^\rho\bs e_{\underline{\Phi}}(0)}_{\pp\mc T_h}\\
&\quad\, -\int_0^T \dualpr{\bs\tau\mathrm P_M({\dot{\bs\varepsilon}_h^u}(t)
-\dot{\bs\varepsilon}_h^{\widehat u}(t)),\mr P_{k+1}^\rho\bs e_{\underline{\Phi}}(t)}_{\pp\mc T_h}\mathrm dt,\\
\int_0^T Q_3(t)\mathrm dt
&=(\rho\ddot{\bs e}_u(0),\mr P_{k+1}^\rho\bs e_{\underline{\Phi}}(0))_{\mc T_h}
+\int_0^T(\rho\dddot{\bs e}_u(t),\mr P_{k+1}^\rho\bs e_{\underline{\Phi}}(t))_{\mc T_h}\mathrm dt,\\
\int_0^T Q_4(t)\mathrm dt
&=\dualpr{\bs\delta(0),\mr P_{k+1}^\rho\bs e_{\underline{\Phi}}(0)}_{\pp\mc T_h}
+\int_0^T\dualpr{\dot{\bs\delta}(t),\mr P_{k+1}^\rho\bs e_{\underline{\Phi}}(t)}_{\pp\mc T_h}\mathrm dt.
\end{align*}
Note that
\begin{align*}
\mr P_{k+1}^\rho\bs e_{\underline{\Phi}}(t) = \mr P_{k+1}^\rho(\Pi\underline{\bs\Phi}(t)-\underline{\bs\Phi}(t))
= \Pi\underline{\bs\Phi}(t) - \underline{\bs\Phi}(t) - \mr P_{k+1}^\rho\underline{\bs\Phi}(t) + \underline{\bs\Phi}(t).
\end{align*}
Combining the above with the convergence properties about $\Pi\bs\Psi(t)$ and $\Pi\bs\Phi(t)$ (see Theorem \ref{th:PROJ}) we have
\begin{align*}
\|\mr P_{k+1}^\rho\bs e_{\underline{\Phi}}(t)\|_{K}
&\lesssim h_K^2 (|\underline{\bs\Phi}(t)|_{2,K} + |\underline{\bs\Psi}(t)|_{1,K}),\\
\|\bs\tau^{1/2}\mr P_{k+1}^\rho\bs e_{\underline{\Phi}}(t)\|_{\pp K}
&\lesssim h_K (|\underline{\bs\Phi}(t)|_{2,K} + |\underline{\bs\Psi}(t)|_{1,K}).
\end{align*}
Now back to the estimate of $\int_0^TQ_i(t)\mathrm dt$, we have
\begin{align*}
\left|\int_0^TQ_1(t)\mathrm dt\right|&\lesssim 
h(\|\bs\varepsilon_h^\sigma(0)\|_{\mc T_h}+T\vvvert\dot{\bs\varepsilon}_h^{\sigma}\vvvert_{\mc T_h,\infty}^{[0,T]})\bs\Theta(T),\\
\left|\int_0^T Q_2(t)\mathrm dt\right|&\lesssim
h(\|\mr P_M\bs\varepsilon_h^u(0)-\bs\varepsilon_h^{\widehat{u}}(0)\|_{\bs\tau}
+T\vvvert\mr P_M\dot{\bs\varepsilon}_h^u-\dot{\bs\varepsilon}_h^{\widehat{u}}\vvvert_{\bs\tau,\infty}^{[0,T]})
\bs\Theta(T),\\
\left|\int_0^T Q_3(t)\mathrm dt\right| &\lesssim h^2 (\|\ddot{\bs e}_u(0)\|_{\mc T_h}+T\vvvert\dddot{\bs e}_u\vvvert_{\mc T_h,\infty}^{[0,T]})
\bs\Theta(T),\\
\left|\int_0^T Q_4(t)\mathrm dt\right|
&\lesssim h (\|\bs\delta(0)\|_{\bs\tau^{-1}}+T\vvvert\dot{\bs\delta}\vvvert_{\bs\tau^{-1},\infty}^{[0,T]})
\bs\Theta(T),
\end{align*}
where we used the fact that $\|\div\bs\sigma\|_K\lesssim h_K^{-1}\|\bs\sigma\|_K$ for any $\bs\sigma\in\mc P_k(K;\Rsym)$. Finally, we use \eqref{eq:thtabnd} to bound $\bs\Theta(T)$ and the proof is completed.
\end{proof}

Now we are ready to prove Theorem \ref{th:elwv_spconv}.

\begin{proof}[Proof of Theorem \ref{th:elwv_spconv}]
Consider Proposition \ref{prop:wav_dual_id}. We will give estimates for the terms $T_i$ for $i=1\rightarrow 7$.

For $T_1$, by \eqref{eq:elwv_reg} we have
\begin{align*}
\left| T_1\right|=\left|(\rho\dot{\bs\Phi}(0),\Pi\bs u_0-\bs u_h(0))_{\mc T_h}\right|
\lesssim \|\bs\varepsilon_h^u(T)\|_{\mc T_h}(\|\Pi\bs u_0-\bs u_0\|_{\mc T_h}+\|\bs u_0 - \bs u_h(0)\|_{\mc T_h}).
\end{align*}
Note that $\bs u_h(0)$ is the solution of the HDG+ scheme \eqref{eq:elwv_HDGt0}. 
By Theorem \ref{th:st} in Section \ref{sec:st_main}, we have
\begin{align*}
\|\bs u_h(0) - \bs u_0\|_{\mc T_h}\lesssim 
h(\|\bs e_\sigma(0)\|_{\mc T_h}+\|\bs\delta(0)\|_{\bs\tau^{-1}}).
\end{align*}
Therefore
\begin{align*}
\left|(\rho\dot{\bs\Phi}(0),\Pi\bs u_0-\bs u_h(0))_{\mc T_h}\right|
\lesssim \|\bs\varepsilon_h^u(T)\|_{\mc T_h}(\|\bs e_{u}(0)\|_{\mc T_h}+h\|\bs e_\sigma(0)\|_{\mc T_h}
+h\|\bs\delta(0)\|_{\bs\tau^{-1}}).
\end{align*}

For $T_2$, we have
\begin{align*}
\left| T_2\right|&=\left|\int_0^T\dualpr{\bs\delta(t),\Pi\bs\Phi(t)-\bs\Phi(t)}_{\pp\mc T_h}\mathrm{d}t\right|\\
&=\left|\dualpr{\bs\delta(0),\Pi\underline{\bs\Phi}(0)
-\underline{\bs\Phi}(0)}_{\pp\mc T_h}
+\int_0^T
\dualpr{\dot{\bs\delta}(t),\Pi\underline{\bs\Phi}(t)
	-\underline{\bs\Phi}(t)}_{\pp\mc T_h}\mathrm{d}t\right|\\
&\le 
\|\bs\delta(0)\|_{\bs\tau^{-1}}\|\Pi\underline{\bs\Phi}(0)-\underline{\bs\Phi}(0)\|_{\bs\tau}
+\int_0^T\|\dot{\bs\delta}(t)\|_{\bs\tau^{-1}}\|\Pi\underline{\bs\Phi}(t)-\underline{\bs\Phi}(t)\|_{\bs\tau}\,\mathrm{d}t\\	
&\lesssim  h\|\bs\varepsilon_h^u(T)\|_{\mc T_h}
\left(
\|\bs\delta(0)\|_{\bs\tau^{-1}}
+T\vvvert\dot{\bs\delta}\vvvert_{\bs\tau^{-1},\infty}^{[0,T]}
\right),
\end{align*}
where we used the convergence properties about $\Pi\underline{\bs\Phi}$ (by Theorem \ref{th:PROJ}) and then equation \eqref{eq:elwv_mrreg} to bound $\bs\Theta(T)$ (see \eqref{eq:def_theta} for the definition).

Using similar ideas, 
for $T_3$  we have
\begin{align*}
\left| T_3\right|&=\left| \int_0^T
\dualpr{\bs\Delta(t),\bs\varepsilon_h^u(t)-\bs\varepsilon_h^{\widehat u}(t)}_{\pp\mc T_h}\mathrm{d}t\right|\\
&=\left|\dualpr{\underline{\bs\Delta}(0),\bs\varepsilon_h^u(0)
	-\bs\varepsilon_h^{\widehat u}(0)}_{\pp\mc T_h}
+\int_0^T\dualpr{\underline{\bs\Delta}(t),\dot{\bs\varepsilon}_h^u(t)
-\dot{\bs\varepsilon}_h^{\widehat u}(t)}_{\pp\mc T_h}\mathrm{d}t\right|\\
&\le 
\|\underline{\bs\Delta}(0)\|_{\bs\tau^{-1}}
\|\mathrm P_M(\bs\varepsilon_h^u(0)-\bs\varepsilon_h^{\widehat u}(0))\|_{\bs\tau}
+\int_0^T
\|\underline{\bs\Delta}(t)\|_{\bs\tau^{-1}}
\|\mathrm P_M(\dot{\bs\varepsilon}_h^u(t)-\dot{\bs\varepsilon}_h^{\widehat u}(t))\|_{\bs\tau}\mathrm{d}t
\\
&\lesssim
h\|\bs\varepsilon_h^u(T)\|_{\mc T_h}
\left(
\|\mathrm P_M(\bs\varepsilon_h^u(0)-\bs\varepsilon_h^{\widehat u}(0))\|_{\bs\tau}
+T\vvvert\mathrm P_M(\dot{\bs\varepsilon}_h^u-\dot{\bs\varepsilon}_h^{\widehat u})\vvvert_{\bs\tau,\infty}^{[0,T]}
\right),
\end{align*}
and for $T_4$, we have
\begin{align*}
\left| T_4\right|&=\left|\int_0^T(\mc A\bs e_\Psi(t),\bs\sigma(t)-\bs\sigma_h(t))_{\mc T_h}\mathrm{d}t\right|\\
&\le \left| (\mc A(\Pi\underline{\bs\Psi}(0)-\underline{\bs\Psi}(0))
,\bs\sigma(0)-\bs\sigma_h(0))_{\mc T_h}\right|\\
&\quad\,+\left|\int_0^T
(\mc A(\Pi\underline{\bs\Psi}(t)-\underline{\bs\Psi}(t))
,\dot{\bs\sigma}(t)-\dot{\bs\sigma}_h(t))_{\mc T_h}\mathrm{d}t\right|\\
&\lesssim
h\|\bs\varepsilon_h^u(T)\|_{\mc T_h}
\left(
\|\bs\sigma(0)-\bs\sigma_h(0)\|_{\mc T_h}
+T\vvvert\dot{\bs\sigma}-\dot{\bs\sigma}_h\vvvert_{\mc T_h,\infty}^{[0,T]}
\right).
\end{align*}

For $T_5$, by \eqref{eq:rmproj} and \eqref{eq:elwv_mrreg}, we have
\begin{align*}
\left|T_5\right|&=\left|\int_0^T(\bs e_\sigma(t),\bs\varepsilon(\bs\Phi)(t))_{\mc T_h}\mathrm{d}t\right|\\
&\le \left|(\Pi\bs\sigma(0)-\bs\sigma(0),\bs\varepsilon(\underline{\bs\Phi})(0)
-\mathrm P_0\bs\varepsilon(\underline{\bs\Phi})(0))_{\mc T_h}\right|\\
&\quad\, +\left|\int_0^T(\Pi\dot{\bs\sigma}(s)-\dot{\bs\sigma}(t),
\bs\varepsilon(\underline{\bs\Phi})(t)
-\mathrm P_0\bs\varepsilon(\underline{\bs\Phi})(t))_{\mc T_h}\mathrm{d}t\right|\\
&\lesssim
h\|\bs\varepsilon_h^u(T)\|_{\mc T_h}
\left(
\|\Pi\bs\sigma(0)-\bs\sigma(0)\|_{\mc T_h}
+T\vvvert\Pi\dot{\bs\sigma}-\dot{\bs\sigma}\vvvert_{\mc T_h,\infty}^{[0,T]}
\right).
\end{align*}

For $T_6$, we simply use \eqref{eq:elwv_reg} and obtain
\begin{align*}
\left| T_6\right|=\left|\int_0^T(\rho\bs\Phi,\ddot{\bs e}_u)_{\mc T_h}\right|
\lesssim \vvvert\bs\Phi\vvvert_{\Omega,\infty}^{[0,T]}\vvvert\ddot{\bs e}_u\vvvert_{\Omega,1}^{[0,T]}
\lesssim T\|\bs\varepsilon_h^u(T)\|_{\mc T_h}
\vvvert\ddot{\bs e}_u\vvvert_{\mc T_h,\infty}^{[0,T]}.
\end{align*}

Now we use Proposition \ref{prop:elwv_tricky} to estimate $T_7$, and finally obtain
\begin{align*}
\|\bs\varepsilon_h^u(T)\|_{\mc T_h}
\lesssim\ 
&(\|\bs e_{u}(0)\|_{\mc T_h}+h\|\bs e_\sigma(0)\|_{\mc T_h}
+h\|\bs\delta(0)\|_{\bs\tau^{-1}})\\
&+h\left(
\|\bs\delta(0)\|_{\bs\tau^{-1}}
+T\vvvert\dot{\bs\delta}\vvvert_{\bs\tau^{-1},\infty}^{[0,T]}
\right)\\
&+h\left(
\|\mathrm P_M(\bs\varepsilon_h^u(0)-\bs\varepsilon_h^{\widehat u}(0))\|_{\bs\tau}
+T\vvvert\mathrm P_M(\dot{\bs\varepsilon}_h^u-\dot{\bs\varepsilon}_h^{\widehat u})\vvvert_{\bs\tau,\infty}^{[0,T]}
\right)\\
&+h\left(
\|\bs\varepsilon_h^\sigma(0)\|_{\mc T_h}
+T\vvvert\dot{\bs\varepsilon}_h^\sigma\vvvert_{\mc T_h,\infty}^{[0,T]}
\right)\\
&+h\left(
\|\bs e_\sigma(0)\|_{\mc T_h}
+T\vvvert\dot{\bs e}_\sigma\vvvert_{\mc T_h,\infty}^{[0,T]}
\right)+T\vvvert\ddot{\bs e}_u\vvvert_{\mc T_h,\infty}^{[0,T]}\\
&+h^2
\bigg(
\|\ddot{\bs e}_u(0)\|_{\mc T_h}
+T\vvvert\dddot{\bs e}_u\vvvert_{\mc T_h,\infty}^{[0,T]}
\bigg).
\end{align*}
Combing the above estimates with Proposition \ref{prop:ene_id_t0} and Theorem \ref{th:elwv_th1}, the proof is completed.

\end{proof}

\section{Numerical experiments}
%\subsection{Convergence}
In this section, we present some numerical experiments to support our error estimates in Section \ref{sec:trans_elas}. Note that there are experiments for the steady-state and time-harmonic cases in \cite{QiShSh:2018} and \cite{HuPrSa:2017} respectively.
\begin{figure}[ht]
\centering
\includegraphics[scale=0.4]{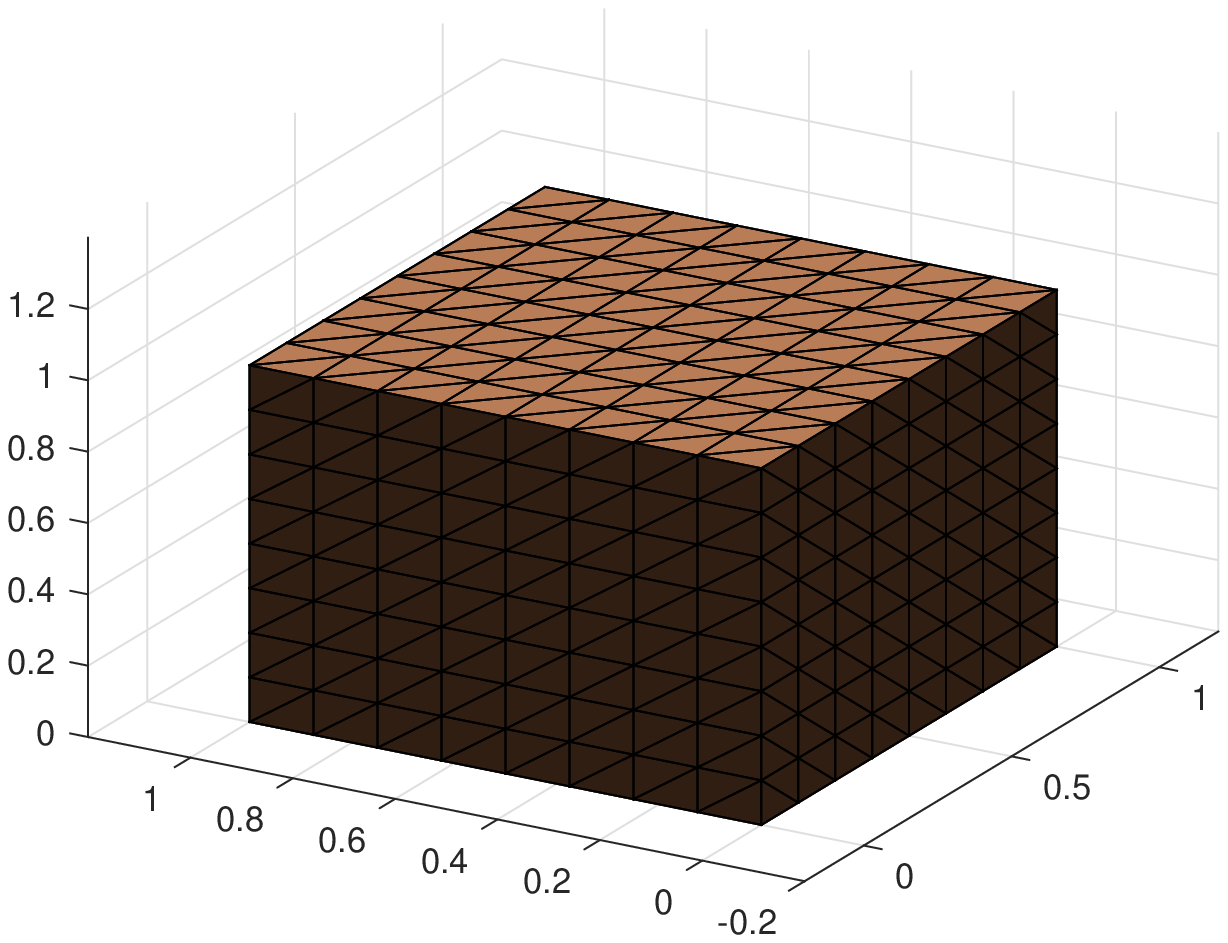}
\includegraphics[scale=0.4]{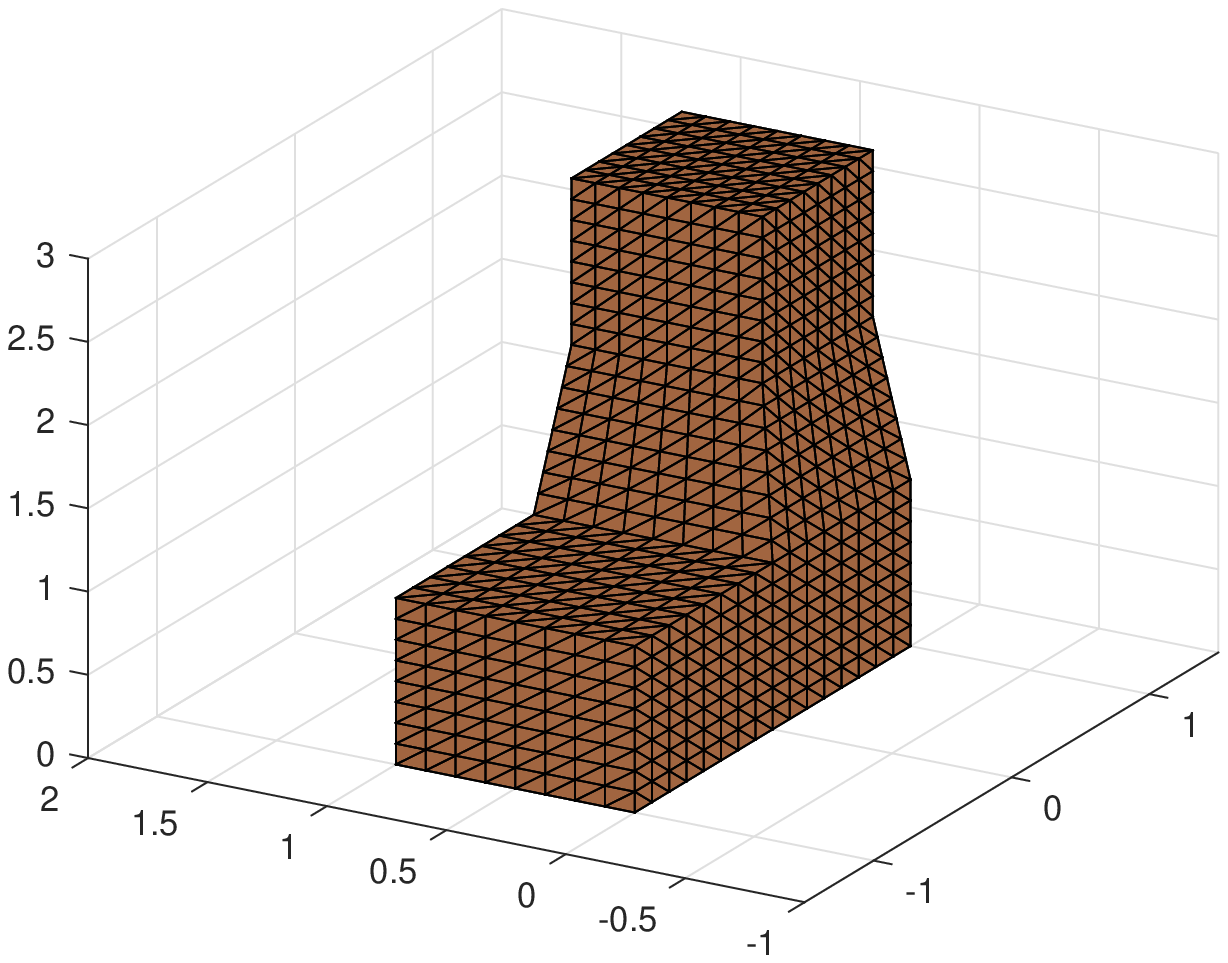}
\caption{Visualization of meshes used for error tests.}
\label{fig:meshvis}
\end{figure}

\noindent{\bf Convergence test.}
The experiments in this part are carried out on a cubic domain
$\Omega = (0,1)^3$ and a non-convex polyhedral domain (we will refer to it as the chimney; see Figure \ref{fig:meshvis}),
the time interval is $[0,T]$ with $T=\frac{3}{2}$, and we aim to estimate the relative $L^2$ errors
\[
E_\sigma:=\frac{\|\bs\sigma_h(T)-\bs\sigma(T)\|_\Omega}{\|\bs\sigma(T)\|_\Omega}, \qquad
E_u:=\frac{\|\bs u_h(T)-\bs u(T)\|_\Omega}{\|\bs u(T)\|_\Omega}.
\]
We consider a non-homogeneous isotropic material
\begin{align*}
\mc A\bs\sigma 
= \frac{1}{2\mu}\bs\sigma -\frac{\lambda}{2\mu(2\mu+3\lambda)} \mr{tr}\,\bs\sigma\,\mb I,
\end{align*}
with Lam\'{e} parameters
\begin{align*}
\lambda = \frac{2+x^2+y^2+z^2}{1+x^2+y^2+z^2},\qquad\mu = 3+\cos(xyz),
\end{align*}
and the mass density is constant $\rho\equiv1$. As exact solution we take $\bs u(\bs x,t):=\bs U(\bs x)H(t)$, where 
\begin{align*}
\bs U=(\cos(\pi x)\sin(\pi y)\cos(\pi z),\ 5x^2yz+4xy^2z+3xyz^2+17,\ \cos(2x)\cos(3y)\cos(z)),
\end{align*}
and the temporal part is $H(t)=t^3(1-t)^2$. The input data $\bs f$ and $\bs g$ are chosen so that \eqref{eq:elwv_pde} is satisfied.
Note that we have chosen the exact solution $\bs u$ that has vanishing initial conditions. This simplifies the calculations of $\bs u_h(0)$ and $\dot{\bs u}_h(0)$ (they are automatically $\bs 0$ by \eqref{eq:elwv_iniv} and \eqref{eq:elwv_HDGt0}). When $\bs u$ does not have vanishing initial conditions, the calculation of $\dot{\bs u}_h(0)$ involves projecting $\dot{\bs u}(0)$ to a space enriched by the $M$-decomposition spaces (the Cockburn-Fu filling), which we have not found an easy way to implement. Finding easier ways of calculating $\dot{\bs u}_h(0)$ for non-vanishing initial conditions will constitute our future works.

For the numerical schemes, we use the HDG+ method \eqref{eq:elwv_HDG} for space discretization and Trapezoidal Rule Convolution Quadrature (TRCQ) (see \cite{Ba:2010,Lu:1988}) for time integration. This is equivalent to using Trapezoidal Rule time-stepping in the semidiscrete system. The time interval $[0,T]$ is equally divided and each timestep is of length $\kappa$. Since the error from the TRCQ is $\mc O(\kappa^2)$, we choose
$\kappa \approx h^{(k+2)/2}$ so that the error from the time discretization does not pollute the order of convergence of the space discretization. 

From Figure \ref{fig:ord_conv} or Table \ref{tab:ord_conv}, 
we observe that the orders of convergence for $\bs\sigma_h(T)$ and $\bs u_h(T)$ are $\mc O(h^{k+1})$ and $\mc O(h^{k+2})$ respectively, agreeing the estimates in 
Theorem \ref{th:elwv_th1} and Theorem \ref{th:elwv_spconv}.

\begin{table}[ht]
\centering
\begin{tabular}{|c|c|c|c|c|c|c|c|c|c|c|c|}
\hline
&&\multicolumn{4}{c|}{Cube} & 
 \multicolumn{4}{c|}{Chimney}\\
\hline
 & & \multicolumn{2}{c|}{$E_\sigma$} & 
 \multicolumn{2}{c|}{$E_u$}&\multicolumn{2}{c|}{$E_\sigma$} & 
 \multicolumn{2}{c|}{$E_u$}\\
\hline
$k$ & $h$ & Error & Order & Error & Order
& Error & Order & Error & Order\\
\hline
	& 1.6329 & 2.82E-1 & -    & 4.32E-2 & - & 1.27E-1 & - & 3.93E-2 & -\\	
1	& 0.8164 & 7.06E-2 & 2.00 & 6.16E-3 & 2.81 & 3.12E-2 & 2.02 & 5.36E-3 & 2.87\\
	& 0.4082 & 1.93E-2 & 1.87 & 5.44E-4 & 3.50 & 8.16E-3 & 1.93 & 4.92E-4 & 3.44\\
	& 0.2041 & 4.79E-3 & 2.01 & 5.65E-5 & 3.27 & 1.99E-3 & 2.03 & 5.09E-5 & 3.28\\
\hline
	& 1.6329 & 1.30E-1 & -    & 1.84E-2 & - & 6.44E-2 & - & 1.68E-2 &-\\	
2	& 0.8164 & 1.72E-2 & 2.92 & 1.52E-3 & 3.59 & 6.83E-3 & 3.24 & 1.23E-3 & 3.78\\
	& 0.4082 & 2.31E-3 & 2.89 & 5.92E-5 & 4.68 & 8.96E-4 & 2.93 & 4.74E-5 & 4.70\\
	& 0.2041 & 2.86E-4 & 3.02 & 2.52E-6 & 4.55 & - & - & -& -\\
\hline
	& 1.6329 & 4.36E-2 & -    & 9.50E-3 & - & 2.18E-2 & - & 7.74E-3 & -\\	
3	& 0.8164 & 3.90E-3 & 3.48 & 3.35E-4 & 4.83 & 1.47E-3 & 3.89 & 2.59E-4 & 4.90\\
	& 0.4082 & 2.56E-4 & 3.93 & 6.34E-6 & 5.72 & 9.54E-5 & 3.94 & 4.88E-6 & 5.73\\
\hline
\end{tabular}
\caption{History of convergence for $\bs\sigma_h(T)$ and $\bs u_h(T)$ with sequence of uniform refinements in space and over-refinements in time. }
\label{tab:ord_conv}
\end{table}

\begin{figure}[ht]
\centering
\includegraphics[scale=0.75]{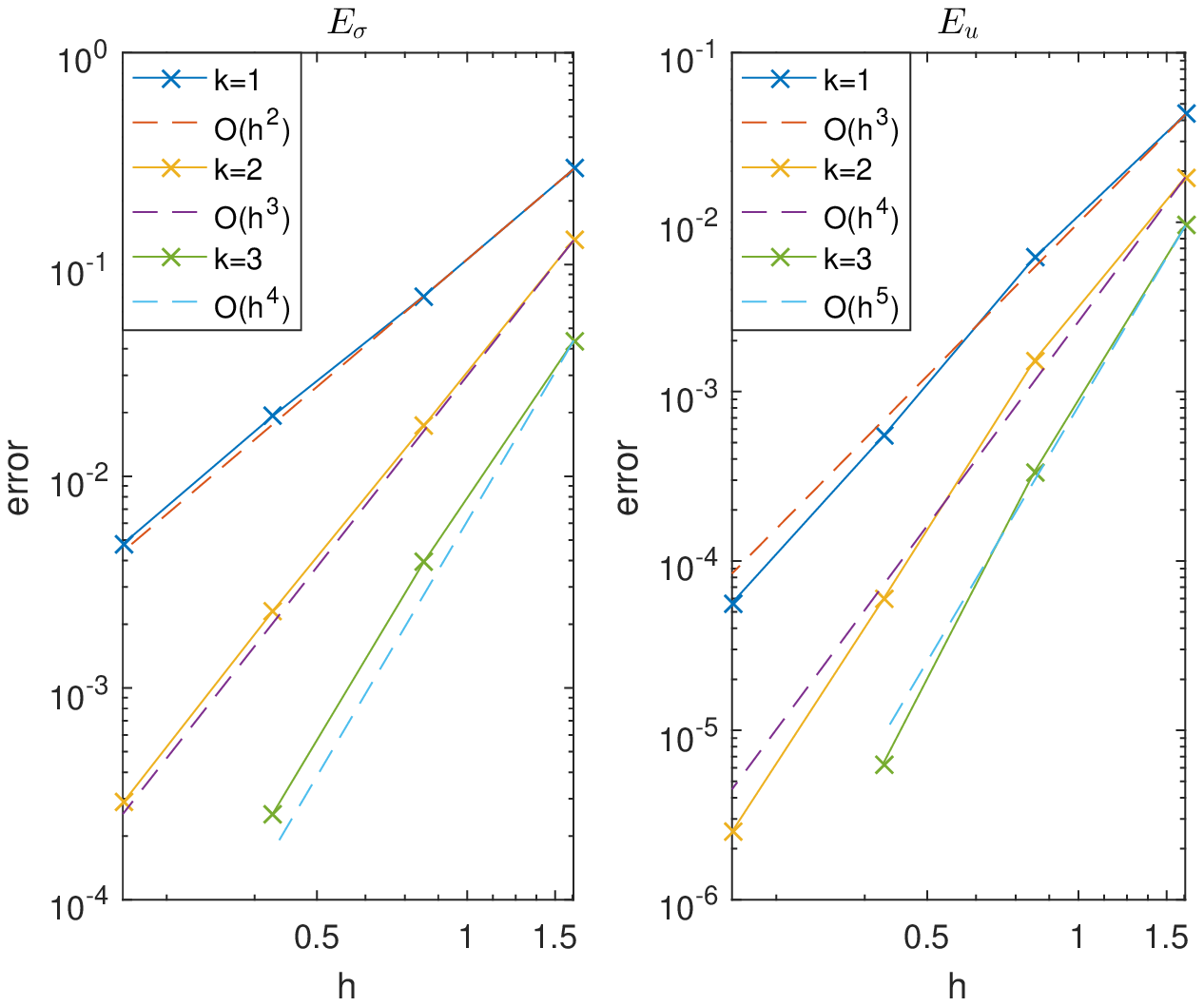}
\includegraphics[scale=0.75]{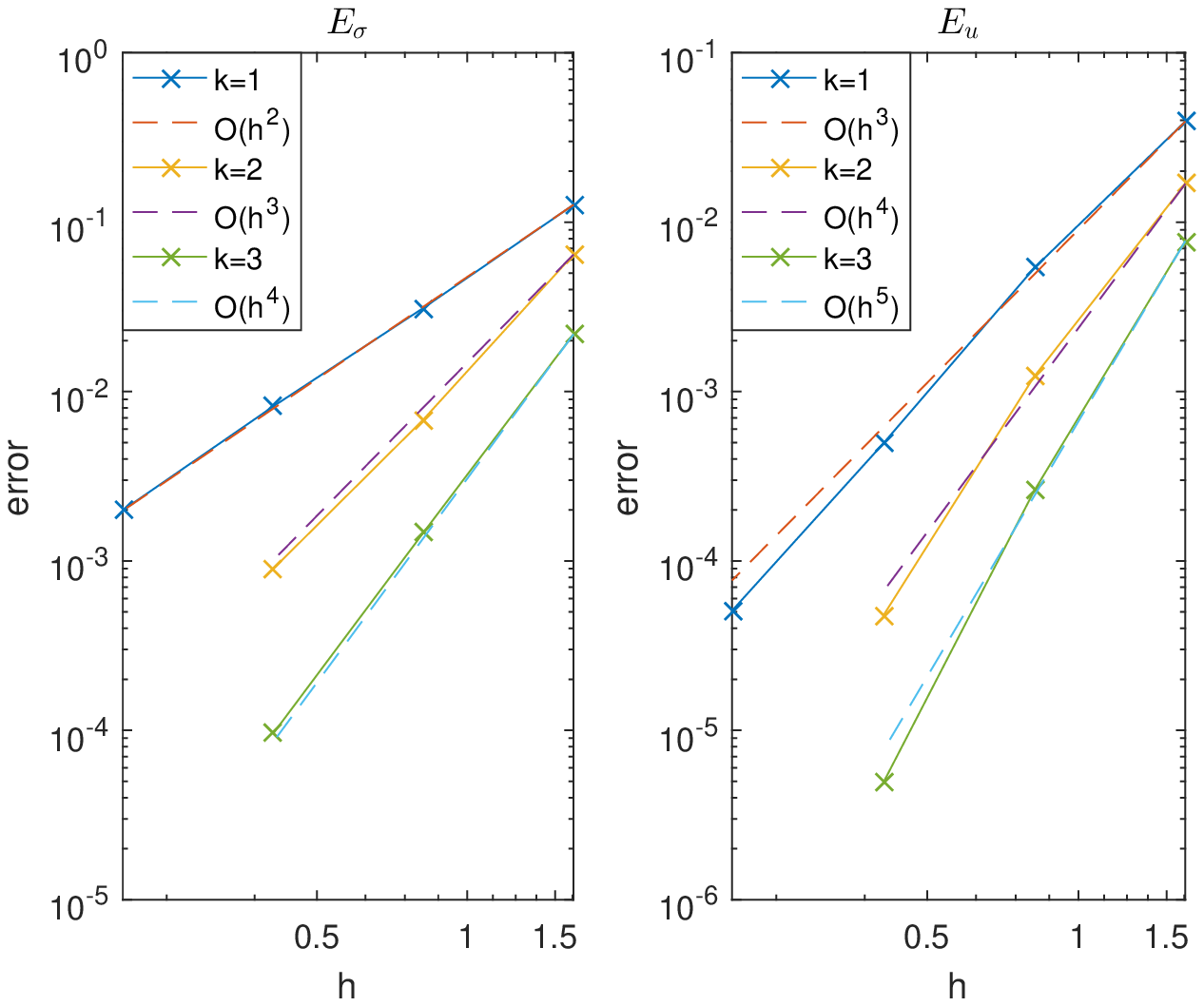}
\caption{History of convergence for $\bs\sigma_h(T)$ and $\bs u_h(T)$ with sequence of uniform refinements in space and over-refinements in time. Top two figures for cubic meshes and bottom two for chimney.}
\label{fig:ord_conv}
\end{figure}

\noindent{\bf Locking test.} Note that the HDG+ method was shown to be free from volumetric locking for steady state system in \cite{QiShSh:2018}. We here conduct some locking experiments for elastic waves. Most of the experiment settings will be the same as the convergence test. Let us just mention the differences. We shall conduct two experiments (denoted by A and B) on the cubic domain where the Lam\'{e} parameters are chosen as $(\lambda,\mu)=(1.5\times10^2,3)$ for test A and $(\lambda,\mu)=(1.5\times10^4,3)$ for test B.
Their corresponding Poisson's ratios can be easily calculated: $\nu\approx0.49$ for test A and $\nu\approx0.4999$ for test B. For the exact solutions, the temporal part $H(t)$ is unchanged while the spacial part is changed to
\begin{align*}
\bs U=\left(-x^2(x-1)^2y(y-1)(2y-1)z(1-z),\ y^2(y-1)^2x(x-1)(2x-1)z(1-z),\ 0\right).
\end{align*}
This choice of the exact solution is a simple 3D adaptation of those used in \cite{BeLi:1979,QiShSh:2018} for locking experiments in 2D. We collect the history of convergence for $\bs\sigma_h(T)$ and $\bs u_h(T)$ in Table \ref{tab:ord_conv_locking}.

\begin{table}[ht]
\centering
\begin{tabular}{|c|c|c|c|c|c|c|c|c|c|c|c|}
\hline
&&\multicolumn{4}{c|}{$\nu\approx0.49$} & 
 \multicolumn{4}{c|}{$\nu\approx0.4999$}\\
\hline
 & & \multicolumn{2}{c|}{$E_\sigma$} & 
 \multicolumn{2}{c|}{$E_u$}&\multicolumn{2}{c|}{$E_\sigma$} & 
 \multicolumn{2}{c|}{$E_u$}\\
\hline
$k$ & $h$ & Error & Order & Error & Order
& Error & Order & Error & Order\\
\hline
  &  1.63e+00  &  8.21e-01  &  -   &    2.29e+00  &  - & 
8.25e-01  &  -   &    2.29e+00  &  -\\ 
1 &  8.16e-01  &  4.66e-01  &  0.82   &    4.47e-01  &  2.36 & 
4.69e-01  &  0.82   &    4.46e-01  &  2.36\\ 
  &  4.08e-01  &  1.78e-01  &  1.39   &    5.78e-02  &  2.95 & 
1.79e-01  &  1.39   &    5.75e-02  &  2.95\\ 
  &  2.04e-01  &  4.52e-02  &  1.98   &    7.43e-03  &  2.96 & 
4.54e-02  &  1.98   &    7.41e-03  &  2.96\\ 
\hline
  &  1.63e+00  &  5.19e-01  &  -   &    1.36e+00  &  - &
5.20e-01  &  -   &    1.36e+00  &  -\\  
2 &  8.16e-01  &  1.98e-01  &  1.39   &    1.24e-01  &  3.45 &
2.00e-01  &  1.38   &    1.24e-01  &  3.45\\ 
  &  4.08e-01  &  3.70e-02  &  2.42   &    7.61e-03  &  4.03 &
3.73e-02  &  2.42   &    7.58e-03  &  4.03\\ 
  &  2.04e-01  &  4.79e-03  &  2.95   &    4.33e-04  &  4.14 &
4.82e-03  &  2.95   &    4.31e-04  &  4.14\\ 
\hline
\end{tabular}
\caption{History of convergence for $\bs\sigma_h(T)$ and $\bs u_h(T)$ with sequence of uniform refinements in space and over-refinements in time. All experiments are conducted on the unit cube.}
\label{tab:ord_conv_locking}
\end{table}

From Table \ref{tab:ord_conv_locking}, we observe no degeneration of the convergence rates for $\bs\sigma_h(T)$ and $\bs u_h(T)$ as the Poisson's ratio $\nu$ approaches the incompressible limit $0.5$. This supports that the HDG+ method is volumetric locking free for elastic waves.

\section{Extensions and conclusion}
For the sake of conciseness, we have limited the discussion to the setting of elastic problems on simplicial meshes. However, the tools we introduce here can be extended to construct HDG projections in a much wider setting. We next discuss three possible extensions.

\begin{enumerate}
\renewcommand{\theenumi}{\Alph{enumi}}
\item {\bf Elasticity on polyhedral meshes.} The HDG+ projection for elasticity can be extended to a projection on polyhedral elements. One way to achieve this is to construct the projection directly on the physical element, instead of first constructing the projection on the reference element and then using a push-forward operator (this is what we did in this paper). This alternative approach is feasible since the $M$-decomposition can be applied on general polyhedral elements (see \cite{CoFuSa:2017,CoFu:2018}). 

\item {\bf HDG+ for elliptic diffusion.}
The HDG+ projection can be constructed for steady-state diffusion. We have explored this in \cite{DuSa:2019} for simplicial meshes. For general polyhedral meshes, the projection can be obtained by following a similar procedure as demonstrated in Figure \ref{fig:flow_chart}. It can be summarized in three steps: (1) Enrich the approximation space for the flux so that the $M$-decomposition is achieved; (2) Define an extended projection by enforcing the weak-commutativity property on the homogeneous polynomial space of order $k+1$ (similar to \eqref{eq:3.1e}); (3) Define a composite projection and collect the remainder term on the boundary of the element. 

\item {\bf Standard HDG for elasticity.} We can also construct a projection for the standard HDG method for elasticity (where polynomial spaces of order $k$ are used for both the stress and the displacement).
This is achieved by defining the composite projection and the boundary reminder before constructing the extended projection. To be more specific, suppose
\begin{align*}
\bs\Pi^M: H^1(K;\mbb R_\mr{sym}^{3\times3})\times H^1(K;\mbb R^3)&\rightarrow \mc P_k(K;\mbb R_\mr{sym}^{3\times3})\oplus\Sigma_\mr{fill}(K)\times \mc P_{k}(K;\mbb R^3)\\
(\bs\sigma,\bs u)&\mapsto (\bs\Pi_\sigma^M(\bs\sigma,\bs u),\bs\Pi_u^M(\bs\sigma,\bs u))
\end{align*}
is the $M$-decomposition associated projection. We then define
\begin{align*}
\bs\Pi(\bs\sigma,\bs u)&:=(\mr P_k\bs\Pi_\sigma^M(\bs\sigma,\bs u),\bs\Pi_u^M(\bs\sigma,\bs u)),\\
\bs\delta&:=\bs\Pi_\sigma^M(\bs\sigma,\bs u)\cdot\mb n-\mr P_k\bs\Pi_\sigma^M(\bs\sigma,\bs u)\cdot\mb n.
\end{align*}
This completes the definition of the projection (and the associate boundary remainder) for the standard HDG method for elasticity. The rest of the  error analysis follows the exact same procedure we have discussed in this paper. For instance, for the steady-state problem, we obtain the same energy estimate \eqref{eq:st_est}, namely,
\begin{align*}
\|\bs\varepsilon_h^\sigma\|_{\mc A}^2+\|\mathrm P_M(\bs\varepsilon_h^u-\bs\varepsilon_h^{\widehat u})\|_{\bs\tau}^2
\le 
\|\bs e_\sigma\|_{\mc A}^2+\|\bs\delta\|_{\bs\tau^{-1}}^2.
\end{align*}
In this case, the term $\|\bs\delta\|_{\bs\tau^{-1}}$ has an $\mc O(h^{k+1/2})$ convergence rate because
 $\bs\tau=\mc O(1)$. We thus recover the existing suboptimal estimates  obtained in \cite{FuCoSt:2015} in a unified way by using the same arguments. The only difference here is a simple change of the projection. 

\end{enumerate}

To conclude, we have proposed some new mathematical tools for the error analysis of HDG methods. 
The two most important ones are: (1) the extended projection constructed by enforcing the weak commutativity on a higher order polynomial space (see \eqref{eq:3.1e}); (2) the boundary remainder reflecting the discrepancy between the normal traces of the $M$-decomposition associate projection and a composite projection (see \eqref{eq:3.10a}). These tools allow us to flexibly devise projections for more variants of HDG methods. We have demonstrated this by constructing the projection for the Lehrenfeld-Sch\"{o}berl HDG (HDG+) method for elasticity. By using the projection, we are able to recover the existing error estimates in a more concise analysis for the steady-state and the time-harmonic elastic problems. For elastic waves,  
we have successfully used the projection to devise a semi-discrete HDG+ scheme (the initial velocity of the semi-discrete scheme is defined by using the HDG+ projection) and prove its uniformly-in-time optimal convergence. Improving the generality of the tools will constitute the future works. 

\bibliographystyle{plain}
\bibliography{references}

\end{document}